\documentclass[11pt]{article}
\usepackage[centertags]{amsmath}
\usepackage{amsfonts}
\usepackage{amssymb}
\usepackage{makeidx}
\usepackage{newlfont}
\usepackage{amsthm} 
\usepackage{stmaryrd}
\usepackage[]{titletoc}  
\dottedcontents{section}[1.5em]{}{2.3em}{1pc}   
\dottedcontents{subsection}[3.8em]{}{3.2em}{1pc}
\usepackage{mathrsfs}
\usepackage{stmaryrd}
 \usepackage{graphicx}  
 \usepackage{xypic}
\usepackage{hyperref}
\usepackage{geometry}
\usepackage[numbers,sort&compress]{natbib}
\newlength{\defbaselineskip}
\newcommand{\setlinespacing}[1]%
           {\setlength{\baselineskip}{#1 \defbaselineskip}}

\theoremstyle{plain}
\newtheorem{thm}{Theorem}[section]
\newtheorem{cor}[thm]{Corollary}
\newtheorem{lem}[thm]{Lemma}
\newtheorem{prop}[thm]{Proposition}
\newtheorem{exam}[thm]{Example}
\newtheorem{rem}[thm]{Remark}

\newtheorem{thmA}{Theorem}
\newtheorem{lemA}[thmA]{Lemma}

\makeatletter\@addtoreset{equation}{section} \makeatother
\begin{document}

\title {The
Kozlov completeness problem
}
\author{
Hui Dan \quad Kunyu Guo
}
\date{}
 \maketitle

\maketitle \noindent\textbf{Abstract:} This paper concerns a long-standing problem raised by
Kozlov on completeness of  the dilation
systems $\{\mathbf{1}_{(\alpha,\beta)}(kx):k=1,2,\cdots\}$  generated by  odd periodic extensions on  $\mathbb{R}$ of 
 characteristic functions $\mathbf{1}_{(\alpha,\beta)}$, where
 $0\leq\alpha<\beta\leq1$.
Up
to now there has only some fragmentary  results under the assumption $\alpha=0$. Focusing  on the dilation completeness problem for  characteristic functions $\mathbf{1}_V$ of   open subsets $V\subset(0,1)$ that are finite  unions of   intervals with rational endpoints,  we exhibit the exact forms of such $V$ in almost all  interesting situations by using substantially
 techniques from analytic number theory. As a consequence, it yields a  complete solution for  the rational version of the
Kozlov completeness problem. Moreover, our results also illustrate the fascinating connection among the Completeness Problem,  the Twin Prime Conjecture and the  Sophie Germain Prime Conjecture.




\tableofcontents

\section{Introduction}

The  space $L^2=L^2[0,1]$ has a canonical orthonormal basis $\{\sqrt{2}\sin (n\pi x): n=1,2,\cdots\}$, and hence each  $\varphi\in L^2$  has its
 Fourier-sine expansion $$\varphi(x)=\sum_{n=1}^{\infty}a_n\sqrt{2}\sin (n\pi x).$$
In this way we  always think  later  that  each function $\varphi$ in $L^2$ is an  odd $2$-periodic function
on the real line $\mathbb{R}$,  still denoted by $\varphi$.
In 1940s, Beurling and Wintner  considered the following fundamental  problem in
classical analysis \cite{Beu,Win}.

\vskip2mm
\noindent\textbf{Beurling-Wintner Problem:} \emph{for which $\varphi\in L^2$, the linear span of the following periodic  dilation system (\text{p.d.s.})
\begin{equation}\label{psd}
 \{\varphi(kx):k=1,2,\cdots\}.
\end{equation}
is dense on $L^2$?  }
\vskip2mm

The B-W problem receives considerable attentions. We refer the readers to some relevant literatures \cite{Bal, BM, Bou1, Bou2, Har, Sz, HW, Koz1, Koz2, Koz3, Ah, NGN, HLS, Mit, Ni1, Ni2, Ni3, Ni5, Ni6, No, DG, MNS}.
In 1948-1950,  Kozlov \cite{Koz1, Koz2, Koz3} initiated the study for the  dilation completeness problem of  characteristic functions of subintervals of $(0,1)$, which served as an important gateway to the B-W problem. The following problem was raised by Kozlov in the paper \cite{Koz3} in 1950 (also see \cite{Ni3, Ni5, DG, Ni4}).
\vskip2mm
\noindent\textbf{Kozlov Completeness Problem:} \emph{for which subinterval $(\alpha,\beta)$ of $(0,1)$ is the \text{p.d.s.} (\ref{psd}) of the characteristic function $\mathbf{1}_{(\alpha,\beta)}$  complete?}
\vskip2mm

 Up to now only some fragmentary  results under the assumption $\alpha=0$ were known on the Kozlov completeness problem.
 Let $\mathcal{C}$ be the set of functions generating complete periodic dilation systems. The conclusion $\mathbf{1}_{(0,1)}\in\mathcal{C}$  was obtained in \cite{Ah} with a long proof.
 Nikolski proved that
$\mathbf{1}_{(0,\beta)}\in\mathcal{C}$  for $\beta=\frac{1}{2},\frac{2}{3}$ and $\mathbf{1}_{(0,\beta)}\notin\mathcal{C}$  for $\beta$ in a neighborhood of $\frac{1}{3}$ or $\beta=\frac{q}{p}$, where $p$ is an odd prime and $q$ is odd with
 $\sin^2\frac{q\pi}{2p}<\frac{1}{p+1}$, which were claimed in \cite{Koz3} (also see \cite{Ni3,Ni5,Ni6,DG,Ni4}). The proofs were exhibited in Nikolski's talk in 2018 \cite{Ni4}. Moreover, he also proved that the \text{p.d.s.}  of
$\mathbf{1}_{(0,\frac{1}{4})}$ is incomplete.

  It is conceivable that any solid progress on the
completeness problem will be highly non-trivial in light of both the problem
itself and its connection to widely open unsolved problems in other areas.
In fact, the completeness problem is  shown to be equivalent to cyclicity of
certain functions in the Hardy spaces of Dirichlet series as well as functions
over infinite polydisk (see \cite{HLS,Ni1}, and yet
determining cyclic vectors in multi-variate Hardy space is still a long-standing
open problem in operator theory.
Some evidences suggest  that there exists a  powerful connection
between the completeness problem and the Riemann Hypothesis \cite{Bal, Ny,BD, No}.

We  refer to \cite{HLS,McCa} for a complete solution when the \text{p.d.s.}  of an $L^2$-function  $\varphi$ forms a Riesz basis of $L^2$, which is a special complete system.
Also see some  closely related works \cite{Da1, Ro1, Ro2, CH1, CH2, HS,  BBB1, BBB2, DGMS, QQ}.

In this paper, we make it explicit by attacking the completeness problem
with a substantial involvement of techniques from analytic number theory.
We focus on characteristic function $\mathbf{1}_V$ of  an open subset $V$ of $(0,1)$
where $V$ is the union of finitely many intervals with rational endpoints. In particular,
we completely settle the rational version of the Kozlov completeness problem:

\begin{thm}\label{interval}
   Let $\alpha,\beta$ be two rational numbers with $0\leq \alpha<\beta\leq1$ and put $I=(\alpha,\beta)$. Then $\mathbf{1}_I\in\mathcal{C}$  if and only if $I$ is one of the following $10$ intervals:
$$(0,1), (0,\frac{1}{2}), (\frac{1}{2},1), (0,\frac{2}{3}), (\frac{1}{3},\frac{2}{3}), (\frac{1}{3},1),(\frac{1}{4},\frac{3}{4}),(\frac{1}{5},\frac{3}{5}),
(\frac{2}{5},\frac{4}{5}),(\frac{1}{6},\frac{5}{6}).$$
 \end{thm}

%
\vskip2mm
When $V$ consists of more than one intervals, the situation becomes more interesting and more  challenging.  By  developing new
machineries to deal with zeros of Dirichlet series,  the \textit{generalized Kozlov  completeness problem} can be fully solved in a  completely general case.

To begin with, let $0\leq\frac{s_1}{t_1}<\frac{s_1'}{t_1'}<\cdots<\frac{s_N}{t_N}<\frac{s_N'}{t_N'}\leq1 \ (N\in\mathbb{N})$ be $2N$ irreducible fractions. Set \begin{equation}\label{Vdisplay}
V=(\frac{s_1}{t_1},\frac{s_1'}{t_1'})
\cup\cdots\cup(\frac{s_N}{t_N},\frac{s_N'}{t_N'}),
\end{equation}
 and \begin{equation}\label{tdisplay}
 t_V=\mathrm{lcm}(t_1, t_1', \cdots, t_N, t_N').
 \end{equation}
  Here, the notation ``lcm" in (\ref{tdisplay}) stands for ``the least common multiple".
   Then $V$ is an open subset of $(0,1)$ with $N$ component intervals, and can be rewritten as
 \begin{equation}\label{Vdisplay'}
   V=(\frac{\alpha_1}{t_V},\frac{\beta_1}{t_V})\cup\cdots\cup(\frac{\alpha_N}{t_V},\frac{\beta_N}{t_V}),
 \end{equation} 
 with $\gcd(\alpha_1,\beta_1,\cdots,\alpha_N,\beta_N,t_V)=1.$

 In this situation the generalized Kozlov  completeness problem is to  give out  the explicit form of $V$ such that  $\mathbf{1}_V\in\mathcal{C}$.
 We exactly determine the case  $t_V\leq6$ (see Theorem \ref{1to6}).  On the other hand, for the case $t_V\geq7$ we also prove, under fairly general assumptions, that $V$  are  of  certain particular  forms. For this, set $$V_{t,0}=\begin{cases}
                        (0,\frac{1}{t})\cup(\frac{2}{t},\frac{3}{t})\cup\cdots\cup(\frac{t-2}{t},\frac{t-1}{t}), & t=2,4,6,\cdots; \\
                        (0,\frac{1}{t})\cup(\frac{2}{t},\frac{3}{t})\cup\cdots\cup(\frac{t-1}{t},1), &  t=3,5,7,\cdots,
                      \end{cases}
       $$ and $$V_{t,1}=\begin{cases}
                        (\frac{1}{t},\frac{2}{t})\cup(\frac{3}{t},\frac{4}{t})\cup\cdots\cup(\frac{t-1}{t},1), & t=2,4,6,\cdots; \\
                        (\frac{1}{t},\frac{2}{t})\cup(\frac{3}{t},\frac{4}{t})\cup\cdots\cup(\frac{t-2}{t},\frac{t-1}{t}), &  t=3,5,7,\cdots.
                      \end{cases}
       $$
We will show that $\mathbf{1}_{V_{t,0}},\mathbf{1}_{V_{t,1}}\in\mathcal{C}$ always holds for every $t\geq2$ (see Example \ref{Vt0Vt1}).
But, more surprisingly, the converse remains true under the hypothesis in Theorem \ref{Vtthm} below.
\begin{thm}\label{Vtthm}
  Let $V$  be given as in (\ref{Vdisplay'}). Suppose that $t_V\geq7$, $3\nmid t_V$, and there exists
 a boundary point $\frac{s}{t_V}$ of $V$ with   $\gcd(s,t_V)=1$. Then
   $\mathbf{1}_V\in\mathcal{C}$ if and only if
  $V=V_{t_V,0}$ or $V_{t_V,1}$.
\end{thm}

 As a consequence, the case when  $t_V$ is a power of some prime is  completely solved.

 \begin{thm}\label{pkthm}
   Let $V$  be given as in (\ref{Vdisplay'}). Suppose that $t_V\geq7$ and $t_V=p^k$ for some  prime $p$ and some $k\in\mathbb{N}$. Then
   $\mathbf{1}_V\in\mathcal{C}$ if and only if
  $V=V_{p^k,0}$ or $V_{p^k,1}$.
 \end{thm}

Note that Theorem \ref{pkthm} fails for $t_V=3, 4, 5$ (see Theorem \ref{1to6}).

\vskip2mm

The following two examples illustrate that neither of  two assumptions (``$3\nmid t_V$" and `` a boundary point $\frac{s}{t_V}$ of $V$ with   $\gcd(s,t_V)=1$")
in Theorem \ref{Vtthm} can be dropped (see Examples \ref{Vtexam1} and \ref{Vtexam2} for details):
\begin{itemize}
  \item [(1)] For $V=(\frac{1}{15},\frac{7}{15})\cup(\frac{11}{15},\frac{13}{15})$,  $W=(\frac{2}{15},\frac{4}{15})\cup(\frac{8}{15},\frac{14}{15})$ we have $\mathbf{1}_V,\mathbf{1}_W\in\mathcal{C}$.
  \item [(2)] Suppose that $p_1, p_2$ are twin primes ($p_1=p_2+2$), and set $$V=(\frac{1}{p_1},\frac{1}{p_2})\cup(\frac{3}{p_1},\frac{3}{p_2})\cup\cdots
     \cup(\frac{2m-1}{p_1},\frac{2m-1}{p_2})\cup\cdots\cup(\frac{p_1-2}{p_1},1)$$
     and $$W=(0,\frac{2}{p_1})\cup(\frac{2}{p_2},\frac{4}{p_1})\cup\cdots\cup(\frac{2m-2}{p_2},\frac{2m}{p_1})
     \cup\cdots\cup(\frac{p_2-1}{p_2},\frac{p_1-1}{p_1}).$$
     Then $\mathbf{1}_V,\mathbf{1}_W\in\mathcal{C}$.
\end{itemize}
Theorem  \ref{Vtthm} indicates that under mild conditions,  $V$ with $\mathbf{1}_V\in\mathcal{C}$ has form $V_{t,0}$ or $V_{t,1}$. One may  want to know exactly what  ``exceptions" are.
Example (2) above  provides us with some information.
 An open subset $V$ of $(0,1)$ is said to be non-degenerated
if $(0,1)\setminus V$ has no isolated point.

\begin{thm}\label{titjthm}
  Let $V$ be a proper and non-degenerated open subset of $(0,1)$, and $\{\frac{s_i}{t_i}\}_{i=1}^{M}$ be all  boundary points of $V$  in $(0,1)$ with  $\mathrm{gcd}(s_i,t_i)=1$ for all $i$.
  Assume that not all denominators $t_i$ are the same, and
  \begin{itemize}
    \item [(1)] for each $1\leq i\leq M$, $t_i$ is odd,  and $t_i$ is not a product of two primes;
    \item [(2)] for each pair $1\leq i,j\leq M$, either $t_i=t_j$ or $\gcd(t_i,t_j)=1$.
  \end{itemize}
Then
   $\mathbf{1}_V\in\mathcal{C}$ if and only if there exist two primes $p_1,p_2$ such that $V$ has one  of  the following forms:
    \begin{itemize}
   \item [(i)] $p_1=p_2+2$ and
     $$V=(\frac{1}{p_1},\frac{1}{p_2})\cup(\frac{3}{p_1},\frac{3}{p_2})\cup\cdots
     \cup(\frac{2m-1}{p_1},\frac{2m-1}{p_2})\cup\cdots\cup(\frac{p_1-2}{p_1},1);$$
     \item [(ii)]   $p_1=p_2+2$ and
  $$V=(0,\frac{2}{p_1})\cup(\frac{2}{p_2},\frac{4}{p_1})\cup\cdots\cup(\frac{2m-2}{p_2},\frac{2m}{p_1})
     \cup\cdots\cup(\frac{p_2-1}{p_2},\frac{p_1-1}{p_1});$$
    \item [(iii)]   $p_1=2p_2+1$ and
     $$V=(0,\frac{2}{p_1})\cup(\frac{1}{p_2},\frac{4}{p_1})\cup\cdots\cup(\frac{m-1}{p_2},\frac{2m}{p_1})
     \cup\cdots\cup(\frac{p_2-1}{p_2},\frac{p_1-1}{p_1});$$
     \item [(iv)] $p_1=2p_2-1$ and
     $$V=(0,\frac{1}{p_2})\cup(\frac{2}{p_1},\frac{2}{p_2})\cup\cdots\cup(\frac{2m-2}{p_1},\frac{m}{p_2})
     \cup\cdots\cup(\frac{p_1-1}{p_1},1).$$
      \end{itemize}
\end{thm}

            A prime $p$ is called a Sophie Germain prime if $2p+1$ or $2p-1$ also is a prime. It remains a famous unanswered
question whether there are infinitely many Sophie Germain primes \cite{Du}.   Theorem \ref{titjthm} illustrates the fascinating connection among the Completeness Problem,  the Twin Prime Conjecture and the  Sophie Germain Prime Conjecture.
 Motivated by Theorem 1.1, we conjecture that there are only finitely many exceptions $V$ (i.e. $V\neq V_{t_V,0},V_{t_V,1}$) such that each denominator $t_i$ is an odd prime and $\mathbf{1}_V\in\mathcal{C}$. However, from Theorem \ref{titjthm}, this  is equivalent to that the following two statements hold simultaneously:
\begin{itemize}
  \item [(i)] there are only finitely many pairs of twin primes.
  \item [(ii)] there are only finitely many Sophie Germain primes.
\end{itemize}
\vskip2mm

The aforementioned theorems will be proved in Sections 6-8. In sections 2,3,4,5 we provide both necessary technical preparations and conceptual guidelines.

As we will explain in Subsection 2.2, a classical treatment to the completeness problem  is to invoke  a unitary transform, the so-called Beurling-Wintner transform (see (\ref{bw-trans})), which reduces the problem into a cyclic vector problem for multiplier algebra of the Hardy space of Dirichlet series.
Via the B-W transform,  the characteristic function $\mathbf{1}_V$ ($V$   given as in (\ref{Vdisplay'}))
  corresponds to  a translation $D(s+1)$ of some Dirichlet series $D$ with periodic coefficients.
Then we will apply a  result in  \cite{SW} to  factorize $D$, provided $D(s+1)$ to be cyclic, into a product of some Dirichlet polynomial and some Dirichlet-$L$ function.
This  reduces  the problem to the existence of such a decomposition and   the treatment  of zeros of the Dirichlet polynomial
\begin{equation}\label{PinIntro}
  P(s)=\sum_{d\mid q}\frac{(f\ast\mu\,\psi)(d)}{d^s}
\end{equation}
in some half-plane, where $q=2t_V$, $\psi$ is a primitive Dirichlet character (see Subsection 2.1 for the definition), and $\frac{\sqrt{2}}{2\pi n}f(n)$  is the Fourier-sine coefficients of $\mathbf{1}_V$, that is, $$\frac{\sqrt{2}}{2\pi n}f(n)=\int_{0}^{1}\mathbf{1}_V\cdot\sqrt{2}\sin (n\pi x)\mathrm{d}x,\quad n=1,2,\cdots.$$ See Subsection 3.1 for details.

%
As it is expectably  difficult to check decomposability of $D$, our observation is that the finite Fourier transform
$$g(m)=\frac{1}{q}\sum_{n=1}^{q}f(n)e^{-2\pi  imn/q},\quad m=1,2,\cdots$$
 of $f$ actually carries the information from boundary points of $V$.
 More precisely, $g(m)$ coincides with the jump of $\mathbf{1}_V$ at the point $\frac{m}{t_V}$.
 With the aid of the theory of the finite Fourier transform, we can establish new criteria of the Kozlov problem for $\mathbf{1}_V$ by using  the information from boundary points of $V$ instead of the Fourier-sine coefficients of $\mathbf{1}_V$. 

 In Section 4, we  determine the Dirichlet character $\psi$ appearing in (\ref{PinIntro}), which is of independent number theoretic interest. We will show if $\mathbf{1}_V\in\mathcal{C}$ then $\psi\equiv1$ under the assumptions $t_V\geq7$ and $3\nmid t_V$ (Theorem \ref{criterion}). The key ingredient to the proof of Theorem \ref{criterion} is a result in pure number theory (Lemma \ref{keylem}), which merely invokes the regularity of distribution of non-vanishing values of $g$.

In Section 5, we establish a formula for the number of  component intervals of $V$ with $\mathbf{1}_V\in\mathcal{C}$ by using the Euler totient function $\phi$.  As a consequence,
we show $t_V\leq6$ when $V$ is a subinterval of $(0,1)$ satisfying $\mathbf{1}_V\in\mathcal{C}$, which together with Theorem \ref{1to6} immediately gives Theorem \ref{interval}.
Moreover,
let $\mathcal{V}_n\ (n\in\mathbb{N})$ denote  the set of all non-degenerated open subsets $V$ of $(0,1)$ which have  at most $n$  component intervals with rational boundary points. It is  shown that the set $\big\{V\in \mathcal{V}_n: \mathbf{1}_V\in\mathcal{C}\big\}$ has  at most ${\binom {2n} {t(n)+1}}$ elements,
where $$t(n)=\prod_{\substack{p\leq 4n+1 \\ p\,\text{prime}}}p^{[\log_p6n]}.$$


\section{Preliminaries}

In this paper, we shall need a number of number-theoretic concepts and related results to  establish  criteria for completeness in more general case.
As a preliminary, this section consists of two parts. In the first part, we briefly recall some basic materials from analytic number theory. We refer the reader to \cite{Apo,Co,Da2,MV,O} for more details.
In the second part, we list some preparatory results on zeros of functions, which will be used in the establishment of  criteria for completeness.

\subsection{Some elements from analytic number theory}

An arithmetical function is a function defined on the set $\mathbb{N}$ of positive integers. The M\"{o}bius function $\mu$ and the Euler totient function $\phi$ are among the most basic arithmetical functions, where
\begin{equation*}
  \mu(n)=\begin{cases}
           1, &  n=1; \\
           (-1)^k, & n\mbox{ is the product of }k\mbox{ distinct primes}; \\
           0, & \mbox{otherwise},
         \end{cases}
\end{equation*} and $$\phi(n)=\sharp\{m:1\leq m\leq n,\gcd(m,n)=1\},\quad n=1,2,\cdots.$$
An arithmetical function $f$ is said to be multiplicative if
$f(1)=1$ and  $f(mn)=f(m)f(n)$ for any $m,n\geq1$ with $\gcd(m,n)=1$;
$f$ is said to be completely multiplicative
if
$f(1)=1$ and
$f(mn)=f(m)f(n)$ for any $m,n\geq1$. Both $\mu$ and $\phi$ are multiplicative.

Dirichlet characters with modulus $q$  arise from  characters
 of the group $(\mathbb{Z}/q\mathbb{Z})^\times$ of reduced residue classes mod $q$.
 Recall that a character of an abelian group $G$ is a group homomorphism from $G$ to the multiplicative group $\mathbb{C}^\times$ of nonzero complex numbers, and a character  is said to be principle if its value is identically $1$.
 Equivalently, Dirichlet characters $\chi$ mod $q$ can be also defined to be
completely multiplicative  and periodic arithmetical functions with period $q$, such that $\chi(n)\neq0$ if $\mathrm{gcd}(n,q)=1$, while $\chi(n)=0$ if $\mathrm{gcd}(n,q)>1$.
Correspondingly, a Dirichlet character is said to be principle if it only takes values $0$ or $1$.
There are exactly $\phi(q)$ distinct Dirichlet characters mod $q$.

 Suppose that $\chi$ is a Dirichlet character mod $q$.
 We say a Dirichlet character $\psi$ mod $d$ induces $\chi$ if
 $d\mid q$ and $\chi$ admits a decomposition $\chi=\psi\chi_q$,
where $\chi_q$ is the principle Dirichlet character mod $q$.
The Dirichlet character $\chi$ is said to be primitive if there is no proper divisor $d$ of $q$, such that $\chi$ is induced by a Dirichlet character mod $d$. From the definitions, it is readily seen that every Dirichlet character is uniquely induced by a primitive Dirichlet character.

The Legendre symbol
$(n\mid p)$ for an  odd prime $p$  is a primitive Dirichlet character mod $p$, which is defined as
\begin{equation*}
  (n\mid p)=\begin{cases}
              0, &  p\mid n; \\
              1, &  n\equiv m^2\,(\mathrm{mod}\,p)
              \mbox{ for some }1\leq m\leq p-1; \\
              -1, & \mbox{otherwise}.
            \end{cases}
\end{equation*}



 For any Dirichelt character $\chi$ mod $q$, the Gauss sum associated with $\chi$ is defined as
 $$\tau(n,\chi):=\sum_{m=1}^{q}\chi(m)e^{2\pi imn/q}.$$
The Gauss sum $\tau(1,\chi)$ is abbreviated as  $\tau(\chi)$.
The following formula  allows us to calculate $\tau(n,\chi)$ in terms of the primitive Dirichlet character inducing $\chi$ \cite[pp. 444-450]{Has} (also see \cite[Theorem 9.12]{MV}).

\begin{lem} \label{tau lemma}
  Let $\chi$ be a Dirichlet character mod $q$ induced by  the primitive Dirichlet character $\psi$ mod $q_0$.  Then for each $n\geq1$,
  $$\tau(n,\chi)=\frac{\phi(q)}{\phi(\frac{q}{\widehat{n}})}
  \mu(\frac{q}{\widehat{n}q_0})\psi(\frac{q}{\widehat{n}q_0})
  \overline{\psi(\frac{n}{\widehat{n}})}\tau(\psi),$$
where $\widehat{n}=\gcd(n,\frac{q}{q_0})$.
  In particular, when $q_0\nmid \frac{q}{\gcd(n,q)}$, we have
  $\tau(n,\chi)=0$ since $\psi(\frac{n}{\widehat{n}})=0$ in this case.
\end{lem}

For an arithmetical function $f$ we put
  $$D_f(s):=\sum_{n=1}^\infty\frac{f(n)}{n^s}.$$
  Then  $D_{f\ast g}=D_f\cdot D_g$ for any arithmetical functions $f,g$, where the Dirichlet convolution $f\ast g$ is given by
  $$(f\ast g)(n)=\sum_{d\mid n}f(\frac{n}{d})g(d),\quad n=1,2,\cdots.$$

If $\chi$ is a Dirichlet character, one defines its Dirichlet $L$-function by
$$ L(s,\chi )=D_{\chi}(s)=\sum _{n=1}^{\infty }\frac {\chi (n)}{n^{s}},\quad s\in\mathbb{C}_1.$$
 Here for a real number $\sigma $,   $$\mathbb{C}_\sigma :=\{s\in\mathbb{C}: \mathrm{Re} \, s>\sigma\}.$$ By analytic continuation, this function can be extended to a meromorphic function on the whole complex plane. Dirichlet $L$-functions are generalizations of the Riemann zeta-function. The case $\chi\equiv1$ corresponds to the Riemann zeta function $$\zeta(s)=1+\frac{1}{2^s}+\frac{1}{3^s}+\cdots.$$
Notice that  $L$-function of $\chi$  can  be written as an Euler product:
 $$L(s,\chi )=\prod_{p \,\text{prime}}\left(1-\chi(p)p^{-s}\right)^{-1},\quad s\in\mathbb{C}_1.$$
   Therefore, each Dirichlet $L$-function has no zeros in $\mathbb{C}_1$.

Dirichlet $L$-function was  introduced by Dirichlet to prove his celebrated
 theorem on arithmetic progressions. Dirichlet's theorem states that for any pair $a,q$ of relatively prime positive integers,
the arithmetic progression
$$a,\ a+q,\ a+2q,\cdots$$
contains infinitely many primes.

From now on, we  fix a positive integer $q$ in this subsection.
Let $H_q$ denote the linear space of all periodic arithmetic functions with period $q$, equipped with the inner product
$$\langle f,g\rangle=\sum_{n=1}^{q}f(n)\overline{g(n)},\quad f,g\in H_q.$$    For a divisor $d$ of $q$, and each Dirichlet character $\chi$ mod $\frac{q}{d}$, define
$$\xi_\chi(n)=\begin{cases}
                        \chi(\frac{n}{d}), &  d\mid n; \\
                        0, & d\nmid n.
                        \end{cases}$$
It was proved in \cite{CDZ} that there are exactly $q$ functions of the type $\xi_\chi$, and  these functions forms an orthogonal basis for $H_q$.

Let $\mathcal{P}_q$ denote the set of primitive Dirichlet characters  whose moduluses  are  divisors of $q$. For $\psi\in\mathcal{P}_q$ with  modulus $q_0$,   set $E_{q,\psi}$ to be the subspace of $H_q$ spanned by $$\{\xi_{\psi_{\frac{q}{d}}}:d\mid \frac{q}{q_0}\},$$
where $\psi_{\frac{q}{d}}$ is the Dirichlet character mod $\frac{q}{d}$ induced by $\psi$. Then $H_q$ admits an orthogonal  decomposition  \cite{CDZ}
\begin{equation}\label{decom of Hq}
  H_q=\bigoplus_{\psi\in\mathcal{P}_q} E_{q,\psi}.
\end{equation}
We give a  characterization of functions in $E_{q,\psi}$  which will be used in the sequel.

\begin{lem} \label{Epsi lemma}
 Suppose $f\in H_q$ and $\psi$ is a Dirichlet character mod $q_0$ in $\mathcal{P}_q$ (here $q_0\mid q$). Then the following statements are equivalent:
\begin{itemize}
  \item [(1)] $f\in E_{q,\psi}$;
   \item [(2)] $f(n)=f(\widehat{n})\psi(\frac{n}{\widehat{n}})$ for each $n\geq1$,
 where $\widehat{n}=\gcd(n,\frac{q}{q_0})$;
 \item [(3)]
 $(f\ast\mu\,\psi)(n)=0$ whenever $n\nmid \frac{q}{q_0}$;
 \item [(4)] $D_f(s)=P(s)L(s,\psi)$ for some Dirichlet polynomial $P$.
 \end{itemize}
\end{lem}

Before proving Lemma \ref{Epsi lemma}, we need an auxiliary lemma.
By the fact that the M\"{o}bius inversion of a multiplicative arithmetical function is also multiplicative, one has

\begin{lem} \label{muchi lemma}
  For $q\geq1$, let $\chi_q$ denote the principle Dirichlet character mod $q$. Then $$
        (\mu\ast\chi_q)(n)=\begin{cases}
                        \mu(n), &  n\mid q; \\
                        0, & n\nmid q.
                      \end{cases}
       $$
\end{lem}

\vskip2mm

\noindent\textbf{Proof of Lemma \ref{Epsi lemma}.}
 Let $\chi_m$ ($m\geq1$) denote the principle Dirichlet character mod $m$.
\vskip2mm
(1)$\Rightarrow$(2). It suffices to prove  the case $f=\xi_\chi$, where $\chi=\psi\chi_{\frac{q}{d}}$ with $d\mid \frac{q}{q_0}$, i.e.,  $\chi$ is the Dirichlet character mod $\frac{q}{d}$ induced by $\psi$.

Given $n\geq1$, we want to show
$\xi_\chi(n)=\xi_\chi(\widehat{n})\psi(\frac{n}{\widehat{n}})$,
 where $\widehat{n}=\gcd(n,\frac{q}{q_0})$.
If $d\nmid n$ then $\xi_\chi(n)=\xi_\chi(\widehat{n})=0$. So we may assume $d\mid n$, which immediately implies $d\mid \widehat{n}$ since
$d\mid\frac{q}{q_0}$. Therefore, we have $$\xi_\chi(n)=\chi(\frac{n}{d})=\psi(\frac{n}{d})
\chi_{\frac{q}{d}}(\frac{n}{d}),$$ $$\xi_\chi(\widehat{n})\psi(\frac{n}{\widehat{n}})
=\chi(\frac{\widehat{n}}{d})\psi(\frac{n}{\widehat{n}})=\psi(\frac{\widehat{n}}{d})
\chi_{\frac{q}{d}}(\frac{\widehat{n}}{d})\psi(\frac{n}{\widehat{n}})
=\psi(\frac{n}{d})\chi_{\frac{q}{d}}(\frac{\widehat{n}}{d}).$$
It remains to show $\chi_{\frac{q}{d}}(\frac{n}{d})
=\chi_{\frac{q}{d}}(\frac{\widehat{n}}{d})$ provided $\gcd(\frac{n}{d},q_0)=1$. If $\gcd(\frac{n}{d},q_0)=1$ then
$$\gcd(\frac{n}{d},\frac{q}{d})=\gcd(\frac{n}{d},\frac{q}{dq_0})
=\frac{1}{d}\gcd(n,\frac{q}{q_0})=\frac{\widehat{n}}{d},
$$
which gives that $\gcd(\frac{n}{d},\frac{q}{d})=1$ if and only if
$\gcd(\frac{\widehat{n}}{d},\frac{q}{d})=\frac{\widehat{n}}{d}=1$.
This completes the proof.

\vskip2mm

(2)$\Rightarrow$(3).
To reach a contradiction, we assume conversely that $(f\ast\mu\,\psi)(n)\neq0$ for some  $n\geq1$ with  $n\nmid \frac{q}{q_0}$. From the assumption in (2), we see that \begin{equation}\label{fmupsi}
(f\ast\mu\,\psi)(n)=\sum_{k\mid n}f(k)\mu(\frac{n}{k})\psi(\frac{n}{k})=\sum_{k\mid n}f(\widehat{k})\psi(\frac{k}{\widehat{k}})\mu(\frac{n}{k})\psi(\frac{n}{k}).
\end{equation}
Note that $\widehat{k}$  runs over all divisors of $\widehat{n}$, and $\widehat{k}=d$ for some divisor $d$ of $\widehat{n}$ if and only if
$k=ld$ with $\gcd(l,r_d)=1$, where $r_d=\frac{q}{dq_0}$.
By (\ref{fmupsi}), we have
\begin{equation*}
  \begin{split}
     (f\ast\mu\,\psi)(n) & =\sum_{d\mid\widehat{n}}f(d)\sum_{\substack{l\mid\frac{n}{d}  \\ (l,r_d)=1}}
     \psi(l)\mu(\frac{n}{ld})\psi(\frac{n}{ld}) \\
       & =\sum_{d\mid\widehat{n}}f(d)\psi(\frac{n}{d})
       \sum_{l\mid\frac{n}{d}}\mu(\frac{n}{ld})\chi_{r_d}(l) \\
       & =\sum_{d\mid\widehat{n}}f(d)\psi(\frac{n}{d})
       (\mu\ast\chi_{r_d})(\frac{n}{d}).
  \end{split}
\end{equation*}
This together with Lemma \ref{muchi lemma} implies $n\mid \frac{q}{q_0}$.

%
%

\vskip2mm
(3)$\Rightarrow$(4). Recall that $L(s,\psi)$  has no zeros in $\mathbb{C}_1$ and
  $\frac{1}{L(s,\psi)}=D_{\mu\,\psi}(s)$ on $\mathbb{C}_1$.
  Hence $$P(s)=\frac{D_f(s)}{L(s,\psi)}=D_f(s)D_{\mu\,\psi}(s)
  =D_{f\ast\mu\,\psi}(s)
  =\sum_{d\mid\frac{q}{q_0}}\frac{(f\ast\mu\,\psi)(d)}{d^s}.$$
In particular, $P$ is a Dirichlet polynomial.
\vskip2mm

(4)$\Rightarrow$(1). Assume that $D_f(s)=P(s)L(s,\psi)$ for some Dirichlet polynomial $P$.
By (\ref{decom of Hq}),  $f$ can be decomposed as $f=\sum_{\chi\in\mathcal{P}_q}f_\chi$, where $f_\chi\in E_{q,\chi}$. We have shown  that for each $\chi\in\mathcal{P}_q$,
there exist a Dirichlet polynomial $P_\chi$, such that $D_{f_\chi}(s)=P_\chi(s)L(s,\chi)$, which gives $$P(s)L(s,\psi)=D_f(s)=\sum_{\chi\in\mathcal{P}_q}D_{f_\chi}(s)
=\sum_{\chi\in\mathcal{P}_q}P_\chi(s)L(s,\chi).$$
Since the Dirichlet $L$-functions of primitive Dirichlet characters are linearly independent over the Dirichlet polynomials    \cite[Lemma 8.1]{KP}, it follows that $P_\psi=P$ and $P_\chi=0$ for any  $\chi\in\mathcal{P}_q$ other than $\psi$.
Then we have $$f=f_\psi\in E_{q,\psi}.$$
$\hfill \square $
\vskip2mm

Combining Lemma \ref{tau lemma} with Lemma \ref{Epsi lemma}, we immediately see the following.

\begin{lem} \label{tau Epsi}
  Suppose $\psi\in\mathcal{P}_q$. Then $\tau(\cdot,\chi)\in E_{q,\bar{\psi}}$, where $\chi$ is the Dirichlet character mod $q$ induced by $\psi$.
\end{lem}
\begin{proof}
  It is clear that $\tau(\cdot,\chi)\in H_q$ by the definition of the Gauss sum. Using Lemma \ref{tau lemma}, one can easily check that $\tau(n,\chi)=\tau(\widehat{n},\chi)
  \overline{\psi(\frac{n}{\widehat{n}})}$ for each $n\geq1$,
  where $\widehat{n}=\gcd(n,\frac{q}{q_0})$.
\end{proof}

Finally, we introduce some duality between the  subspaces $ E_{q,\psi}$ and $E_{q,\bar{\psi}}$.
For each $\eta\in H_q$, one can define a linear transform $\Gamma_\eta$ on $H_q$ by putting
$$(\Gamma_\eta f)(m)=\sum_{n=1}^{q}\eta(mn)f(n),\quad f\in H_q, \,\, n=1,2,\cdots.$$
Then  $\Gamma_\eta(E_{q,\psi})\subseteq {E_{q,\bar{\psi}}}$ for any $\psi\in\mathcal{P}_q$ \cite[Corollary 1]{CDZ}.
When $\eta$ is the exponential function $\eta(n)=\frac{1}{q}e^{-2\pi i n/q}$ ($n\geq1$), the transform $\mathcal{T}:=\Gamma_\eta$ coincides with the Fourier transform  on the group $\mathbb{Z}/q\mathbb{Z}$ (here functions in $H_q$ are naturally identified with functions on  $\mathbb{Z}/q\mathbb{Z}$).
Since the Fourier transform $\mathcal{T}$ is invertible and $\dim E_{q,\psi}=\dim E_{q,\bar{\psi}}$, we further conclude the following result.
\begin{lem} \label{Fourier transform}
 If $f\in H_q$ and $\psi\in\mathcal{P}_q$, then $f\in E_{q,\psi}$ if and only if $\mathcal{T}f\in E_{q,\bar{\psi}}$.
\end{lem}

\subsection{The Beurling-Wintner transform and some preparatory
results on zeros of functions}

Let us first recall     the  Beurling-Wintner transform.    Suppose that $\varphi\in L^2$ has its  Fourier-sine expansion $$\varphi(x)=\sum_{n=1}^{\infty}a_n\sqrt{2}\sin (n\pi x),\quad 0<x<1.$$
The Beurling-Wintner transform $\mathcal{D}\varphi$ of $\varphi \in L^2$ is the Dirichlet series
\begin{equation}\label{bw-trans}
\mathcal{D}\varphi(s):=\sum_{n=1}^{\infty}a_n n^{-s}.
\end{equation}
This transform $\mathcal{D}$ has the following property,
\begin{equation}\label{bw-mul}
\mathcal{D}(\sum_{k=1}^{K}c_k\varphi(kx))
=(\sum_{k=1}^{K}c_kk^{-s})\cdot\mathcal{D}\varphi.
\end{equation}
That is to say, a linear combination of  dilations of $\varphi$ is
   transformed into  multiplication of $\mathcal{D}\varphi$ by a Dirichlet polynomial.

 In fact, the B-W transform $\mathcal{D}$ is a unitary transform  from  $L^2$ onto
the Hardy  space of Dirichlet series
$$\mathcal{H}^2:=
\{D=\sum_{n=1}^{\infty}a_nn^{-s}:\|D\|^2=\sum_{n=1}^{\infty}|a_n|^2<\infty\},$$
which is introduced in \cite{HLS}.

Writing a complex variable $s=\sigma + it$,  for each Dirichlet series $D$, there exists a unique  $\sigma_a(D)\in[-\infty,+\infty]$ (called  the abscissa of absolute convergence of $D$) such that  if  $\sigma> \sigma_a(D)$,  the  series $D$ converges absolutely,  but not if  $\sigma< \sigma_a(D)$.  By the Cauchy-Schwarz inequality, for  each $D\in\mathcal{H}^2$ one has
$\sigma_a(D)\leq\frac{1}{2}$. Thus $D$ defines a holomorphic function in $\mathbb{C}_{\frac{1}{2}}$.
It was shown in \cite{HLS} that the set $\mathcal{H}^\infty$ of
Dirichlet series that can be extended to bounded holomorphic functions on $\mathbb{C}_0$  coincides with the multiplier algebra
of $\mathcal{H}^2$.

The Bohr transform $\mathcal{B}$ appeared much earlier. Let  $p_j\ (j\in\mathbb{N})$ be the $j$-th prime number. In 1913, Bohr noticed that
the terms $p_1^{-s},p_2^{-s},\cdots$ in Dirichlet series possess some kind of independence \cite{Bo}.
By the variable substitution
 $$z_1=p_1^{-s}, z_2=p_2^{-s}, \cdots,$$ a Dirichlet series $D$  is transformed into
 a power series $\mathcal{B}D$ in infinitely many variables.
To be more specific, for any $n\geq1$ let $n=p_1^{\alpha_1}\cdots p_l^{\alpha_l}$ be its prime factorization. Then one can define a finitely supported sequence $\alpha(n)$ for each natural number $n$ by putting
$$\alpha(n)=(\alpha_1,\cdots,\alpha_l,0,0,\cdots).$$
Denote the monomial $z_1^{\alpha_1}\cdots z_l^{\alpha_l}$ by $\mathbf{z}^{\alpha(n)}$. The Bohr transform of a Dirichlet series 
is defined as follows,
\begin{equation}\label{bohr-trans}
\mathcal{B}(\sum_{n=1}^{\infty}a_nn^{-s})
=\sum_{n=1}^{\infty}a_n\mathbf{z}^{\alpha(n)}.
\end{equation}

Set $\mathcal{F}=\mathcal{B}\mathcal{D}$ and let $\mathbb{D}_2^\infty$ denote  Hilbert's multidisk
$$\mathbb{D}_2^\infty=\{\mathbf{z}=(z_1,z_2,\cdots)\in l^2:|z_j|<1\text{ for each }j\in\mathbb{N}\}.$$
Let $\varphi\in L^2$, and  $\varphi(x)=\sum_{n=1}^{\infty}a_n\sqrt{2}\sin (n\pi x) $ be its  Fourier-sine expansion. Then by the Cauchy-Schwarz inequality,  the power series $\mathcal{F}\varphi=\sum_{n=1}^{\infty}a_n\mathbf{z}^{\alpha(n)}$ converges pointwise in $\mathbb{D}_2^\infty$.

The Hardy space
$H^2(\mathbb{D}_2^\infty)$ over the infinite polydisk $\mathbb{D}_2^\infty$, defined as
$$H^2(\mathbb{D}_2^\infty):=\{F=\sum_{n=1}^{\infty}a_n\mathbf{z}^{\alpha(n)}:
\|F\|^2=\sum_{n=1}^{\infty}|a_n|^2<\infty\},$$ is an analytic function space on $\mathbb{D}_2^\infty$ \cite{Ni1,Ni5}.
The multiplier algebra $H^\infty(\mathbb{D}_2^\infty)$  of $H^2(\mathbb{D}_2^\infty)$ is exactly the set of bounded holomorphic functions on $\mathbb{D}_2^\infty$ \cite{Ni1}.

The Bohr transform $\mathcal{B}$,  restricted on $\mathcal{H}^2$,  is a unitary transform from
$\mathcal{H}^2$ onto $H^2(\mathbb{D}_2^\infty)$ which has the following properties \cite{HLS}:
\begin{itemize}
  \item [(1)] if $f$ is a multiplier of $\mathcal{H}^2$, then $\mathcal{B}f $ is a  multiplier of $H^2(\mathbb{D}_2^\infty)$, and $\mathcal{B}(fh)=\mathcal{B}f\mathcal{B}h, \,\, h\in \mathcal{H}^2$;
    \item [(2)] the Bohr transform $\mathcal{B}$  establishes an  isometric isomorphism from  the Banach algebra $\mathcal{H}^\infty$ onto the Banach algebra $H^\infty(\mathbb{D}_2^\infty)$.
 \end{itemize}

Following \cite{HLS, Ni1},
we say that a Dirichlet series $D\in\mathcal{H}^2$ is cyclic if the multiplier invariant subspace  generated by $D$ is the whole space $\mathcal{H}^2$, and correspondingly, a function $F\in H^2(\mathbb{D}_2^\infty)$ is cyclic if the multiplier invariant subspace  generated by $F$ is the whole space $H^2(\mathbb{D}_2^\infty)$.

The following result, coming from \cite{HLS} (also see \cite{Ni1,Ni5}), translates the B-W problem into
the  cyclic vector problem in $\mathcal{H}^2$ or $H^2(\mathbb{D}_2^\infty)$, respectively.

\begin{prop} \label{cyclic}
  Suppose $\varphi\in L^2$. Then the following statements are equivalent:
\begin{itemize}
  \item [(1)] the \text{p.d.s.} $\{\varphi(kx):k\in\mathbb{N}\}$ of $\varphi$ is complete;
  \item [(2)] $\mathcal{D}\varphi$ is cyclic in $\mathcal{H}^2$;
 \item [(3)]
 $\mathcal{F}\varphi$ is cyclic in $H^2(\mathbb{D}_2^\infty)$.
 \end{itemize}
\end{prop}

Beurling first gave
a necessary condition for $\varphi$ to be in $\mathcal{C}$, which is
the following \cite{Beu} (also see  \cite[Corollary 6.6.3]{Ni5}).
\begin{lem} \label{Beurling}
 If $\varphi\in \mathcal{C}$, then $\mathcal{F}\varphi$ has no zero in $\mathbb{D}_2^\infty$.
\end{lem}

We also record the following  lemma, which  is needed in this paper when we consider zeros of Dirichlet series.
Note that if $D$ is a Dirichlet series with $\sigma_a(D)\leq0$,
then $D(s+\delta)\in\mathcal{H}^\infty$ for each $\delta>0$, and thus $\mathcal{B}D$ is well-defined on
$$\{(p_1^{-\delta}z_1,p_2^{-\delta}z_2,\cdots): \text{for all} \,\delta>0, \text{and}\,
\mathbf{z}=(z_1,z_2,\cdots)\in\mathbb{D}_2^\infty\},$$
which contains   the set $\mathbb{D}_0^\infty$, where
$$\mathbb{D}_0^\infty=\{\mathbf{z}\in\mathbb{D}_2^\infty: \text{ $\mathbf{z}$  has  only   finitely many nonzero entries}\}.$$

\begin{lem} \label{zero and Bohr}
Let $D(s)=\sum_{n=1}^{\infty}a_nn^{-s}$ be a Dirichlet series with
$\sigma_a(D)\leq0$.
 Then $\mathcal{B}D$  has no zeros in $\mathbb{D}_0^\infty$ if and only if $a_1\neq0$ and $D$ has  no zero in    $\mathbb{C}_0$.
\end{lem}
\begin{proof}
Set $F=\mathcal{B}D$.
For the necessity, assume that $F$  has no zeros in $\mathbb{D}_0^\infty$.  Let $p_j\ (j\in\mathbb{N})$ denote the $j$-th prime, and
put $$D_j(s)=\sum_{n=1}^{\infty}a_n\rho_j(n)n^{-s},\,  j=1,2,\cdots,$$
where $\rho_j$ is a completely multiplicative arithmetical functions defined by putting $\rho_j(p_1)=\cdots=\rho_j(p_j)=\rho_j(1)=1$ and
$\rho_j(p_{j+1})=\rho_j(p_{j+2})=\cdots=0$.
Then for each $j\in\mathbb{N}$, $\sigma_a(D_j)\leq0$ and
$$D_j(s)=F(p_1^{-s},\cdots,p_j^{-s},0,0,\cdots),\quad s\in\mathbb{C}_0,$$
which yields that $D_j$ has no zeros in    $\mathbb{C}_0$.
Moreover, $\{D_j\}_{j\geq1}$ converges  uniformly to $D$  on each $\mathbb{C}_\sigma$ with $\sigma>0$. Combining $a_1=F(\mathbf{0})\neq0$ with Hurwitz’s theorem  shows that $D$ has no zeros in   $\mathbb{C}_0$.

Now suppose that $a_1\neq0$ and $D$ has no zeros in  $\mathbb{C}_0$.
  To reach a contradiction, assume conversely that $F$ has a zero $\mathbf{w}=(w_1,\cdots,w_N,0,0,\cdots)$ $(N\in\mathbb{N})$ in $\mathbb{D}_0^\infty$. Then one can take some $\delta>0$ such that $|w_j|< p_j^{-\delta}\ (1\leq j\leq N)$, which implies that
  $$\mathcal{B}(D(s+\delta))(\mathbf{z})
  =F(p_1^{-\delta}z_1,p_2^{-\delta}z_2,\cdots),\quad
  \mathbf{z}\in\mathbb{D}_2^\infty$$ also has a zero $\widetilde{\mathbf{w}}=(p_1^{\delta}w_1,\cdots,p_N^{\delta}w_N,0,0,\cdots)\in\mathbb{D}_0^\infty$.
  We conclude that $D(s+\delta)$ is  not invertible in the algebra $\mathcal{H}^\infty$, otherwise $\frac{1}{D(s+\delta)}\in\mathcal{H}^\infty$ and since
  $\mathcal{B}$ is multiplicative on $\mathcal{H}^\infty$,
  $$1
  =\mathcal{B}(D(s+\delta))(\widetilde{\mathbf{w}})
  \cdot\mathcal{B}(\frac{1}{D(s+\delta)})(\widetilde{\mathbf{w}})=0.$$
  Hence we can find a sequence $\{s_k\}_{k\geq1}$  in $\mathbb{C}_\delta$ satisfying  $D(s_k)\rightarrow0\ (k\rightarrow\infty)$.

 Since $D$ is zero-free and bounded on  $\mathbb{C}_{\frac{\delta}{2}}$, the functions
 $$h_k(s)=D(s+s_k-\delta),\quad k=1,2,\cdots$$
 constitute a  normal family on $\mathbb{C}_{\frac{\delta}{2}}$.
 Then one can take a subsequence $\{h_{k_i}\}_{i\geq1}$ converging to a holomorphic function $h$ uniformly  on compact subsets of $\mathbb{C}_{\frac{\delta}{2}}$. Since   each $h_k$ is zero-free and $$h(\delta)=\lim_{i\rightarrow\infty}h_{k_i}(\delta)=0,$$ we have $h\equiv0$ by Hurwitz’s theorem.
 On the other hand, recall that a Dirichlet series converges uniformly to its constant term when $\mathrm{Re}\,s$ tends to $+\infty$, and hence we can choose $M$ sufficiently large so that $|D(s)|\geq\frac{|a_1|}{2}$  for $s\in\mathbb{C}_M$. This gives that $|h_k(s)|\geq\frac{|a_1|}{2}$  whenever $\mathrm{Re}\,s>M$, which contradicts with $h\equiv0$.
 This shows that $F=\mathcal{B}D$  has no zero in $\mathbb{D}_0^\infty$, completing the proof.
\end{proof}

The following result follows immediately.
\begin{lem} \label{D cyclic no zeors}
  If $D\in\mathcal{H}^2$ and $D$ is cyclic, then
  $D$ has a nonzero constant term and no zeros in $\mathbb{C}_\sigma$, where $\sigma=\max\{\sigma_a(D),0\}$.
\end{lem}
\begin{proof}
  Assume $D\in\mathcal{H}^2$ is cyclic. Then $D(s+\sigma)$ is also cyclic in $\mathcal{H}^2$, which follows from
 $$\inf_{\psi\in\mathcal{H}^\infty}\|\psi D_\sigma-1\|\leq\inf_{\psi\in\mathcal{H}^\infty}\| \psi_\sigma D_\sigma-1\|\leq\inf_{\psi\in\mathcal{H}^\infty}\| \psi D-1\|=0,$$
 where $D_\sigma(s)=D(s+\sigma), \, \psi_\sigma(s)=\psi(s+\sigma)$.    Combining Proposition
  \ref{cyclic} with Lemma \ref{Beurling}, we see that $\mathcal{B}D_\sigma$ has no zero in $\mathbb{D}_2^\infty$. Therefore, Lemma \ref{zero and Bohr}
  yields the desired conclusion.
\end{proof}

It was shown in \cite{NGN} that Beurling's necessary condition (Lemma \ref{Beurling}) is also sufficient for a finite linear combinations of $\sin(\pi x),\sin(2\pi x),\cdots$. One can use Lemma \ref{zero and Bohr} to get its Dirichlet series version.

\begin{lem} \label{NGN Dirichlet}
  Let $P$ be a Dirichlet polynomial. Then $P$ is cyclic in $\mathcal{H}^2$ if and only if $P$ has a nonzero constant term and no zero in $\mathbb{C}_0$.
\end{lem}

Below  we turn to functions in $H_q$.  In \cite{SW}, Saias and Weingartner considered  functions  that do not belong to any  $E_{q,\psi}$ (namely $f\notin\bigcup_{\psi\in\mathcal{P}_q}E_{q,\psi}$), and proved that the Dirichlet series $D_f$ always have zeros in $\mathbb{C}_1$.
Their result can be refined as follows (remark that the equivalence of (2) and (4) has been established in \cite{SW}) which will be used in the next section.

\begin{prop} \label{Hq no zero}
  Suppose $f\in H_q$. Then the following statements are equivalent:
\begin{itemize}
  \item [(1)] $D_f(s+1)$ is cyclic in $\mathcal{H}^2$;
  \item [(2)] $f(1)\neq0$ and $D_f$ has   no zero in $\mathbb{C}_1$;
 \item [(3)]  $f(1)\neq0$, $f\in E_{q,\psi}$ for some $\psi\in\mathcal{P}_q$, and the Dirichlet polynomial $$P(s)=\sum_{d\mid\frac{q}{q_0}}\frac{(f\ast\mu\,\psi)(d)}{d^s}$$
  has no zeros in $\mathbb{C}_1$;
     \item [(4)] $f(1)\neq0$ and $D_f(s)=P(s)L(s,\psi)$, where $\psi\in\mathcal{P}_q$ and $P$ is a Dirichlet polynomial that has  no zero in $\mathbb{C}_1$.
 \end{itemize}
\end{prop}
\begin{proof}
  The implication (1)$\Rightarrow$(2) has been established by Lemma \ref{D cyclic no zeors}. The implications (2)$\Rightarrow$(3)$\Rightarrow$(4) immediately follows from the result of Saias and Weingartner and Lemma \ref{Epsi lemma}.

  \noindent(4)$\Rightarrow$(1): Let $\chi_n\ (n\geq1)$ denote  the principle Dirichlet character mod $n$, and $M$  the  multiplier invariant subspace of $\mathcal{H}^2$ generated by $Q(s)L(s+1,\psi)$, where
$Q(s)=P(s+1)$. By Lemma \ref{NGN Dirichlet}, it suffices to show $Q\in M$.  To this end, we only need to prove  $Q(s)L(s+1,\psi\chi_{k!})\in M$ for each $k>q$ since
$L(s+1,\psi\chi_{k!})$ converges to the constant function $1$ in $\mathcal{H}^2$-norm
as $k\rightarrow\infty$.

For each $n\geq1$ put
$\rho_{n}=\mu\ast\chi_{n}$. Then using Lemma \ref{muchi lemma} we see that
$D_{\psi\rho_{n}}\ (n\geq1)$ is a Dirichlet polynomial, and hence $D_{\psi\rho_{n}}(s+1)\in\mathcal{H}^\infty$. Since for any $k>q$,
$$\psi\ast\psi\rho_{k!}=\psi\ast\mu\,\psi\ast\psi\chi_{k!}
=\psi\chi_{k!},$$
we have
$$Q(s)L(s+1,\psi\chi_{k!})=Q(s)L(s+1,\psi)D_{\psi\rho_{k!}}(s+1)\in M.$$
This completes the proof.
\end{proof}


%

\section{Some criteria}
In this section, we first give a preliminary criterion (Theorem \ref{precri} below) to determine when the characteristic function of an open subset of $(0,1)$ consisting of finitely many component intervals with rational endpoints generates complete periodic dilation system.
Then based on Theorem \ref{precri}, we establish new criteria (Theorems 3.9, 3.10 and 3.12) by virtue of information from boundary points of $V$. These criteria will be used in the subsequent sections.

\subsection{A preliminary criterion}
To state Theorem \ref{precri}, we need to establish some  notations, which will be used throughout the rest of this paper.

Let us recall the notations in Introduction.
Suppose $$0\leq\frac{s_1}{t_1}<\frac{s_1'}{t_1'}<\cdots<\frac{s_N}{t_N}<\frac{s_N'}{t_N'}\leq1,$$ where  $\frac{s_i}{t_i},\frac{s_i'}{t_i'}\ (i=1,2,\cdots,N)$ are irreducible fractions. Set $$V=(\frac{s_1}{t_1},\frac{s_1'}{t_1'})\cup\cdots
\cup(\frac{s_N}{t_N},\frac{s_N'}{t_N'}),\eqno{(\ref{Vdisplay})}$$
   $$t_V=\mathrm{lcm}(t_1, t_1', \cdots, t_N, t_N').\eqno{(\ref{tdisplay})}$$
Then we can rewrite $V$ as $$V=(\frac{\alpha_1}{t_V},\frac{\beta_1}{t_V})\cup\cdots
\cup(\frac{\alpha_N}{t_V},\frac{\beta_N}{t_V}).\eqno{(\ref{Vdisplay'})}
$$

Since $\{\sqrt{2}\sin(n\pi x):\, n\in \mathbb{N}\}$ is  a canonical  orthonormal
  basis of $L^2$, we have $$\mathbf{1}_V=2\sum_{n=1}^\infty\Big[\int_{0}^{1}\mathbf{1}_V\sin (n\pi x)\mathrm{d}x\,\Big]\,\sin (n\pi x).$$
  Set
  \begin{equation}\label{q}
                                                                  q=2t_V
                                                                \end{equation} and
\begin{equation}\label{f(n)}
f(n)=2\pi n\int_{0}^{1}\mathbf{1}_V\sin (n\pi x)\mathrm{d}x=2\pi n\int_{V}\sin (n\pi x)\mathrm{d}x,\quad n=1,2,\cdots.\end{equation}
 Then  by (\ref{Vdisplay'}),  \begin{equation}\label{cos-cos}
        f(n)=2\pi n\sum_{i=1}^{N}\int_{\frac{\alpha_i}{t_V}}^{\frac{\beta_i}{t_V}}\sin (n\pi x)\mathrm{d}x=2\sum_{i=1}^{N}\left(\cos\frac{n\pi\alpha_i}{t_V}-\cos\frac{n\pi\beta_i}{t_V}\right)
      \end{equation}
      In particular,
      \begin{equation}\label{f(1)}
        f(1)=2\sum_{i=1}^{N}\left(\cos\frac{\pi\alpha_i}{t_V}-\cos\frac{\pi\beta_i}{t_V}\right)>0
      \end{equation} since the function $\cos x$ is strictly decreasing on $[0,\pi]$.
It follows  from (\ref{cos-cos}) that  $f\in H_q$, i.e., $f(n+q)=f(n)$ for each $n\geq1$. An application of  the B-W transform to $$\mathbf{1}_V=\frac{1}{\pi}\sum_{n=1}^\infty\frac{f(n)}{n}\,\sin (n\pi x)$$ gives
\begin{equation} \label{B-W trans}
\mathcal{D}\mathbf{1}_V(s)=\frac{\sqrt{2}}{2\pi}\sum_{n=1}^{\infty} \frac{f(n)}{n^{s+1}}=\frac{\sqrt{2}}{2\pi}D_f(s+1).
\end{equation}
Recall that the Dirichlet series $D_f$ is defined by
  $$D_f(s)=\sum_{n=1}^\infty\frac{f(n)}{n^s}.$$
 It  is deduced from (\ref{B-W trans}) that the B-W transform $\mathcal{D}\mathbf{1}_V$  of $\mathbf{1}_V$ is a translation of $\frac{\sqrt{2}}{2\pi}D_f$.

 \begin{thm}\label{precri}
   Let $V,q$ and $f$ be given as in (\ref{Vdisplay}), (\ref{q}) and (\ref{f(n)}). Then the following statements are equivalent:
   \begin{itemize}
     \item [(1)] $\mathbf{1}_V\in\mathcal{C}$;
     \item [(2)] $\mathcal{D}\mathbf{1}_V$ has no zero in $\mathbb{C}_0$;
     \item [(3)] $\mathcal{F}\mathbf{1}_V$ has no zero in $\mathbb{D}_2^\infty$;
     \item [(4)] there exists some  $\psi\in\mathcal{P}_q$ such that $f\in E_{q,\psi}$ and  the Dirichlet polynomial $$P(s)=\sum_{d\mid\frac{q}{q_0}}\frac{(f\ast\mu\,\psi)(d)}{d^s}$$
  has no zero in $\mathbb{C}_1$, where $q_0$ is the modulus of $\psi$.
     \end{itemize}
 \end{thm}
 \begin{proof}
   Since $\mathcal{D}\mathbf{1}_V(s)=\frac{\sqrt{2}}{2\pi}D_f(s+1)$ for some $f\in H_q$ with $f(1)\neq0$ (see (\ref{B-W trans}) and (\ref{f(1)})), the equivalence of (1), (2) and (4) immediately follows from Proposition \ref{Hq no zero}.
  Moreover, the implication (1)$\Rightarrow$(3) is due to Beurling's necessary condition (Lemma \ref{Beurling}), and the implication (3)$\Rightarrow$(2) follows immediately from Lemma \ref{zero and Bohr}.
 \end{proof}

 \begin{rem}\label{uniquepsi}
 \begin{itemize}
 \item[(1)] When $\mathbf{1}_V\in\mathcal{C}$, the primitive Dirichlet character $\psi$ in Theorem \ref{precri} (4) is uniquely determined by the requirement $f\in E_{q,\psi}$. In fact, if $f\in E_{q_1,\psi_1}$ for some $q_1\in\mathbb{N}$ and some $\psi_1\in\mathcal{P}_{q_1}$, then it follows from Lemma \ref{Epsi lemma} that there exist Dirichlet polynomials $P$ and $P_1$ such that
 $$P_1(s)L(s,\psi_1)=D_f(s)=P(s)L(s,\psi).$$
 By  the fact that the Dirichlet $L$-functions of primitive Dirichlet characters are linearly independent over the Dirichlet polynomials  \cite[Lemma 8.1]{KP}, one  has $\psi_1=\psi$.
 \item[(2)]  Beurling's necessary condition (Lemma \ref{Beurling}) is also sufficient for the characteristic function $\mathbf{1}_V$ of an open subset $V$ of $(0,1)$ consisting of finitely many component intervals with rational endpoint.
 \end{itemize}
 \end{rem}

 Combining Theorem \ref{precri} with Remark \ref{uniquepsi} (1), one obtains

 \begin{cor}\label{corof1VC}
   Let $V,q$ and $f$ be given as in (\ref{Vdisplay}), (\ref{q}) and (\ref{f(n)}). If $f\in E_{q,\psi}$ for some $\psi\in\mathcal{P}_q$, then $\mathbf{1}_V\in\mathcal{C}$ if and only if
    the Dirichlet polynomial $$P(s)=\sum_{d\mid\frac{q}{q_0}}\frac{(f\ast\mu\,\psi)(d)}{d^s}$$
  has no zeros in $\mathbb{C}_1$, where $q_0$ is the modulus of $\psi$.
 \end{cor}

 Since $f(1)\neq0$ (see (\ref{f(1)})), we have the following direct consequence of Corollary \ref{corof1VC}.

\begin{cor}\label{q=q01}
   Let $V,q$ and $f$ be given as in (\ref{Vdisplay}), (\ref{q}) and (\ref{f(n)}). If $f\in E_{q,\psi}$ for some $\psi\in\mathcal{P}_q$, and the modulus of $\psi$ is $q$, then $\mathbf{1}_V\in\mathcal{C}$.
\end{cor}

We can obtain some examples of open sets $V$ with $\mathbf{1}_V\in\mathcal{C}$ by using Theorem \ref{precri}.

 \begin{exam}\label{01}
   For $V=(0,1)$, we have
   $$f(n)=2(\cos0-\cos(n\pi))=2-2\cos (n\pi)=4\chi_2(n),$$
   where $\chi_2$ is the Dirichlet character mod $2$. Therefore, $$\mathcal{D}\mathbf{1}_V(s)=\frac{\sqrt{2}}{2\pi}D_f(s+1)
   =\frac{\sqrt{2}}{\pi}L(s+1,\chi_2)$$
   has no zeros in $\mathbb{C}_0$, forcing $\mathbf{1}_V\in\mathcal{C}$.
 \end{exam}

\begin{exam}\label{Vt0Vt1} Set $$V_{t,0}=\begin{cases}
	(0,\frac{1}{t})\cup(\frac{2}{t},\frac{3}{t})\cup\cdots\cup(\frac{t-2}{t},\frac{t-1}{t}), & t=2,4,6,\cdots; \\
	(0,\frac{1}{t})\cup(\frac{1}{t},\frac{2}{t})\cup\cdots\cup(\frac{t-1}{t},1), &  t=3,5,7,\cdots,
\end{cases}
$$ and $$V_{t,1}=\begin{cases}
	(\frac{1}{t},\frac{2}{t})\cup(\frac{3}{t},\frac{4}{t})\cup\cdots\cup(\frac{t-1}{t},1), & t=2,4,6,\cdots; \\
	(\frac{1}{t},\frac{2}{t})\cup(\frac{3}{t},\frac{4}{t})\cup\cdots\cup(\frac{t-2}{t},\frac{t-1}{t}), &  t=3,5,7,\cdots.
\end{cases}
$$
Then for $t\geq2$, $\mathbf{1}_{V_{t,0}},\mathbf{1}_{V_{t,1}}\in\mathcal{C}$.
In fact, set $\varphi\equiv1$ on $(0,1)$ and extend $\varphi$ to an odd $2$-periodic function on $\mathbb{R}$, which is still denoted by $\varphi$.
  Hence $$\mathbf{1}_{V_{t,0}}(x)=\frac{\varphi(x)+\varphi(tx)}{2},\quad 0<x<1$$ and
  $$\mathbf{1}_{V_{t,1}}(x)=\frac{\varphi(x)-\varphi(tx)}{2},\quad 0<x<1,$$
   which gives
$$\mathcal{D}\mathbf{1}_{V_{t,0}}(s)=\frac{1+t^{-s}}{2}\mathcal{D}\varphi(s)
=\frac{\sqrt{2}}{\pi}(1+t^{-s})L(s+1,\chi_2)$$
and $$\mathcal{D}\mathbf{1}_{V_{t,1}}(s)=\frac{1-t^{-s}}{2}\mathcal{D}\varphi(s)
=\frac{\sqrt{2}}{\pi}(1-t^{-s})L(s+1,\chi_2).$$
Obviously, both $\mathcal{D}\mathbf{1}_{V_{t,0}}$ and $\mathcal{D}\mathbf{1}_{V_{t,1}}$ have no zero in $\mathbb{C}_0$.
\end{exam}

\subsection{Some criteria via boundary points of $V$}
 We continue to establish several criteria that are more convenient in some special cases, especially when the least common denominator $t_V$ is small (Theorem \ref{1to6}).
 To this end, we  represent the function $f$ in another way.

Let $J_V(x)\ (x\in\mathbb{R})$ denote the jump of $\mathbf{1}_V$  at $x$:
$$J_V(x)=\mathbf{1}_V(x^+)-\mathbf{1}_V(x^-)=\lim_{u\rightarrow x^+}\mathbf{1}_V(u)-\lim_{v\rightarrow x^-}\mathbf{1}_V(v),$$
where $\mathbf{1}_V$ is identified with its odd $2$-periodic extension on $\mathbb{R}$.
The  function $\mathbf{1}_V$ can be represented as
$$\mathbf{1}_V=\frac{1}{2}J_V(0)\, \mathbf{1}_{(0,1)}
+\sum_{m=1}^{t_V-1}J_V(\frac{m}{t_V})\mathbf{1}_{(\frac{m}{t_V},1)}.$$
Then  by (\ref{f(n)}),
\begin{equation} \label{rep of f}
  \begin{split}
     f(n) & =J_V(0)+J_V(1)\cos (n\pi)+2\sum_{m=1}^{t_V-1}J_V(\frac{m}{t_V})\cos \frac{mn\pi}{t_V} \\
       & =\sum_{m=1}^{q}J_V(\frac{m}{t_V})e^{2\pi  imn/q}.
  \end{split}
\end{equation}
 Put \begin{equation}\label{g(m)}
       g(m)=J_V(\frac{m}{t_V}),\quad m\in\mathbb{Z}.
     \end{equation}
 Then $g$ coincide with  the Fourier transform $\mathcal{T}f$ of $f$ since
 $$g(m)=\frac{1}{q}\sum_{n=1}^{q}f(n)e^{-2\pi  imn/q}$$  by (\ref{rep of f}) (see \cite[Theorem 8.4]{Apo}, also see Subsection 2.1 for the definition of the transform $\mathcal{T}$).
Since $f$ is real-valued, if $f\in E_{q,\psi}$ for some $\psi\in\mathcal{P}_q$ then $\psi$ is necessarily real-valued by Lemma \ref{q0qd} below, which yields that $g$  belongs to the same space $E_{q,\psi}$ by Lemma \ref{Fourier transform}. In the same way, if $g\in E_{q,\psi}$ for some $\psi\in\mathcal{P}_q$ then $f\in E_{q,\psi}$.

 \begin{lem} \label{q0qd}
 Suppose  $\psi\in\mathcal{P}_q$ is of modulus $q_0$, and $f\in E_{q,\psi}$ with $f\not\equiv0$.
  If $f$ is real-valued, then $\psi$ is also real-valued.
\end{lem}
\begin{proof}
 It suffices to show $\psi(a)\in\mathbb{R}$  for every  $a\in\mathbb{N}$ with $\gcd(a,q_0)=1$.
By Dirichlet’s Theorem on arithmetic progressions (see Subsection 2.1), for  such $a$ there is a prime $p>q$ such that
$p\equiv a\,(\mathrm{mod}\,q_0)$. Choose $m\in\mathbb{N}$ satisfying $f(m)\neq0$, and put $n=\gcd(m,\frac{q}{q_0})$. Then $$\gcd(pm,\frac{q}{q_0})=\gcd(m,\frac{q}{q_0})=n,$$
which gives $$f(pm)=f(n)\psi(\frac{pm}{n})=f(n)\psi(\frac{m}{n})\psi(p)=f(m)\psi(p).$$
Therefore
 $\psi(a)=\psi(p)=\frac{f(pm)}{f(m)}$ is a real number.
%
%
\end{proof}

 Intuitively, it is much easier to check that $g$ belongs to one such  space other than $f$ since $g(m)$ only takes values $0,1$ or $-1$ for $1\leq m\leq t_V-1$. So one may expect to use the information from boundary points of $V$ instead of the sine-Fourier coefficients of $\mathbf{1}_V$.
   We summarize the above discussion into the following result.

     \begin{cor} \label{fg}
       Let $V,q,f$ and $g$ be given as in (\ref{Vdisplay}), (\ref{q}), (\ref{f(n)}) and (\ref{g(m)}). If $f\in E_{q,\psi}$ for some $\psi\in\mathcal{P}_q$, then  $g\in E_{q,\psi}$,
       and vice versa.

       In particular,
       if $\mathbf{1}_V\in\mathcal{C}$ then $g\in E_{q,\psi}$ for some $\psi\in\mathcal{P}_q$.
     \end{cor}

We will use Corollary \ref{fg} to establish Theorems \ref{q=q02}, \ref{specriterion} and \ref{Eq}.

\begin{thm}\label{q=q02}
  Let $V,q$ and $g$ be given as in (\ref{Vdisplay}), (\ref{q}) and (\ref{g(m)}). Suppose that $g\in E_{q,\psi}$ for some $\psi\in\mathcal{P}_q$, and the modulus of $\psi$ is $q$, then $\mathbf{1}_V\in\mathcal{C}$.
\end{thm}

Theorem \ref{q=q02} can be obtained by combining  Corollary \ref{fg} with Corollary \ref{q=q01}.

\begin{thm}\label{specriterion}
Let $V,q$ and $g$ be given as in (\ref{Vdisplay}), (\ref{q}) and (\ref{g(m)}). Suppose that $g\in E_{q,\psi}$ for some $\psi\in\mathcal{P}_q$, and let $q_0$ denote the modulus of $\psi$. We also assume that
\begin{itemize}
  \item [(1)] $\frac{q}{q_0}$ is a product of some distinct primes that cannot divide $q_0$.
  \item [(2)] $g(d_0)\neq0$ for some divisor $d_0$ of $\frac{q}{q_0}$, and $g(d)=0$ for any other divisor $d$ of $\frac{q}{q_0}$.
\end{itemize}
Then $\mathbf{1}_V\in\mathcal{C}$.
\end{thm}

To prove Theorem \ref{specriterion}, we need the following.

\begin{lem} \label{f=gtau} Suppose that $f\in E_{q,\psi}$ for some $\psi\in\mathcal{P}_q$, and the modulus of $\psi$ is $q_0$.
Then
 $$f(n)=\sum_{d\mid \frac{q}{q_0}}g(d)\tau(n,\overline{\psi_{\frac{q}{d}}}),\quad n=1,2,\cdots,$$
 where $g=\mathcal{T}f$ is the Fourier transform of $f$,  $\psi_{\frac{q}{d}}$ is the Dirichlet character mod $\frac{q}{d}$ induced by  $\psi$.
\end{lem}
\begin{proof} By Lemma \ref{Fourier transform}, we have $g\in E_{q,\bar{\psi}}$, which together with Lemma \ref{Epsi lemma} in turn gives
$$
  f(n)=\sum_{m=1}^{q}g(m)e^{2\pi imn/q}
=\sum_{m=1}^{q}g(\widehat{m})\overline{\psi(\frac{m}{\widehat{m}})}e^{2\pi imn/q}.
$$
Note that
$\widehat{m}=d$ for some divisor $d$ of $\frac{q}{q_0}$ if and only if $m=ld$ with $\gcd(l,r_d)=1$, where $r_d=\frac{q}{dq_0}$.
It follows that
\begin{equation*}
  \begin{split}
     f(n) & =\sum_{d\mid \frac{q}{q_0}}g(d)\sum_{\substack{l=1  \\ (l,r_d)=1}}^{\frac{q}{d}}\overline{\psi(l)}
     e^{2\pi ildn/q} \\
       & =\sum_{d\mid \frac{q}{q_0}}g(d)\sum_{l=1}^{\frac{q}{d}}\overline{\psi(l)\chi_{r_d}(l)}
     e^{2\pi ildn/q} \\
       & =\sum_{d\mid \frac{q}{q_0}}g(d)\sum_{l=1}^{\frac{q}{d}}\overline{\psi_{\frac{q}{d}}(l)}
     e^{2\pi iln/\frac{q}{d}}\\
       & =\sum_{d\mid \frac{q}{q_0}}g(d)\tau(n,\overline{\psi_{\frac{q}{d}}}),
  \end{split}
\end{equation*} where $\chi_{r_d}$ is the principle Dirichlet character mod $r_d$.
The proof is complete.
\end{proof}

\noindent\textbf{Proof of Theorem \ref{specriterion}.}
  Let $d_0$ be the unique divisor of $\frac{q}{q_0}$ satisfying $g(d_0)\neq0$. Then by Lemma \ref{f=gtau},
  \begin{equation}\label{f(n)=}
    f(n)=g(d_0)\tau(n,\psi_{\frac{q}{d_0}}),\quad n=1,2,\cdots.
  \end{equation}
  This together with Lemma \ref{tau Epsi} yields
  $f\in E_{\frac{q}{d_0},\psi}$, and thus
  $(f\ast\mu\,\psi)(n)=0$ whenever $n\nmid \frac{q}{d_0q_0}$ by Lemma \ref{Epsi lemma}.
  It suffices to show that
  $$P(s)=\sum_{d\mid\frac{q}{d_0q_0}}(f\ast\mu\,\psi)(d)d^{-s}$$
  has no zeros in $\mathbb{C}_1$.
  Fix a divisor $k$ of $\frac{q}{d_0q_0}$. Then by Assumption (1), $$\gcd(k,\frac{q}{kq_0})=\gcd(k,q_0)=1$$ and $$\mu(k)^2=\psi(k)^2=1.$$
  Thus, $\gcd(k,\frac{q}{k})=1$,
  $\phi(k)\phi(\frac{q}{kd_0})=\phi(\frac{q}{d_0})$, and by Lemma \ref{tau lemma},
  $$\tau(k,\psi_{\frac{q}{d_0}})=\frac{\phi(\frac{q}{d_0})}{\phi(\frac{q}{kd_0})}
  \mu(\frac{q}{kd_0q_0})\psi(\frac{q}{kd_0q_0})\tau(\psi)
  =\phi(k)\mu(\frac{q}{kd_0q_0})\psi(\frac{q}{kd_0q_0})\tau(\psi).$$
  It follows that
 for any divisor $l$ of $\frac{q}{k}$, \begin{equation}\label{mucpsic}
\begin{split}
   \tau(k,\psi_{\frac{q}{d_0}})\mu(l)\psi(l) & =\phi(k)\tau(\psi)\mu(\frac{q}{kd_0q_0})\psi(\frac{q}{kd_0q_0})\mu(l)\psi(l) \\ & =\phi(k)\tau(\psi)\mu(\frac{q}{kd_0q_0})\psi(\frac{q}{kd_0q_0})\mu(l)\psi(l)\mu(k)^2\psi(k)^2 \\
     & =\phi(k)\tau(\psi)\mu(\frac{q}{d_0q_0})\psi(\frac{q}{d_0q_0})\mu(kl)\psi(kl).
\end{split}
  \end{equation}
  Therefore,
  \begin{equation*}
    \begin{split}
       (f\ast\mu\,\psi)(d) & =\sum_{k\mid d}f(k)\mu(\frac{d}{k})\psi(\frac{d}{k}) \\
         & =g(d_0)\sum_{k\mid d}\tau(k,\psi_{\frac{q}{d_0}})\mu(\frac{d}{k})\psi(\frac{d}{k}) \\
         & =C\mu(d)\psi(d)\sum_{k\mid d}\phi(k) \\
         & =Cd\mu(d)\psi(d),
    \end{split}
  \end{equation*}
  where $$C=g(d_0)\tau(\psi)\mu(\frac{q}{d_0q_0})\psi(\frac{q}{d_0q_0})\neq0$$ (see \cite[Theorem 8.15]{Apo}), the second identity follows from (\ref{f(n)=}), and the third identity follows from (\ref{mucpsic}).
 This gives \begin{equation*}
    \begin{split}P(s) & =C\sum_{d\mid\frac{q}{d_0q_0}}\mu(d)\psi(d)d^{1-s} \\ & =C\prod_{\substack{p\mid\frac{q}{d_0q_0} \\ p\,\text{prime}}}(\mu(p)\psi(p)+p^{1-s}) \\
    & =C\prod_{\substack{p\mid\frac{q}{d_0q_0} \\ p\,\text{prime}}}(-\psi(p)+p^{1-s}).\end{split}
  \end{equation*}
     It is clear that $P(s)\neq0$ whenever $\mathrm{Re}\, s>1$.
$\hfill \square $
\vskip2mm

 To continue, we need to introduce some more notations. For $q\in\mathbb{N}$, set $$M_q(\mathbf{z})=\sum_{d\mid q}\mu(\frac{q}{d})\mathbf{z}^{\alpha(d)}$$
and $$S_q=\{j\in\mathbb{N}:p_j\mid q\},$$ where $p_j$  is the $j$-th prime.
 It is clear  that $M_q$ is a polynomial only involving complex variables $\{z_j:j\in S_q\}$.
 Since $$\sum_{d\mid q}\mu(\frac{q}{d})d^{-s}=q^{-s}\sum_{d\mid q}\mu(\frac{q}{d})(\frac{q}{d})^s
 =q^{-s}\prod_{j\in S_q}(1-p_j^s)=q^{-s}\prod_{j\in S_q}\frac{p_j^{-s}-1}{p_j^{-s}},$$
 we have
 $$M_q=\mathcal{B}(\sum_{d\mid q}\mu(\frac{q}{d})d^{-s})=\mathcal{B}(q^{-s}\prod_{j\in S_q}\frac{p_j^{-s}-1}{p_j^{-s}}).$$
 It follows that
 if $q\ (q\geq2)$  has prime factorization $q=p_{j_1}^{k_1}\cdots p_{j_l}^{k_l}$,  then $S_q=\{j_1,\cdots,j_l\}$ and \begin{equation}\label{repofMq}
   M_q(z_{j_1},\cdots,z_{j_l})=z_{j_1}^{k_1-1}(z_{j_1}-1)\cdots z_{j_l}^{k_l-1}(z_{j_l}-1).
 \end{equation}
 From this, one concludes that $M_{q_1q_2}=M_{q_1}M_{q_2}$ whenever $\gcd(q_1,q_2)=1$.

 The space $E_{q,\chi_1}$ is abbreviated as $E_q$, where $\chi_1\equiv1$.

 Lastly,  for a nonempty subset $S$ of $\mathbb{N}$, define
 $$\mathbb{D}^S=\{\mathbf{z}\in\mathbb{D}_2^\infty: \text{ $\mathbf{z}_j=0$  for $j\notin S$}\}.$$
 Then $\mathbb{D}^S$ is identified with the $n$-polydisk in an obvious way, where $n=\sharp S$.

\begin{thm}\label{Eq}
Let $V,q$ and $g$ be given as in (\ref{Vdisplay}), (\ref{q}) and (\ref{g(m)}).
If $g\in E_{q}$,  then $\mathbf{1}_V\in\mathcal{C}$ if and only if the polynomial \begin{equation}\label{R}
  R=\sum_{d\mid q}g(\frac{q}{d})M_d
\end{equation} has no zero in $\mathbb{D}^{S_q}$.
\end{thm}

Ramanujan's sum $c_q$ is defined by
$$c_q(n):=\tau(n,\chi_q)=\sum_{\substack{m=1  \\ (m,q)=1}}^{q}e^{2\pi imn/q},$$ where $\chi_q$ is
the principle Dirichlet character mod $q$. The following formula will be used to prove Theorem \ref{Eq}.

\begin{lem}\label{mucq}
  For $q\geq1$ we have $$
        (\mu\ast c_q)(n)=\begin{cases}
                        n\mu(\frac{q}{n}), &  n\mid q; \\
                        0, & n\nmid q.
                      \end{cases}
       $$
\end{lem}
\begin{proof}
  Fix $n\in\mathbb{N}$. It follows from \cite[Theorem 8.6]{Apo} that for any divisor $d$ of $n$,
  $$c_q(d)=\sum_{k\mid(d,q)}k\mu(\frac{q}{k}).$$
  Hence, we have
   $$(\mu\ast c_q)(n)=\sum_{d\mid n}c_q(d)\mu(\frac{n}{d})=\sum_{d\mid n}\sum_{k\mid(d,q)}k\mu(\frac{q}{k})\mu(\frac{n}{d})=\sum_{k\mid(n,q)}k\mu(\frac{q}{k})\sum_{d\in A_k}\mu(\frac{n}{d}),$$
   where $A_k=\{d:k\mid d,\ d\mid n\}$. Note that for any divisor $k$ of $\gcd(n,q)$, $d\in A_k$ if and only if $d=ck$ for some integer $c$ with $c\mid \frac{n}{k}$. The sum $\sum_{d\in A_k}\mu(\frac{n}{d})$ in the above equality is equals to
   $$\sum_{c\mid\frac{n}{k}}\mu(\frac{n}{ck})=\begin{cases}
                        1, &  \frac{n}{k}=1; \\
                        0, & \frac{n}{k}>1.
                      \end{cases}$$
   If $n\nmid q$ then $\gcd(n,q)<n$, and thus $(\mu\ast c_q)(n)=0$.
   Now assume $n\mid q$. Then
   $$(\mu\ast c_q)(n)=\sum_{k\mid n}k\mu(\frac{q}{k})\sum_{d\in A_k}\mu(\frac{n}{d})= n\mu(\frac{q}{n}).$$ This completes the proof.
\end{proof}

\noindent\textbf{Proof of Theorem \ref{Eq}.} It suffices to show
 $$\sum_{d\mid q}\frac{(f\ast\mu)(d)}{d}\mathbf{z}^{\alpha(d)}
 =\mathcal{B}(\sum_{d\mid\frac{q}{q_0}}\frac{(f\ast\mu\,\psi)(d)}{d^s})=\sum_{d\mid q}g(\frac{q}{d})M_d.$$
 By Lemma \ref{f=gtau}, we have
\begin{equation*}
  \begin{split}
     (f\ast\mu)(d) & =\sum_{k\mid d}f(k)\mu(\frac{d}{k}) \\
       & = \sum_{k\mid d}\mu(\frac{d}{k})\sum_{c\mid q}g(c)\tau(k,\chi_{\frac{q}{c}})\\
       & = \sum_{c\mid q}g(c)\sum_{k\mid d}\mu(\frac{d}{k})c_{\frac{q}{c}}(k)\\
       & = \sum_{c\mid q}g(c)(\mu\ast c_{\frac{q}{c}})(d),
  \end{split}
\end{equation*} where $\chi_{\frac{q}{c}}$ is
the principle Dirichlet character mod $\frac{q}{c}$.
For a divisor $c$ of $q$, Lemma \ref{mucq} gives $$
        (\mu\ast c_{\frac{q}{c}})(d)=\begin{cases}
                        d\mu(\frac{q}{cd}), & c\mid\frac{q}{d}; \\
                        0, & c\nmid\frac{q}{d}.
                      \end{cases}$$
                      Therefore,
                      $$(f\ast\mu)(d)=\sum_{c\mid\frac{q}{d}}g(c)(\mu\ast c_{\frac{q}{c}})(d)
                      =d\sum_{c\mid\frac{q}{d}}g(c)\mu(\frac{q}{cd}),$$
                      which yields
\begin{equation*}
  \begin{split}
    \sum_{d\mid q}\frac{(f\ast\mu)(d)}{d}\mathbf{z}^{\alpha(d)}
       & =\sum_{d\mid q}(\sum_{c\mid\frac{q}{d}}g(c)\mu(\frac{q}{cd}))\mathbf{z}^{\alpha(d)} \\
       & =\sum_{c\mid q}g(c)\sum_{d\mid \frac{q}{c}}\mu(\frac{q}{cd})\mathbf{z}^{\alpha(d)} \\
       & = \sum_{c\mid q}g(c)M_{\frac{q}{c}}\\
       & =\sum_{d\mid q}g(\frac{q}{d})M_d.
  \end{split}
\end{equation*}
$\hfill \square $
\vskip2mm
%
%
%
\section{Determining the character $\psi$}

In this section, we will show that for most $V$, which is given as in (\ref{Vdisplay}) and satisfying $\mathbf{1}_V\in\mathcal{C}$, the Dirichlet character $\psi$ determined by $V$ (see Theorem \ref{precri} and Remark \ref{uniquepsi}) coincides with the Dirichlet character $\chi_1$ mod $1$, i.e., $\psi\equiv1$.

\begin{thm} \label{criterion}
  Let $V,q$ and $g$ be given as in (\ref{Vdisplay}),  (\ref{q}) and (\ref{g(m)}).
   Assume that one of the following two conditions is satisfied:
   \begin{itemize}
     \item [(1)] $t_0=1,2$ or $t_0=p^k$ for some $k\in\mathbb{N}$ and some prime $p$ with $p\equiv 3\,(\mathrm{mod}\,4)$, where $t_0=\gcd(t_1, t_1', \cdots, t_N, t_N')$.
     \item [(2)] $t_V\geq7$ and $3\nmid t_V$.
   \end{itemize}
   If $\mathbf{1}_V\in\mathcal{C}$ then
    $g\in E_{q}$, i.e., $g(m_1)=g(m_2)$ whenever $\gcd(m_1,q)=\gcd(m_2,q)$.

    As a consequence, $\mathbf{1}_V\in\mathcal{C}$ if and only if $g\in E_{q}$ and the polynomial $R$ in  (\ref{R})  has no zero in $\mathbb{D}^{S_q}$.
\end{thm}

\subsection{Proof of Theorem \ref{criterion}}

The key ingredient to the proof of Theorem \ref{criterion} is the following result that only involves arithmetical functions in the space $E_{q,\psi}$.

\begin{lem}\label{keylem}
Suppose that $q$ is an even positive integer with $q\geq14$ and $3\nmid q$, and that $\psi$ is an even Dirichlet character in $\mathcal{P}_q$.
 Assume that  there is a function $g\in E_{q,\psi}$, such that $$\gcd\left(\{m\in\mathbb{Z}:g(m)\neq0\}\cup\{\frac{q}{2}\}\right)=1,$$ and that the sum function $S(m)=\sum_{i=1}^{m}g(i)$ of $g$  only takes values $0$ or $1$ for $1\leq m\leq\frac{q}{2}-1$.
Then  $\psi\equiv1$.
\end{lem}

Due to the length of the proof of Lemma \ref{keylem}, we first use Lemma \ref{keylem} and the following four lemmas to prove Theorem \ref{criterion}, and then prove Lemma \ref{keylem} in Subsection 4.3.

\begin{lem}\cite[Proposition 2.1.34]{Co} \label{q1q2}
  Suppose that $q$ is the product of two relatively prime positive integers $q_1, q_2$, and $\psi$ is a primitive Dirichlet character mod $q$. Then $\psi=\psi_1\psi_2$, where $\psi_i\ (i=1,2)$  is a primitive Dirichlet characters mod $q_i$.
\end{lem}

Remark that the decomposition $\psi=\psi_1\psi_2$ in Lemma \ref{q1q2} is unique due to the Chinese remainder theorem.

Lemma \ref{realpc} below comes from standard analytic number theory, see
\cite[pp. 37-38]{Da2} for instance.

\begin{lem}\label{realpc}
\begin{itemize}
  \item [(1)] There are only three real primitive Dirichlet characters mod  $2^k\  (k\in\mathbb{N})$: the Dirichlet character $\psi_4^{\mathrm{pr}}$
   mod $4$ determined by $\psi_4^{\mathrm{pr}}(3)=-1$; the Dirichlet character $\psi_{8,1}^{\mathrm{pr}}$
   mod $8$ determined by $\psi_{8,1}^{\mathrm{pr}}(3)=1$ and $\psi_{8,1}^{\mathrm{pr}}(5)=\psi_{8,1}^{\mathrm{pr}}(7)=-1$; the Dirichlet character $\psi_{8,2}^{\mathrm{pr}}$
   mod $8$ determined by $\psi_{8,2}^{\mathrm{pr}}(3)=\psi_{8,2}^{\mathrm{pr}}(5)=-1$ and $\psi_{8,2}^{\mathrm{pr}}(7)=1$.
  \item [(2)] For an odd prime $p$,  the only real primitive Dirichlet character
   mod $p$ is the Legendre symbol $(n\mid p)$, while if $k\geq2$ there are no real primitive Dirichlet characters mod $p^k$.
\end{itemize}
\end{lem}

\begin{rem} \label{2ks}
In light of Lemmas \ref{q1q2} and \ref{realpc}, for a real primitive Dirichlet character $\psi$,
its modulus $q_0$ has form
$$q_0=2^ks,$$ where $k=0, 2$ or $3$, and $s=1$ or $s=p_1\cdots p_l$, a product of some distinct primes. Write $\psi(n)=\psi_1(n)\psi_2(n)$,
where $\psi_1\equiv1$ or $\psi_1=\psi_4^{\mathrm{pr}}, \psi_{8,1}^{\mathrm{pr}}, \psi_{8,2}^{\mathrm{pr}}$, and $\psi_2\equiv1$ or
$$\psi_2(n)=(n\mid s):=\prod_{i=1}^{l}(n\mid p_i).$$
Since $\psi_4^{\mathrm{pr}}(-1)=\psi_{8,1}^{\mathrm{pr}}(-1)=-1$, $\psi_{8,2}^{\mathrm{pr}}(-1)=1$,
and $\psi_2(-1)=\pm1$ if and only if $s\equiv \pm1\,(\mathrm{mod}\,4)$ (see \cite[Theorem 9.10]{Apo}),
we conclude that when $\psi$ is even (i.e., $\psi(-1)=1$),
\begin{itemize}
  \item [(1)] $s\equiv 1\,(\mathrm{mod}\,4)$ if $k=0$ or $k=3$, $\psi_1=\psi_{8,2}^{\mathrm{pr}}$;
  \item [(2)] $s\equiv 3\,(\mathrm{mod}\,4)$ if $k=2$ or $k=3$, $\psi_1=\psi_{8,1}^{\mathrm{pr}}$.
\end{itemize}
\end{rem}

\begin{lem}\label{gatb}
Suppose $q=2t\ (t\in\mathbb{N})$, $g\in E_{q,\psi}$ for some $\psi\in\mathcal{P}_q$, and $\frac{a}{b}$ is an irreducible fraction in $(0,1)$ with $b\mid t$. Put $\tilde{a}=\gcd(a,2)$ and $\tilde{b}=\gcd(b,2)$. Then we have
\begin{itemize}
  \item [(1)] $q_0\mid b\tilde{b}$ provided $g(\frac{at}{b})\neq0$, where $q_0$ is the modulus of $\psi$.
  \item [(2)]
$g(\frac{a t}{b})=g(\frac{\tilde{a}t}{b})\psi(\frac{a}{\tilde{a}})$ and
$|g(\frac{a t}{b})|=|g(\frac{\tilde{a}t}{b})|$.
\end{itemize}
\end{lem}
\begin{proof}
A direct calculation gives
$$
  \gcd(\frac{at}{b},\frac{q}{q_0})
=\frac{t}{bq_0}\gcd(aq_0,2b)
=\frac{\tilde{a} t}{bq_0}\gcd(\frac{a}{\tilde{a}}q_0,\frac{2}{\tilde{a}}b)
=\frac{\tilde{a} t}{bq_0}\gcd(q_0,\frac{2}{\tilde{a}}b).
$$
Since   $g\in E_{q,\psi}$,
\begin{equation}\label{s0t0}
 g(\frac{at}{b})
=g(\gcd(\frac{at}{b},\frac{q}{q_0}))\psi(c),
\end{equation}
where
$$c=\frac{\frac{at}{b}}{\gcd(\frac{at}{b},\frac{q}{q_0})}
=\frac{aq_0}{\tilde{a}\cdot\gcd(q_0,\frac{2 b}{\tilde{a}})}.$$

To show (1), assume $g(\frac{at}{b})\neq0$.
By (\ref{s0t0}), we have
 $\psi(c)\neq0$, equivalently, $\gcd(c,q_0)=1$. This yields $q_0=\gcd(q_0,\frac{2 b}{\tilde{a}})$, forcing $q_0\mid \frac{2 b}{\tilde{a}}$. Since there is no primitive Dirichlet character mod $2$,
Lemma \ref{q1q2}  further implies
 $q_0\mid b$ if  $b$ is odd. Therefore, we have proved (1).

By (1), if $q_0\nmid b\tilde{b}$ then $g(\frac{at}{b})=g(\frac{\tilde{a}t}{b})=0$.
So we only need  to prove (2) for the case that
$q_0\mid b\tilde{b}$. In this case, we have $q_0\mid \frac{2 b}{\tilde{a}}$, and thus $c=\frac{a}{\tilde{a}}$. Moreover,
$$\gcd(\frac{at}{b},\frac{q}{q_0})
=\frac{at}{cb}
=\frac{\tilde{a} t}{b}.$$
Hence by (\ref{s0t0}),
$$g(\frac{at}{b})
=g(\frac{\tilde{a} t}{b})\psi(\frac{a}{\tilde{a}}).$$
 It remains to  show $\gcd(a,q_0)=1$, which would imply $|\psi(\frac{a}{\tilde{a}})|=1$.
When $b$ is odd,  $q_0\mid b$; when $b$ is even,  $q_0\mid 2b$ and $a$ is necessarily odd since $\gcd(a,b)=1$. In either case, we have proved $\gcd(a,q_0)=1$. The proof is complete.
\end{proof}

\begin{lem}\label{sumfunclem}
  Let $V$ and $g$ be given as in (\ref{Vdisplay})   and (\ref{g(m)}). Then
  the sum function $S(m)=\sum_{i=1}^{m}g(i)$ of $g$ only takes values $-\mathbf{1}_V(0^+)$ or $1-\mathbf{1}_V(0^+)$ for $1\leq m\leq t_V-1$.
\end{lem}
\begin{proof}
  For $1\leq m\leq t_V-1$,
$$S(m)=\sum_{i=1}^{m}J_V(\frac{i}{t_V})
=\mathbf{1}_V((\frac{m}{t_V})^+)-\mathbf{1}_V(0^+).$$
This proves the lemma.
\end{proof}

\vskip2mm
\noindent\textbf{Proof of Theorem \ref{criterion}.}
Assume $\mathbf{1}_V\in\mathcal{C}$. Then by Corollary \ref{fg}, $g\in E_{q,\psi}$ for some $\psi\in\mathcal{P}_q$. Since $g$ is an even function by its definition,
it follows that $\psi$ is also even, i.e., $\psi(-1)=1$.

Assume (1) holds and list all of the different numbers appearing in $\{t_i\}_{i=1}^{N}\cup\{t_i'\}_{i=1}^{N}$: $b_1$, $\cdots$, $b_n\ (n\in\mathbb{N})$.
Then there exist $n$ integers $a_1$, $\cdots$, $a_n$, such that for each $1\leq i\leq n$, $\gcd(a_i,b_i)=1$ and $g(\frac{a_it_V}{b_i})=J_V(\frac{a_i}{b_i})\neq0$. By Lemma \ref{gatb} (1),
$q_0\mid2b_i\ (1\leq i\leq n)$, which yields $q_0\mid2t_0$ since $t_0=\gcd(b_1, \cdots, b_n)$.
Thus, the characterization of $q_0$ in Remark \ref{2ks} forces $q_0=1$.

Now assume (2) holds. As mentioned in Introduction (see (\ref{Vdisplay'})), one can rewrite $V$ as
$$
   V=(\frac{\alpha_1}{t_V},\frac{\beta_1}{t_V})\cup\cdots\cup(\frac{\alpha_N}{t_V},\frac{\beta_N}{t_V}),
$$ where  $\alpha_1,\beta_1,\cdots,\alpha_N,\beta_N$ are integers with $
   \gcd(\alpha_1,\beta_1,\cdots,\alpha_N,\beta_N,t_V)=1.$
   Since $t_V=\frac{q}{2}$ and $g(\alpha_i), g(\beta_i)\neq0\ (1\leq i\leq N)$, it follows that the greatest common divisor of $$\{m\in\mathbb{Z}:g(m)\neq0\}\cup\{\frac{q}{2}\}$$ is $1$.
Hence, Lemma \ref{sumfunclem} implies that  $(1-2\mathbf{1}_V(0^+))g$ satisfies all hypothesis in  Lemma \ref{keylem}. This completes the proof.
$\hfill \square $
\vskip2mm


\subsection{Preparatory lemmas}
In this subsection, we establish some preparatory lemmas for the proof of Lemma \ref{keylem}.


\begin{lem} \label{pclem}
\begin{itemize}
  \item [(1)] For every prime $p$ with $p\geq7$, there exist two integers $a, b$ with $2\leq a, b\leq p-2$, such that $(1-a^2\mid p)=1$ and $(1-b^2\mid p)=-1$.
  \item [(2)] For $a, b\in\mathbb{Z}$ with $5\nmid ab$, there exists
$k\in\{1, 2\}$ such that $((k a)^2-b^2\mid 5)=-1$.
\end{itemize}
\end{lem}
\begin{proof}
(1) It suffices to find two integers $\tilde{a}, \tilde{b}$ with $2\leq \tilde{a}, \tilde{b}\leq p-2$, such that $$(\tilde{a}^2-1\mid p)=1,\quad(\tilde{b}^2-1\mid p)=-1.$$

If $(2\mid p)=1$ and $(3\mid p)=-1$, then $$(3^2-1\mid p)=(8\mid p)=(2\mid p)=1,$$
$$(2^2-1\mid p)=(3\mid p)=-1.$$

If $(2\mid p)=-1$ and $(3\mid p)=1$, then $$(2^2-1\mid p)=(3\mid p)=1,$$ $$(3^2-1\mid p)=(2\mid p)=-1.$$

If $(2\mid p)=(3\mid p)=-1$, then $$(5^2-1\mid p)=(24\mid p)=(6\mid p)=1,$$ $$(2^2-1\mid p)=(3\mid p)=-1.$$

Now suppose $(2\mid p)=(3\mid p)=1$ and let $n$ be the smallest positive integer such that  $(n\mid p)=-1$. Then $n\geq4$, and
$$((n-2)^2-1\mid p)=((n-1)(n-3)\mid p)=(n-1\mid p)(n-3\mid p)=1,$$
$$((n-1)^2-1\mid p)=(n(n-2)\mid p)=(n\mid p)(n-2\mid p)=-1.$$ This proves (1).

(2) Note that $a^2\equiv \pm1\,(\mathrm{mod}\,5)$, $b^2\equiv \pm1\,(\mathrm{mod}\,5)$. Then one can finish the proof
by taking $k=1$ if $a^2\equiv -b^2\,(\mathrm{mod}\,5)$ or $k=2$ if $a^2\equiv b^2\,(\mathrm{mod}\,5)$.
\end{proof}

\begin{lem} \label{psi1lem}
  Suppose that the Dirichlet character $\psi\in\mathcal{P}_q$ with modulus $q_0$ is even, and  $n\in\mathbb{N}$ with $n\mid\frac{q}{q_0}$.
  \begin{itemize}
    \item [(1)] Let $q_0=q_1q_2$ with $\gcd(q_1,q_2)=1$, and $\psi=\psi_1\psi_2$ be the decomposition of $\psi$ according to Lemma \ref{q1q2}. Then for any $d\mid q_1$ and any $k\in\mathbb{N}$, we have $$\psi((\frac{k q}{dn})^2-1)=\psi_1(1-(\frac{k q}{dn})^2).$$
    \item [(2)] If $p$ is an odd prime divisor of  $q_0$  and $p^2\mid \frac{q}{n}$, then $\psi((\frac{k q}{pn})^2-1)=1$ for every $k\in\mathbb{N}$.
    \item [(3)] If $q_0$ has a prime divisor $p\ (p\geq7)$ and $p^2\nmid \frac{q}{n}$, then there exist two integers $k, l$ with
    $1\leq k, l\leq\frac{p-1}{2}$, such that $\psi((\frac{k q}{pn})^2-1)=1$ and $\psi((\frac{l q}{pn})^2-1)=-1$.
  \end{itemize}
\end{lem}
\begin{proof}
(1) Given  $d\mid q_1$ and  $k\in\mathbb{N}$, since $d\mid q_1$ and $n\mid\frac{q}{q_0}$ one has $dn\mid\frac{q}{q_2}$, i.e., $q_2\mid\frac{q}{dn}$.
This gives $$\psi_2((\frac{k q}{pn})^2-1)=\psi_2(-1)=\frac{\psi(-1)}{\psi_1(-1)}
=\frac{1}{\psi_1(-1)}=\psi_1(-1),$$
and thus
\begin{equation*}
  \begin{split}
      \psi((\frac{k q}{dn})^2-1)  & =\psi_1((\frac{k q}{dn})^2-1)\psi_2((\frac{k q}{dn})^2-1) \\
       & =\psi_1((\frac{k q}{dn})^2-1)\psi_1(-1) \\
       & =\psi_1(1-(\frac{k q}{dn})^2).
  \end{split}
\end{equation*}

(2) Since $\gcd(p,\frac{q_0}{p})=1$ (see Remark \ref{2ks}), by (1) we have
$$\psi((\frac{k q}{pn})^2-1)=(1-(\frac{k q}{pn})^2\mid p),\quad k\in\mathbb{N}.$$
This proves (2) since $p\mid \frac{q}{pn}$.

(3) We need the fact that if $p\nmid m$ then there exists an integer $k$ with
    $1\leq k\leq\frac{p-1}{2}$, such that
    $m\equiv k\,(\mathrm{mod}\,p)$ or $$m\equiv p-k\equiv -k\,(\mathrm{mod}\,p),$$ and thus
    $m^2\equiv k^2\,(\mathrm{mod}\,p)$.

    It follows from Lemma \ref{pclem} (1) that there exists an integer $a$ with $2\leq a\leq p-2$, such that $(1-a^2\mid p)=1$.
    Since $p\nmid\frac{q}{pn}$, one can choose $c\in\mathbb{N}$ satisfying $\frac{cq}{pn}\equiv 1\,(\mathrm{mod}\,p)$. By the above fact, there exists an integers $k$ with
    $1\leq k\leq\frac{p-1}{2}$, such that
    $k^2\equiv (ac)^2\,(\mathrm{mod}\,p)$,
    and therefore $$(\frac{k q}{pn})^2
    \equiv(\frac{acq}{pn})^2\equiv a^2(\frac{cq}{pn})^2\equiv a^2\,(\mathrm{mod}\,p).$$
    Hence, similar to the proof of (2), we have
    $$\psi((\frac{k q}{pn})^2-1)=(1-(\frac{k q}{pn})^2\mid p)=(1-a^2\mid p)=1.$$

    Similarly, one can  find an integer $l$ with
    $1\leq l\leq\frac{p-1}{2}$, such that $\psi((\frac{l q}{pn})^2-1)=-1$.
\end{proof}

\begin{lem} \label{abclem}
\begin{itemize}
  \item [(1)] Suppose $a, b, c\in\{-1,0,1\}$, $|a+b|\leq1$, $|b+c|\leq1$ and $|a+b+c|\leq1$. Then $b=-a=-c$ if $ac=1$, while $b=0$ if $ac=-1$.
  \item [(2)] Suppose $n\in\mathbb{N}$, $a_j\in\{-1,1\}\ (j=1,2,\cdots,n)$, and $|a_j+a_{j+1}|\leq1$ for every $1\leq j\leq n-1$. Then $a_j=(-1)^{j-1}a_1$.
\end{itemize}
\end{lem}

One can check Lemma \ref{abclem} directly.

\begin{lem}\label{Conclusion 1} Suppose that the sum function $S(m)=\sum_{i=1}^{m}g(i)$ of an arithmetic function $g$ only takes values $0$ or $1$ for $1\leq m\leq M\ (M\in\mathbb{N})$.
 For $1\leq m_1<m_2\leq M$, $$\sum_{i=m_1+1}^{m_2-1}g(i)=-g(m_1)=-g(m_2)$$ if $g(m_1)g(m_2)=1$, while $\sum_{i=m_1+1}^{m_2-1}g(i)=0$ if $g(m_1)g(m_2)=-1$.
\end{lem}
\begin{proof}
 Set
$$\alpha=\sum_{i=m_1+1}^{m_2-1}g(i)=S(m_2-1)-S(m_1).$$ Then $g(m_1), \alpha, g(m_2)\in\{-1,0,1\}$,
$$|g(m_1)+\alpha|=|S(m_2-1)-S(m_1-1)|\leq1,$$
$$|\alpha+g(m_2)|=|S(m_2)-S(m_1)|\leq1,$$
and
$$|g(m_1)+\alpha+g(m_2)|=|S(m_2)-S(m_1-1)|\leq1,$$
here we define $S(0)=0$.
 Thus Lemma \ref{abclem} (1), applied to $g(m_1), \alpha$ and $g(m_2)$, finishes the proof.
\end{proof}

\subsection{Proof of Lemma \ref{keylem}}
This subsection is dedicated to the proof of Lemma \ref{keylem}.

It is equivalent to prove that the modulus $q_0$ of $\psi$ is $1$.
To reach a contradiction,  assume conversely $q_0>1$. The proof can be roughly divided into three steps. In the first step, we establish several lemmas as preparation (Lemmas 4.13-4.18). In the second step, we show $q_0=5$ (Lemma \ref{q0lem}). The proof of Lemma \ref{keylem} will be completed in the last step.



Since $q_0>1$, by Lemma \ref{gatb} (1), $g(q)=g(\frac{q}{2})=0$ and the following conclusion holds.
\begin{lem}\label{gq4}  If $4\mid q$ then $g(\frac{q}{4})=0$.
\end{lem}

%
%
%
Set $\widehat{m}=\gcd(m,\frac{q}{q_0})\ (m\in\mathbb{Z})$.
Since for $1\leq m\leq\frac{q}{2}-1$, $S(m)$ only takes values $0$ or $1$, and $$g(q-m)=g(\widehat{q-m})\psi(\frac{q-m}{\widehat{q-m}})
=g(\widehat{m})\psi(\frac{q}{\widehat{m}}-\frac{m}{\widehat{m}})=g(\widehat{m})\psi(\frac{m}{\widehat{m}})\psi(-1)=g(m),$$
 it follows that $g$ only takes values $0$ or $\pm1$. In particular, Lemma \ref{q0qd}
implies that $\psi$ is real-valued.

%
%

 Take $m_0$ to be the smallest positive integer among those $m$'s satisfying $g(m)\neq0$. Then $g(m_0)=S(m_0)=1$. From the identity $$g(m_0)=g(\widehat{m_0})\psi(\frac{m_0}{\widehat{m_0}})$$ we deduce that
 $m_0=\widehat{m_0}$, i.e., $m_0\mid\frac{q}{q_0}$.

Let us first assume that the following inequlity holds, and put its proof at the end of this subsection.
\begin{lem}\label{m0q5q0}
 We have $m_0\leq\frac{q}{5q_0}$.
\end{lem}

 We also need
the following identity.
 Suppose $d\mid q_0$ with $d\geq4$,  $k\in\mathbb{N}$ with $k<\frac{d}{2}$, and $m\in\mathbb{Z}$. Then we have
\begin{equation} \label{lamdqd}
 g(\frac{k q}{d}+m)=g(\widehat{m})\psi(\frac{k q}{d\widehat{m}}+\frac{m}{\widehat{m}}).
\end{equation}
Indeed, since $\frac{q}{q_0}\mid\frac{q}{d}$,
 $\gcd(\frac{k q}{d}+m,\frac{q}{q_0})
=\gcd(m,\frac{q}{q_0})=\widehat{m}$, which immediately gives (\ref{lamdqd}). As a consequence,
\begin{equation}\label{glambdaqd-+m0}
  g(\frac{k q}{d}-m_0)g(\frac{k q}{d}+m_0)
=g(m_0)\psi(\frac{k q}{dm_0}-1)g(m_0)\psi(\frac{k q}{dm_0}+1)
=\psi((\frac{k q}{dm_0})^2-1),
\end{equation} and
$g(\frac{k q}{d}+m)=0$ whenever $1\leq|m|<m_0$ (in the case $m_0>1$) since
$\widehat{m}\leq|m|<m_0$, which yields
\begin{equation}\label{sumglambdaqd+m}
 \sum_{m=-m_0+1}^{m_0-1}g(\frac{k q}{d}+m)=g(\frac{k q}{d})
=g(\frac{q}{q_0})\psi(\frac{k q_0}{d}).
\end{equation}
 Note that   $$1\leq\frac{k q}{d}-m_0<\frac{k q}{d}+m_0\leq\frac{q}{2}-1$$
(since $m_0\leq\frac{q}{5q_0}\leq\frac{q}{5d}$ by Lemma \ref{m0q5q0}), and $\psi(\frac{k q_0}{d})\neq0$ only when $$\frac{q_0}{d}\cdot\gcd(k,d)=\gcd(\frac{k q_0}{d},q_0)=1.$$ With (\ref{glambdaqd-+m0}) and (\ref{sumglambdaqd+m}), invoking Lemma \ref{Conclusion 1} we  obtain the following three conclusions.

\begin{lem}\label{Conclusion 2}
  If $d\geq4$ is a proper divisor of $q_0$, then
$\psi((\frac{k q}{dm_0})^2-1)\neq1$ for $1\leq k<\frac{d}{2}$.
\end{lem}

\begin{lem}\label{Conclusion 3}
  If $\psi((\frac{k q}{q_0m_0})^2-1)=1$ for some $1\leq k<\frac{q_0}{2}$, then $$g(\frac{q}{q_0})\psi(k)=-\psi(\frac{k q}{q_0m_0}+1),$$ equivalently, $$g(\frac{q}{q_0})=-\psi(k)\psi(\frac{k q}{q_0m_0}+1).$$
\end{lem}

\begin{lem}\label{Conclusion 4}
  If $\psi((\frac{l q}{q_0m_0})^2-1)=-1$ for some $1\leq l<\frac{q_0}{2}$, then $g(\frac{q}{q_0})\psi(l)=0$.
\end{lem}

Consider the case that $q_0$ is an odd prime. If  $q_0^2\mid\frac{q}{m_0}$, then $\psi(\frac{k q}{q_0m_0}+1)=\psi(1)=1$ and $\psi(\frac{k q}{q_0m_0}-1)=\psi(-1)=1$ for every $1\leq k\leq\frac{q_0-1}{2}$ since $q_0\mid\frac{q}{q_0m_0}$ and $\psi$ is even.
By Lemma \ref{Conclusion 3},
$$\psi(k)=\psi(k)\psi(\frac{k q}{q_0m_0}+1)=-g(\frac{q}{q_0})=\psi(1)\psi(\frac{ q}{q_0m_0}+1)=1$$
holds for  $2\leq k\leq\frac{q_0-1}{2}$.
This is impossible because it would imply that  $\psi$ does not take value $-1$ (note that $\psi(q_0-k)=\psi(-k)=\psi(k)$).
Hence we further conclude that

\begin{lem}\label{Conclusion 5}
  If $q_0$ is an odd prime, then  $q_0^2\nmid\frac{q}{m_0}$.
\end{lem}

We are ready to prove the following conclusion.
\begin{lem}\label{q0lem}
  We have $q_0=5$ and  $\psi(m)=(m\mid 5)$.
\end{lem}
 \begin{proof}
 Note that $3\nmid q_0$. Then by Remark \ref{2ks}, it suffices to show that
 \begin{itemize}
   \item [(1)] $q_0$ cannot be divided by a prime $p$ with $p\geq7$;
   \item [(2)] $q_0$ is odd.
 \end{itemize}

(1) Assume that $q_0$ has a prime divisor $p$ with $p\geq7$.  Combining Lemma \ref{psi1lem} (2) with (3), we see that there exists $k\in\mathbb{N}$ with $k\leq\frac{p-1}{2}$ such that $\psi((\frac{k q}{pm_0})^2-1)=1$. Lemma \ref{Conclusion 2} shows that $p$ cannot be a proper divisor of $q_0$, i.e., $q_0=p$,
and then by  Lemma \ref{Conclusion 5}, $q_0^2\nmid\frac{q}{m_0}$. Moreover, it follows from Lemma \ref{Conclusion 3} that $g(\frac{q}{q_0})\neq0$. Again by Lemma \ref{psi1lem} (3), there exists $l\in\mathbb{N}$ with $l\leq\frac{q_0-1}{2}$ such that $\psi((\frac{l q}{q_0m_0})^2-1)=-1$, and thus $g(\frac{q}{q_0})\psi(l)=0$ by Lemma \ref{Conclusion 4}. This yields $\psi(l)=0$,  contradicting with the fact $\gcd(l,q_0)=1$.

\vskip2mm

(2) Assume that $q_0$ is even. Then we can write $q_0=q_1q_2$, where $q_1=4$ or $8$ and $q_2=1$ or $5$.
By Lemma \ref{q1q2}, for each $i\in\{1,2\}$ one can find a primitive Dirichlet characters $\psi_i$ mod $q_i$, such that $\psi=\psi_1\psi_2$.  Noting that $\psi_2(-1)=1$ whether $q_2$ equals   $1$ or $5$, we necessarily have $\psi_1(-1)=1$ since $\psi(-1)=1$, and then $q_1=8$, $\psi_1=\psi_{8,2}^{\mathrm{pr}}$.
By combining Lemma \ref{psi1lem} (1) with Lemma \ref{Conclusion 2},
we have $$\psi_{8,2}^{\mathrm{pr}}(1-(\frac{q}{4m_0})^2)
=\psi((\frac{q}{4m_0})^2-1)\neq1,$$ which yields
$16\nmid\frac{q}{m_0}$ (otherwise, $4\mid\frac{q}{4m_0}$ and then $\psi_{8,2}^{\mathrm{pr}}(1-(\frac{q}{4m_0})^2)=1$).
That is to say, $\frac{q}{8m_0}$ is odd, and hence so is $\frac{q}{q_0m_0}$.
Thus, for any $k\in\mathbb{N}$,  $$\widehat{2^km_0}=\gcd(2^km_0,\frac{q}{q_0})
=m_0\cdot\gcd(2^k,\frac{q}{q_0m_0})=m_0.$$ Note that $q_2\mid\frac{q}{8d}$ for every divisor $d$ of $\frac{q}{q_0}$ since $q_2=\frac{q_0}{8}$. Then
\begin{equation}\label{psi2}
  \psi_2(\frac{q}{8\widehat{m}}+\frac{m}{\widehat{m}})
=\psi_2(\frac{m}{\widehat{m}})=\psi_2(\frac{|m|}{\widehat{m}}),\quad m\in\mathbb{Z}.
\end{equation}
 Combining (\ref{lamdqd}) with (\ref{psi2}) one obtains that for any $m\in\mathbb{Z}$,
\begin{equation} \label{q8m}
  \begin{split}
      g(\frac{q}{8}+m)  & =g(\widehat{m})
\psi(\frac{q}{8\widehat{m}}+\frac{m}{\widehat{m}}) \\
       & =g(\widehat{m})\psi_{8,2}^{\mathrm{pr}}(\frac{q}{8\widehat{m}}
+\frac{m}{\widehat{m}})\psi_2(\frac{q}{8\widehat{m}}+\frac{m}{\widehat{m}}) \\
       & =g(\widehat{m})\psi_{8,2}^{\mathrm{pr}}(\frac{q}{8\widehat{m}}
+\frac{m}{\widehat{m}})\psi_2(\frac{|m|}{\widehat{m}}).
  \end{split}
\end{equation}
In particular,
\begin{equation} \label{q84m0}
  \begin{split}
      g(\frac{q}{8}-4m_0)g(\frac{q}{8}+4m_0)  & =g(m_0)\psi_{8,2}^{\mathrm{pr}}(\frac{q}{8m_0}-4)\psi_2(4)
      g(m_0)\psi_{8,2}^{\mathrm{pr}}(\frac{q}{8m_0}+4)\psi_2(4) \\
       & =\psi_{8,2}^{\mathrm{pr}}((\frac{q}{8m_0})^2-16) \\
       & =\psi_{8,2}^{\mathrm{pr}}((\frac{q}{8m_0})^2) \\
       & =1,
  \end{split}
\end{equation} and since $n^2\equiv 1\,(\mathrm{mod}\,8)$ for every odd integer $n$,
\begin{equation} \label{q82m0}
  \begin{split}
      g(\frac{q}{8}-2m_0)g(\frac{q}{8}+2m_0)  & =g(m_0)\psi_{8,2}^{\mathrm{pr}}(\frac{q}{8m_0}-2)\psi_2(2)
      g(m_0)\psi_{8,2}^{\mathrm{pr}}(\frac{q}{8m_0}+2)\psi_2(2) \\
       & =\psi_{8,2}^{\mathrm{pr}}((\frac{q}{8m_0})^2-4) \\
       & =\psi_{8,2}^{\mathrm{pr}}(-3) \\
       & =-1.
  \end{split}
\end{equation}
By Lemma \ref{Conclusion 1},
\begin{equation} \label{2m0}
\sum_{m=-2m_0+1}^{2m_0-1}g(\frac{q}{8}+m)=0.
\end{equation}
Moreover, it follows from (\ref{q8m}) that $g(\frac{q}{8}+m)\neq0$ if and only if $g(\widehat{m})\neq0$, $\frac{q}{8\widehat{m}}
+\frac{m}{\widehat{m}}$ is odd, and $\gcd(\frac{|m|}{\widehat{m}},q_2)=1$.
Therefore for $m\in\mathbb{N}$, $g(\frac{q}{8}+m)$ and $g(\frac{q}{8}-m)$ are either simultaneously zero or simultaneously nonzero since
$\frac{q}{8\widehat{m}}+\frac{m}{\widehat{m}}$ and $\frac{q}{8\widehat{m}}-\frac{m}{\widehat{m}}$
are simultaneously odd or even. Put
$$m_1=\min\{2m_0\leq m\leq 4m_0:g(\frac{q}{8}-m)g(\frac{q}{8}+m)=1\},$$
where the above set is nonempty since it contains $4m_0$ by (\ref{q84m0}).
Then $m_1>2m_0$ by (\ref{q82m0}) and
$$g(\frac{q}{8}-m)+g(\frac{q}{8}+m)=0,\quad 2m_0\leq m<m_1.$$ It follows from
(\ref{2m0}) that
$$\sum_{m=-m_1+1}^{m_1-1}g(\frac{q}{8}+m)=\sum_{m=-2m_0+1}^{2m_0-1}g(\frac{q}{8}+m)=0.$$
However, since $g(\frac{q}{8}-m_1)g(\frac{q}{8}+m_1)=1$ we should have
$$\sum_{m=-m_1+1}^{m_1-1}g(\frac{q}{8}+m)\neq0$$ by Lemma \ref{Conclusion 1}.
This shows that $q_0$ cannot be even.
\end{proof}

We proceed to complete the proof of Lemma \ref{keylem}.
It follows from Lemmas \ref{m0q5q0} and  \ref{q0lem} that $\frac{q}{5m_0}$ is an integer not less than $5$.
We also have $3\nmid\frac{q}{5m_0}$ since $3\nmid q$, and $5\nmid\frac{q}{5m_0}$ by Lemma \ref{Conclusion 5}, forcing $\frac{q}{5m_0}\in\{7, 8\}$ or $\frac{q}{5m_0}\geq11$.
Moreover, \begin{equation}\label{hat5m0}
 \widehat{5m_0}=\gcd(5m_0,\frac{q}{5})
=m_0\cdot\gcd(5,\frac{q}{5m_0})=m_0.
\end{equation}

By (\ref{lamdqd}) we have
\begin{equation} \label{q5m}
g(\frac{k q}{5}+m)=g(\widehat{m})(\frac{k q}{5\widehat{m}}+\frac{m}{\widehat{m}}\mid 5),\quad m\in\mathbb{Z},\ k=1,2.
\end{equation}
In particular, by (\ref{hat5m0}), for $k\in\{1, 2\}$,
\begin{equation} \label{q55m0}
  \begin{split}
      g(\frac{k q}{5}-5m_0)g(\frac{k q}{5}+5m_0)  & =g(m_0)(\frac{k q}{5m_0}-5\mid 5)g(m_0)(\frac{k q}{5m_0}+5\mid 5) \\
       & =((\frac{k q}{5m_0})^2\mid 5), \\
       & =1.
  \end{split}
\end{equation}
\noindent\textbf{Claim.}  $25\mid q$ and when $\frac{q}{5m_0}=7$ or $8$, one has $g(\frac{q}{25})\neq0$
and hence for $k\in\{1, 2\}$,
\begin{equation} \label{q5q25}
  \begin{split}
      g(\frac{k q}{5}-\frac{q}{25})g(\frac{k q}{5}+\frac{q}{25})  & =g(\frac{q}{25})(5k-1\mid 5)g(\frac{q}{25})(5k+1\mid 5) \\
       & =g(\frac{q}{25})^2(-1\mid 5) \\
       & =1.
  \end{split}
\end{equation}
For this, we need the following calculations: if $25\mid q$ then
$$\widehat{\frac{k q}{25}}=\gcd(\frac{k q}{25},\frac{q}{5})=\frac{q}{25}\cdot\gcd(k,5)=\frac{q}{25}$$ for $k\in\mathbb{N}$ with $5\nmid k$. Similarly, if $35\mid q$ then $\widehat{\frac{k q}{35}}=\frac{q}{35}$ for $k\in\mathbb{N}$ with $7\nmid k$; if $40\mid q$ then $\widehat{\frac{k q}{40}}=\frac{q}{40}$ for $k\in\mathbb{N}$ with $2\nmid k$. We also note that if $g(m)\neq0$ then $\widehat{m}\mid\frac{q}{5}$ and $\frac{q}{5\widehat{m}}\leq \frac{q}{5m_0}$.

Assume $\frac{q}{5m_0}=7$. For $m\in\mathbb{N}$ satisfying $g(m)\neq0$,
$\frac{q}{5\widehat{m}}=1, 2, 4, 5$ or $7$, and thus $\frac{q}{5}$   divides $4m, 5m$ or $7m$. For $m\in\mathbb{N}$ with $\frac{2q}{35}<m<\frac{3q}{35}$, we have
$\frac{q}{5}<4m<\frac{2q}{5}$ and $\frac{2q}{5}<7m<\frac{3q}{5}$. This shows that
there is no integer $m$ between $\frac{2q}{35}$ and  $\frac{3q}{35}$ such that
$\frac{q}{5}$ divides $4m$ or $7m$.
Since $$g(\frac{2q}{35})g(\frac{3q}{35})=g(\frac{q}{35})(2\mid 5)g(\frac{q}{35})(3\mid 5)=(6\mid 5)=1,$$ it follows from Lemma \ref{Conclusion 1} that $$\sum_{m=\frac{2q}{35}+1}^{\frac{3q}{35}-1}g(m)\neq0.$$ That is to say, there is an integer $m_1$ between $\frac{2q}{35}$ and  $\frac{3q}{35}$ such that $g(m_1)\neq0$ and
$\frac{q}{5}\mid 5m_1$. Since $\frac{q}{5}<5m_1<\frac{3q}{5}$ one has $m_1=\frac{2q}{25}$, which yields  $25\mid q$ and $g(\frac{q}{25})=g(\widehat{\frac{2q}{25}})\neq0$.

Assume $\frac{q}{5m_0}=8$. For $m\in\mathbb{N}$ satisfying $g(m)\neq0$,
$\frac{q}{5\widehat{m}}=1, 2, 4, 5, 7$ or $8$, and thus $\frac{q}{5}$   divides $5m, 7m$ or $8m$.
Consider the sets $$A_r=\{m\in\mathbb{Z}:\frac{9q}{40}<m<\frac{11q}{40},g(m)\neq0,\frac{q}{5}\mid rm\}, \quad r=5, 7, 8.$$
Since $$g(\frac{9q}{40})g(\frac{11q}{40})=g(\frac{q}{40})(9\mid 5)g(\frac{q}{40})(11\mid 5)=(99\mid 5)=1,$$ it follows from Lemma \ref{Conclusion 1} that $$\sum_{m\in A_5}g(m)+\sum_{m\in A_7}g(m)+\sum_{m\in A_8}g(m)=\sum_{m=\frac{9q}{40}+1}^{\frac{11q}{40}-1}g(m)\neq0.$$
By using an argument similar to that in the last paragraph, it suffices to prove
$$\sum_{m\in A_7}g(m)=\sum_{m\in A_8}g(m)=0.$$
Since $\frac{9q}{5}<8m<\frac{11q}{5}$, we see that if $A_8\neq\varnothing$ then it has exactly one element $m=\frac{q}{4}$. Therefore, $A_8=\varnothing$ by Lemma \ref{gq4}.
%
On the other hand, since $\frac{7q}{5}<7m<2q$, $A_7$ contains one or two elements in $\{\frac{8q}{35}, \frac{9q}{35}\}$, and in either case one has $35\mid q$.
From the identities
$$g(\frac{8q}{35})=g(\frac{q}{35})(8\mid 5)=-g(\frac{q}{35}),$$
$$g(\frac{9q}{35})=g(\frac{q}{35})(9\mid 5)=g(\frac{q}{35}),$$
we deduce that $A_7=\{\frac{8q}{35}, \frac{9q}{35}\}$, and then
$$\sum_{m\in A_7}g(m)=g(\frac{8q}{35})+g(\frac{9q}{35})=-g(\frac{q}{35})+g(\frac{q}{35})=0.$$ This proves the claim.

We have shown that
there exists $m_2\in\mathbb{N}$ with $m_2\leq\min\{5m_0,\frac{q}{11}\}$ such that
$$g(\frac{k q}{5}-m_2)g(\frac{k q}{5}+m_2)=1,\quad k=1,2.$$
Indeed, if $\frac{q}{5m_0}=7$ or $8$, one can take $m_2=\frac{q}{25}$ by (\ref{q5q25}),
while if $\frac{q}{5m_0}\geq11$ one can take $m_2=5m_0$
 by (\ref{q55m0}).
Put
$$m_3=\min\{m_0< m\leq m_2:g(\frac{k q}{5}-m)g(\frac{k q}{5}+m)=1,k=1,2\},$$
and
$$m_4=\max\{m_0\leq m<m_3:g(\widehat{m})\neq0\}.$$
Now we will show $5\nmid\frac{m_4}{\widehat{m_4}}$
and $5\nmid\frac{q}{5\widehat{m_4}}$.
If $5\mid\frac{m_4}{\widehat{m_4}}$ then
$$m_4\geq5\widehat{m_4}\geq5m_0\geq m_2\geq m_3>m_4,$$
which is impossible. A direct calculation gives
\begin{equation*}
  \begin{split}
      g(\frac{k q}{5}-m_4)g(\frac{k q}{5}+m_4)  & =g(\widehat{m_4})(\frac{k q}{5\widehat{m_4}}-\frac{m_4}{\widehat{m_4}}\mid 5)g(\widehat{m_4})(\frac{k q}{5\widehat{m_4}}+\frac{m_4}{\widehat{m_4}}\mid 5) \\
       & =g(\widehat{m_4})^2((\frac{k q}{5\widehat{m_4}})^2-(\frac{m_4}{\widehat{m_4}})^2\mid 5) \\
       & =((\frac{k q}{5\widehat{m_4}})^2-(\frac{m_4}{\widehat{m_4}})^2\mid 5)
  \end{split}
\end{equation*} for $k\in\{1, 2\}$. This shows that
if $5\mid\frac{q}{5\widehat{m_4}}$ then
$$g(\frac{k q}{5}-m_4)g(\frac{k q}{5}+m_4)=(-(\frac{m_4}{\widehat{m_4}})^2\mid 5)=1,
\quad k=1,2,$$  contradicting with the definition of $m_3$. Therefore, $5\nmid\frac{m_4}{\widehat{m_4}}$
and $5\nmid\frac{q}{5\widehat{m_4}}$.

By Lemma \ref{pclem}, there exists $k_0\in\{1,2\}$ such that
\begin{equation}\label{k0q5}
 g(\frac{k_0 q}{5}-m_4)g(\frac{k_0 q}{5}+m_4)
=((\frac{k_0 q}{5\widehat{m_4}})^2-(\frac{m_4}{\widehat{m_4}})^2\mid 5)=-1.
\end{equation}
Since
$m_3\leq m_2\leq\frac{q}{11}$ and $q\geq20$, one has
 $$1\leq\frac{q}{5}-\frac{q}{11}-1\leq\frac{k_0 q}{5}-m_3-1\leq \frac{k_0 q}{5}+m_3\leq\frac{2q}{5}+\frac{q}{11}<\frac{q}{2}.$$
 Moreover, for any integer $m$ between $m_4$ and $m_3$ we have $g(\widehat{m})=0$, and thus by (\ref{q5m}), $$g(\frac{k_0 q}{5}-m)=g(\frac{k_0 q}{5}+m)=0.$$
It follows that $$|g(\frac{k_0 q}{5}-m_3)+g(\frac{k_0 q}{5}-m_4)|
=|S(\frac{k_0 q}{5}-m_4)-S(\frac{k_0 q}{5}-m_3-1)|\leq1,$$ and $$|g(\frac{k_0 q}{5}+m_4)+g(\frac{k_0 q}{5}+m_3)|
=|S(\frac{k_0 q}{5}+m_3)-S(\frac{k_0 q}{5}+m_4-1)|\leq1.$$ This together with (\ref{k0q5}) yields $g(\frac{k_0 q}{5}-m_3)=0$ or $g(\frac{k_0 q}{5}+m_3)=0$,
contradicting with the definition of $m_3$.
Hence we  complete the proof of Lemma \ref{keylem}.
\vskip2mm

It remains to prove Lemma \ref{m0q5q0}.

\noindent\textbf{Proof of Lemma \ref{m0q5q0}.}
Assume conversely $m_0>\frac{q}{5q_0}$, which is equivalent to that $m_0= \frac{q}{q_0}, \frac{q}{2q_0}$ or $\frac{q}{4q_0}$ since $m_0\mid\frac{q}{q_0}$ and $3\nmid q$.
It is clear that each divisor $d$ of $\frac{q}{q_0}$ satisfying $d\geq m_0$ is necessarily divided by $m_0$. Thus, $m_0\mid m$ for every $m\in\mathbb{N}$ with $g(m)\neq0$ (since $g(\widehat{m})\neq0$, the definition of $m_0$ gives $\widehat{m}\geq m_0$).
Combining this with the assumption that $$\gcd\left(\{m\in\mathbb{Z}:g(m)\neq0\}\cup\{\frac{q}{2}\}\right)=1,$$ one has $\gcd(m_0,\frac{q}{2})=1$.
A direct calculation gives
$$
        \gcd(m_0,\frac{q}{2})=\begin{cases}
                        m_0, &  \frac{q}{m_0}\ \mathrm{is}\ \mathrm{even}; \\
                        \frac{m_0}{2}, & \frac{q}{m_0}\ \mathrm{is}\ \mathrm{odd}.
                      \end{cases}
       $$ Then $m_0=1$ if $m_0=\frac{q}{2q_0},\ \frac{q}{4q_0}$ or $q_0$ is even, while
       $m_0=2$ if $m_0=\frac{q}{q_0}$ and  $q_0$ is odd.
 Now we have the following five cases.
\vskip2mm
\noindent\textbf{Case 1.}  $m_0=\frac{q}{q_0}$ and $q_0$ is even.

In this case, $m_0=1$ and $$g(m)=g(\widehat{m})\psi(\frac{m}{\widehat{m}})=g(1)\psi(m)=\psi(m), \quad m\in\mathbb{Z}.$$
Since $q_0=q\geq14$ and $3\nmid q_0$, we have
$q_0=2^ls$, where $l=2$ or $3$, and $s\geq5$ is odd.
Moreover, if $s=5$ then  $\psi(m)=\psi_{8,2}^{\mathrm{pr}}(m)\cdot(m \mid 5)$ (see Remark \ref{2ks}).  In this case, one has $$\psi(3)=\psi_{8,2}^{\mathrm{pr}}(3)\cdot(3 \mid 5)=1,$$ and then
$$g(1)g(3)=\psi(1)\psi(3)=1.$$ It follows from Lemma \ref{Conclusion 1} that $\psi(2)=g(2)\neq0$,  which contradicts with the fact $2\mid q_0$.
From this,  we deduce that $s\geq7$, i.e., $q_0$ has a prime divisor $p\geq7$. By Lemma \ref{psi1lem} (2), there exists $k\in\mathbb{N}$ with $1\leq k\leq\frac{p-1}{2}$, such that $\psi((\frac{k q_0}{p})^2-1)=1$, and then $$g(\frac{k q_0}{p}-1)g(\frac{k q_0}{p}+1)
=\psi(\frac{k q_0}{p}-1)\psi(\frac{k q_0}{p}+1)
=\psi((\frac{k q_0}{p})^2-1)=1.$$
Again by Lemma \ref{Conclusion 1}, $\psi(\frac{k q_0}{p})=g(\frac{k q_0}{p})\neq0$,
which is also a contradiction since
$$\gcd(\frac{k q_0}{p},q_0)\geq\frac{q_0}{p}>1.$$
\vskip2mm
\noindent\textbf{Case 2.}  $m_0=\frac{q}{q_0}$ and $q_0$ is odd.

 In this case, $m_0=2$, $g(1)=0$,  and \begin{equation}\label{g(m)=0foroddm}
                                g(m)=g(\widehat{m})\psi(\frac{m}{\widehat{m}})=\begin{cases}
                        g(1)\psi(m)=0, &  m\ \mathrm{is}\ \mathrm{odd}; \\
                        g(2)\psi(\frac{m}{2})=\psi(\frac{m}{2}), & m\ \mathrm{is}\ \mathrm{even}.
                      \end{cases}
                              \end{equation}
        Moreover, $q_0\geq13$ since $q\geq14$ and $q_0\equiv 1\,(\mathrm{mod}\,4)$ (see Remark \ref{2ks}).
 Since $\psi(1)=\psi(4)=1$ and $|\psi(2)|=|\psi(3)|=1$,
 there exists $c\in\{1, 2, 3\}$ such that $\psi(c)\psi(c+1)=1$.
 Therefore $$g(2c)g(2c+2)=\psi(c)\psi(c+1)=1,$$ forcing $g(2c+1)\neq0$ by Lemma \ref{Conclusion 1}.
However, this  contradicts with (\ref{g(m)=0foroddm}).
\vskip2mm
\noindent\textbf{Case 3.}  $m_0=\frac{q}{2q_0}$ or $\frac{q}{4q_0}$, and $q_0$ is even.

In this case, $m_0=1$ and $4\mid q_0$ (see Remark \ref{2ks}). This gives $8\mid q$ since $\frac{q}{q_0}=2$ or $4$, and thus $\frac{q}{4}\pm1$ is odd.
Hence $\widehat{\frac{q}{4}\pm1}=1$ and
$g(\frac{q}{4}\pm1)=g(1)\psi(\frac{q}{4}\pm1)=\psi(\frac{q}{4}\pm1)$.
Since $$\psi(\frac{q}{4}-1)=\psi(1-\frac{q}{4})=\psi(1-\frac{q}{4}+\frac{q}{2})
=\psi(\frac{q}{4}+1),$$ and $\gcd(\frac{q}{4}\pm1,q_0)=1$,
we have $g(\frac{q}{4}-1)g(\frac{q}{4}+1)=1$, forcing $g(\frac{q}{4})\neq0$ by Lemma \ref{Conclusion 1}. This  contradicts with Lemma \ref{gq4}.
\vskip2mm
\noindent\textbf{Case 4.}  $m_0=\frac{q}{2q_0}$ and $q_0$ is odd.

In this case, $m_0=1$, $\frac{q}{q_0}=2$, and similar to Case 2, we  have
$q_0\geq13$ and
$$
        g(m)=\begin{cases}
                        \psi(m), &  m\ \mathrm{is}\ \mathrm{odd}; \\
                        g(2)\psi(\frac{m}{2}), & m\ \mathrm{is}\ \mathrm{even}.
                      \end{cases}
       $$
Again take $c\in\{1, 2, 3\}$ satisfying $\psi(c)\psi(c+1)=1$.
Then $$g(q_0-2c-2)g(q_0-2c)=\psi(q_0-2c-2)\psi(q_0-2c)=\psi^2(-2)\psi(c)\psi(c+1)=1,$$
forcing $g(q_0-2c-1)\neq0$ by Lemma \ref{Conclusion 1}, and thus $g(2)\neq0$ since $$g(q_0-2c-1)=g(2)
\psi(\frac{q_0-2c-1}{2}).$$ Let $p$ be the smallest odd prime of $q_0$. For any $m\in\mathbb{N}$ with $m<p$, we have $g(m)\neq0$ since
$\psi(m)\neq0$. Moreover, $$|g(m)+g(m+1)|=|S(m+1)-S(m-1)|\leq1$$ for $1\leq m\leq p-2$.
It follows from Lemma \ref{abclem} (2) that $$g(m)=(-1)^{m-1}g(1)=(-1)^{m-1},\quad m=1,2,\cdots,p-1.$$
In particular, the identity $g(4)=g(2)\psi(2)$ gives $\psi(2)=1$.
We further have $\psi(\frac{p+1}{2})=1$. In fact,
if $\frac{p+1}{2}$ is odd then $\psi(\frac{p+1}{2})=g(\frac{p+1}{2})=1$; if $\frac{p+1}{2}$ is even,
then $$g(\frac{p+1}{2})=g(2)\psi(\frac{p+1}{4})=-\psi(\frac{p+1}{4}),$$ and thus $$\psi(\frac{p+1}{2})=\psi(2)\psi(\frac{p+1}{4})=-\psi(2)g(\frac{p+1}{2})=1.$$
It follows from this claim that $$g(p+1)=g(2)\psi(\frac{p+1}{2})=-1=g(p-1),$$
forcing $\psi(p)=g(p)=1$ by Lemma \ref{Conclusion 1}. This contradict with the fact $\gcd(p,q_0)=p>1$.

\vskip2mm
\noindent\textbf{Case 5.}  $m_0=\frac{q}{4q_0}$ and $q_0$ is odd.

In this case, $m_0=1$, $\frac{q}{q_0}=4$ and $q_0\equiv 1\,(\mathrm{mod}\,4)$.
This yields $$\widehat{q_0}=\widehat{q_0\pm2}=\widehat{q_0\pm4}=1,$$
$$\widehat{q_0-3}=\widehat{q_0+1}=2$$ and
$$\widehat{q_0-1}=\widehat{q_0+3}=4.$$
It follows that
$$g(q_0\pm2)=g(1)\psi(q_0\pm2)=\psi(\pm2)=\psi(2),$$
$$g(q_0\pm4)=g(1)\psi(q_0\pm4)=\psi(\pm4)=1,$$
which shows $g(q_0-3)=g(q_0+3)$ by Lemma \ref{Conclusion 1}. Moreover,
$$g(q_0-3)=g(2)\psi(\frac{q_0-3}{2})=g(2)\psi(\frac{q_0-3}{2})\psi(4)
=g(2)\psi(2q_0-6)=g(2)\psi(6),$$
$$g(q_0-1)=g(4)\psi(\frac{q_0-1}{4})=g(4)\psi(\frac{q_0-1}{4})\psi(4)
=g(2)\psi(q_0-1)=g(4),$$
and similarly, $g(q_0+1)=g(2)\psi(2)$, $g(q_0+3)=g(4)\psi(3)$.
Therefore $g(2)\psi(2)\psi(3)=g(4)\psi(3)$, forcing $g(q_0-1)=g(q_0+1)$.
Since $g(q_0-2)g(q_0+2)=1$, again by Lemma \ref{Conclusion 1} we have $g(q_0)\neq0$.
However, $$g(q_0)=g(\widehat{q_0})\psi(\frac{q_0}{\widehat{q_0}})=g(1)\psi(q_0)=0.$$
This completes the proof.
 $\hfill \square $
\vskip2mm

\section{The number of  component intervals of $V$}

Using the Euler totient function $\phi$, we obtain a formula for the number of  component intervals of $V$ with $\mathbf{1}_V\in\mathcal{C}$. As a consequence, it will be shown that when $V$ is a subinterval of $(0,1)$ that satisfies $\mathbf{1}_V\in\mathcal{C}$, the least common denominator $t_V$ is not larger than $6$.
\begin{thm}\label{num of component}
  Let $V,q$ and $g$ be given as in (\ref{Vdisplay}), (\ref{q}) and (\ref{g(m)}). Suppose $\mathbf{1}_V\in\mathcal{C}$.
  \begin{itemize}
    \item [(1)] The set of all boundary points of $V$ in $(0,1)$ is
    \begin{equation} \label{set of boundary}\bigsqcup_{\substack{n\geq3 \\ J_V(\frac{2}{n})\neq0}}\Big\{\frac{2m}{n}:1\leq m<\frac{n}{2},\gcd(m,n)=1\Big\},\end{equation}
     and the number of component intervals of $V$ is
     \begin{equation}\label{N=}
       N=\frac{\mathbf{1}_V(0^+)+\mathbf{1}_V(1^-)}{2}+\frac{1}{4}\, \sum_{\substack{n\geq3 \\ J_V(\frac{2}{n})\neq0}}\phi(n).
     \end{equation}
    \item [(2)] If $g\in E_q$ then
    $$2(\mathbf{1}_V(0^+)+\mathbf{1}_V(1^-))+\sum_{\substack{n\geq3 \\ J_V(\frac{2}{n})=1}}\phi(n)=\sum_{\substack{n\geq3 \\ J_V(\frac{2}{n})=-1}}\phi(n).$$
  \end{itemize}
\end{thm}
\begin{proof}
Let $M$ be the number of  boundary points of $V$ in $(0,1)$. Then
\begin{equation}\label{NM}
N=\frac{1}{2}(\mathbf{1}_V(0^+)+\mathbf{1}_V(1^-)+M).
\end{equation}
We also put $$F_n=\{m:1\leq m<\frac{n}{2},\gcd(m,n)=1\},\quad n=1,2,\cdots.$$
It is easy to see that to any rational number $r\in(0,1)$,  there corresponds a unique positive integer $n$ with $n\geq3$, such that $r=\frac{2m}{n}$ for some $m\in F_n$.

(1)
By Lemma \ref{gatb},  $|J_V(\frac{2m}{n})|=|J_V(\frac{2}{n})|$ for any
relatively prime positive integers $m,n$.
Thus for $n\geq3$ and $m\in F_n$,  $\frac{2m}{n}$  is a
boundary point of $V$ in $(0,1)$ if and only if $J_V(\frac{2}{n})\neq0$.
Therefore,
the set of all boundary points of $V$ in $(0,1)$ coincides with the set given in (\ref{set of boundary}), and
$$M=\sum_{\substack{n\geq3 \\ J_\varphi(\frac{2}{n})\neq0 }}\sharp F_n
=\frac{1}{2}\sum_{\substack{n\geq3 \\ J_\varphi(\frac{2}{n})\neq0 }} \phi(n).$$ This together with (\ref{NM}) gives (\ref{N=}).

\vskip2mm

(2) Assume $g\in E_q$. Then by Lemma \ref{gatb},  $J_V(\frac{2m}{n})=J_V(\frac{2}{n})$ for any
relatively prime positive integers $m,n$. Thus for $n\geq3$ and $m\in F_n$,  $\frac{2m}{n}$  is a  left endpoint of some component interval of $V$ if and only if $J_V(\frac{2}{n})=1$. It follows that
$$N=\mathbf{1}_V(0^+)+\sum_{\substack{n\geq3 \\ J_\varphi(\frac{2}{n})=1 }}\sharp F_n
=\mathbf{1}_V(0^+)+\frac{1}{2}\sum_{\substack{n\geq3 \\ J_\varphi(\frac{2}{n})=1 }} \phi(n),$$
which proves (2) by (1).
\end{proof}

 We record the following corollary.
Setting $$t(n)=\prod_{\substack{p\leq 4n+1 \\ p\,\text{prime}}}p^{[\log_p6n]},\quad n=1,2,\cdots,$$
where $[\log_p6n]$ is the integer part of $\log_p6n$,
we have

\begin{cor} \label{cor of jump discon}
Let $V$ be given as in (\ref{Vdisplay}). If $\mathbf{1}_V\in\mathcal{C}$ then $t_V\mid t(N)$, i.e., every boundary point of $V$ has form $\frac{s}{t(N)}$ for some integer $s$.
\end{cor}
\begin{proof}
To reach a contradiction, assume conversely that $V$ has a boundary point $\frac{s_0}{t_0}\in(0,1)$ with $t_0\nmid t(N)$, where $s_0,t_0$ is a pair of relatively prime positive integers.
Then there is a prime divisor $p$ of $t_0$, such that either $p>4N+1$ or $p^{[\log_p6N]+1}|t_0$.
In either case, we would have \begin{equation}\label{>4N}
                                \phi(t_0\cdot\gcd(t_0,2))>4N.
                              \end{equation}  In fact,
if there is a prime divisor $p$ of $t_0$, such that $p>4N+1$,
then $$\phi(t_0)\geq\phi(p)=p-1>4N;$$
if $2^{[\log_26N]}+1|t_0$, then
$$\phi(t_0)\geq\phi(2^{[\log_26N]+1})=2^{[\log_26N]}
>2N;$$
if there is an odd prime divisor $p$ of $t_0$, such that $p^{[\log_p6N]+1}|t_0$, then
$$\phi(t_0)\geq\phi(p^{[\log_p6N]+1})=p^{[\log_p6N]}(p-1)
\geq\frac{2}{3}p^{[\log_p6N]+1}>4N.$$
Therefore, when $t_0$ is odd, $\phi(t_0)>4N$;
when $t_0$ is even, $\phi(t_0)>N$ and thus $\phi(2t_0)=2\phi(t_0)>4N$.

Applying Proposition \ref{gatb} (2), we see that
$$\frac{2}{t_0\cdot\frac{2}{\gcd(s_0,2)}}=\frac{\gcd(s_0,2)}{t_0}$$ is also a boundary point of $V$.
Then by Theorem \ref{num of component} (1) and (\ref{>4N}),
$$\phi(t_0)\leq\phi(t_0\cdot\frac{2}{\gcd(s_0,2)})
\leq4N<\phi(t_0\cdot\gcd(t_0,2))\leq\phi(2t_0).$$
This implies that both $s_0$ and $t_0$ are even, which contradicts with the assumption that
$s_0$ and $t_0$ are relatively prime.
The proof is complete.
\end{proof}

  Corollary \ref{cor of jump discon} leads to an upper bound  for open subsets $V$ with $\mathbf{1}_V\in\mathcal{C}$ when the number of component intervals of $V$ is restricted. More precisely,
let $\mathcal{V}_n\ (n\in\mathbb{N})$ denote  the set of all non-degenerated open subsets $V$ of $(0,1)$  which have  at most $n$  component intervals with rational boundary points. It follows from Corollary \ref{cor of jump discon} that  there are at most ${\binom {2n} {t(n)+1}}$ elements in
$$\big\{V\in \mathcal{V}_n: \mathbf{1}_V\in\mathcal{C}\big\}.$$

We end this section with the following corollary.

\begin{cor} \label{t<7}
 Let $V$  be given as in (\ref{Vdisplay}). If $N=1$ (i.e., $V$ is a subinterval of $(0,1)$) and $\mathbf{1}_V\in\mathcal{C}$, then $t_V\leq6$.
\end{cor}
\begin{proof}
Since $t(1)=60$, Corollary \ref{cor of jump discon} gives $t_V\mid60$. Thus, if $4\nmid t_V$ and $5\nmid t_V$ then $t_V\mid 6$.

Assume $5\mid t_V$. Then by Theorem \ref{num of component} (1), there exists
$l\in\mathbb{N}$ such that $5\mid l$ and $J_V(\frac{2}{l})\neq0$. Moreover,
$$4\geq\sum_{\substack{n\geq3 \\ J_V(\frac{2}{n})\neq0}}\phi(n)\geq\phi(l)\geq\phi(5)=4,$$
which yields that $l=5$ or $10$, and that for $n\geq3$, $J_V(\frac{2}{n})\neq0$ only when $n=l$. Therefore, $t_V=5\leq6$.

Assume $4\mid t_V$. Similarly, there exists
$l'\in\mathbb{N}$ such that $8\mid l'$, $J_V(\frac{2}{l'})\neq0$ and $$4\geq\sum_{\substack{n\geq3 \\ J_V(\frac{2}{n})\neq0}}\phi(n)\geq\phi(l')\geq\phi(8)=4,$$
which gives $l'=8$ and $t_V=4\leq 6$.
 \end{proof}

\section{The case $t_V\leq6$}

In  this section, we explicitly determine $V$ in case $t_V\leq6$. This together with Corollary \ref{t<7} gives Theorem \ref{interval}.
 \begin{thm} \label{1to6}
   Let $V$ be given as in (\ref{Vdisplay}).
\begin{itemize}
  \item[(1)] When $t_V=1,2$,  $\mathbf{1}_{(0,1)},\mathbf{1}_{(0,\frac{1}{2})},\mathbf{1}_{(\frac{1}{2},1)}\in\mathcal{C}$;
  \item[(2)] When $t_V=3$, $\mathbf{1}_V\in\mathcal{C}$ if and only if $V=(0,\frac{2}{3}),(\frac{1}{3},1),(\frac{1}{3},\frac{2}{3})$ or $(0,\frac{1}{3})\cup(\frac{2}{3},1)$;
  \item[(3)] When $t_V=4$, $\mathbf{1}_V\in\mathcal{C}$ if and only if $V=(\frac{1}{4},\frac{3}{4}),
      (0,\frac{1}{4})\cup(\frac{1}{2},\frac{3}{4})$ or $(\frac{1}{4},\frac{1}{2})\cup(\frac{3}{4},1)$;
  \item[(4)] When $t_V=5$, $\mathbf{1}_V\in\mathcal{C}$ if and only if $V=(\frac{1}{5},\frac{3}{5}),
      (\frac{2}{5},\frac{4}{5}),(\frac{1}{5},\frac{2}{5})
      \cup(\frac{3}{5},\frac{4}{5})$ or
      $(0,\frac{1}{5})\cup(\frac{2}{5},\frac{3}{5})\cup
      (\frac{4}{5},1)$.
  \item[(5)] When $t_V=6$, $\mathbf{1}_V\in\mathcal{C}$ if and only if $V=(\frac{1}{6},\frac{5}{6}),
      (0,\frac{1}{6})\cup(\frac{1}{3},\frac{5}{6}),
      (0,\frac{1}{3})\cup(\frac{1}{2},\frac{2}{3}),
      (0,\frac{1}{3})\cup(\frac{1}{2},1),
      (0,\frac{1}{2})\cup(\frac{2}{3},1),
      (\frac{1}{6},\frac{2}{3})\cup(\frac{5}{6},1),
      (0,\frac{1}{6})\cup(\frac{1}{3},\frac{1}{2})\cup(\frac{2}{3},\frac{5}{6})$
      or $(\frac{1}{6},\frac{1}{3})\cup(\frac{1}{2},\frac{2}{3})\cup(\frac{5}{6},1)$.
\end{itemize}
 \end{thm}

To prove Theorem \ref{1to6}, we need to establish the following.

\begin{lem} Let $V$ and $g$ be given as in (\ref{Vdisplay}) and (\ref{g(m)}).
  \begin{itemize}
     \item [(1)] If $t_V=4$
     and $g\in E_8$, then $V=V_{4,0}$ or $V_{4,1}$.
     \item [(2)] If $t_V=5$
     and $g\in E_{10}$, then $V=V_{5,0}$ or $V_{5,1}$.
   \end{itemize}
\end{lem}
\begin{proof}
  (1) Since $a:=g(1)=g(3)\neq0$, it follows from Lemma \ref{Conclusion 1} that $g(2)=-a$. Thus, $V=V_{4,1}$ when $a=1$; $V=V_{4,0}$ when $a=-1$.
  \vskip2mm
   \noindent(2) Note that $a:=g(1)=g(3)$, $b:=g(2)=g(4)$, and at least one of $a$ and $b$ is not equal to $0$. Then Lemma \ref{Conclusion 1} yields $ab=-1$, which completes the proof.
\end{proof}

\begin{lem}\label{t=36} Let $V,g$ and $R$ be given as in (\ref{Vdisplay}), (\ref{g(m)}) and (\ref{R}).
  \begin{itemize}
     \item [(1)] If $t_V=3$
     and $g\in E_6$, then $R(z_1,z_2)$ has no zeros in $\mathbb{D}^{2}$ if and only if $V=V_{3,0}, V_{3,1}, (0,\frac{2}{3})$ or $(\frac{1}{3},1)$.
     \item [(2)] If $t_V=6$
     and $g\in E_{12}$, then $R(z_1,z_2)$ has no zeros in $\mathbb{D}^{2}$ if and only if $V=V_{6,0}, V_{6,1}, (0,\frac{1}{6})\cup(\frac{1}{3},\frac{5}{6}),
      (0,\frac{1}{3})\cup(\frac{1}{2},\frac{2}{3}),
      (0,\frac{1}{3})\cup(\frac{1}{2},1),
      (0,\frac{1}{2})\cup(\frac{2}{3},1)$ or
      $(\frac{1}{6},\frac{2}{3})\cup(\frac{5}{6},1)$.
   \end{itemize}
\end{lem}
The proof of Lemma \ref{t=36} is due to a direct calculation. We refer to Lemma \ref{t=36'} in Appendix for the details of the calculation.

\vskip2mm
\noindent\textbf{Proof of Theorem \ref{1to6}.}
(1) has been shown in Examples \ref{01} and \ref{Vt0Vt1}. For the rest of the proof, suppose $\mathbf{1}_V\in\mathcal{C}$ and put $g(m)=J_V(\frac{m}{t_V})\ (m\in\mathbb{Z})$. According to Corollary \ref{fg}, we can assume $g\in E_{q,\psi}$ for some $\psi\in\mathcal{P}_q$, where $q=2t_V$. Let $q_0$ denote the modulus of $\psi$.

With Theorem \ref{Eq}, the case  $q_0=1$ is completely solved due to the previous two lemmas. In particular, we have proved (2) (see Remark \ref{2ks} for possible values of $q_0$).
From now on, we also assume $q_0>1$.
%


   \vskip2mm
   \noindent(3) The case $t_V=4$: since $q_0>1$ we have $q_0=8$, $\psi=\psi_{8,2}^{\mathrm{pr}}$.
   Hence,
    $$g(3)=g(1)\psi(3)=-g(1),$$ and by Theorem \ref{criterion}, $g(0)=g(2)=g(4)=0$. This together with Theorem \ref{q=q02} implies that when $q_0>1$, $\mathbf{1}_V\in\mathcal{C}$ if and only if $V=(\frac{1}{4},\frac{3}{4})$.
   \vskip2mm

   \noindent(4) The case $t_V=5$: since $q_0>1$ we have   $q_0=5$, $\psi(n)=(n\mid 5)$.
  Hence, $$g(3)=g(1)\psi(3)=-g(1),$$ $$g(4)=g(2)\psi(2)=-g(2),$$ and by Theorem \ref{criterion},  $g(0)=g(5)=0$. This together with Theorem \ref{specriterion} implies that when $q_0>1$,  $\mathbf{1}_V\in\mathcal{C}$ if and only if $V=(\frac{1}{5},\frac{3}{5})$
      or $(\frac{2}{5},\frac{4}{5})$.
   \vskip2mm
   \noindent(5) The case $t_V=6$: since $q_0>1$ we have   $q_0=12$, $\psi(n)=\psi_4^{\mathrm{pr}}(n)\cdot(n\mid 3)$.
Hence, $$g(5)=g(1)\psi(5)=-g(1),$$ and by Theorem \ref{criterion}, $g(0)=g(2)=g(3)=g(4)=g(6)=0$. This together with Theorem \ref{q=q02} implies that when $q_0>1$, $\mathbf{1}_V\in\mathcal{C}$ if and only if $V=(\frac{1}{6},\frac{5}{6})$.
 $\hfill \square $
\vskip2mm


\section{When $V=V_{t_V,0}$ or $V_{t_V,1}$?}
In this section, we will prove Theorems \ref{Vtthm} and \ref{pkthm}, which are restated below.
\vskip2mm

\noindent\textbf{Theorem \ref{Vtthm}.}
   \textit{Let $V$  be given  as in (\ref{Vdisplay}). Suppose that $t_V\geq7$, $3\nmid t_V$, and there exists
 a boundary point $\frac{s}{t_V}$ of $V$ with   $\gcd(s,t_V)=1$. Then
   $\mathbf{1}_V\in\mathcal{C}$ if and only if
  $V=V_{t_V,0}$ or $V_{t_V,1}$.}
\vskip2mm

\noindent\textbf{Theorem \ref{pkthm}.}
   \textit{Let $V$  be given as in (\ref{Vdisplay}). Suppose that $t_V\geq7$ and $t_V=p^k$ for some prime $p$ and some $k\in\mathbb{N}$. Then
   $\mathbf{1}_V\in\mathcal{C}$ if and only if
  $V=V_{p^k,0}$ or $V_{p^k,1}$.}
  \vskip2mm

Neither of the two assumptions
\begin{itemize}
  \item [(1)] $3\nmid t_V$;
  \item [(2)] there exists
 a boundary point $\frac{s}{t_V}$ of $V$ with   $\gcd(s,t_V)=1$
\end{itemize}
in Theorem \ref{Vtthm} can be dropped. This can be illustrated by the following  examples.
\begin{exam}\label{Vtexam1}
  For $V=(\frac{1}{15},\frac{7}{15})\cup(\frac{11}{15},\frac{13}{15})$, $W=(\frac{2}{15},\frac{4}{15})\cup(\frac{8}{15},\frac{14}{15})$ and $m\in\mathbb{Z}$, put $g_V(m)=J_V(\frac{m}{15})$ and $g_W(m)=J_W(\frac{m}{15})$. Then $g_V, g_W\in E_{30, \psi}$, where $\psi(n)=(n\mid 5)$. It follows from Theorem \ref{specriterion} that $\mathbf{1}_V,\mathbf{1}_W\in\mathcal{C}$.
\end{exam}

\begin{exam}\label{Vtexam2}
  Suppose that $p_1, p_2$ are twin primes ($p_1=p_2+2$), and set $$V=(\frac{1}{p_1},\frac{1}{p_2})\cup(\frac{3}{p_1},\frac{3}{p_2})\cup\cdots
     \cup(\frac{2m-1}{p_1},\frac{2m-1}{p_2})\cup\cdots\cup(\frac{p_1-2}{p_1},1)$$
     and $$W=(0,\frac{2}{p_1})\cup(\frac{2}{p_2},\frac{4}{p_1})\cup\cdots\cup(\frac{2m-2}{p_2},\frac{2m}{p_1})
     \cup\cdots\cup(\frac{p_2-1}{p_2},\frac{p_1-1}{p_1}).$$
     Then $\mathbf{1}_V,\mathbf{1}_W\in\mathcal{C}$.

     In fact, for $m\in\mathbb{Z}$, put $g_V(m)=J_V(\frac{m}{p_1p_2})$ and $g_W(m)=J_W(\frac{m}{p_1p_2})$. Then $g_V,g_W\in E_{2p_1p_2}$. By a direct calculation,
     \begin{equation*}
       \begin{split}
          R_V(\mathbf{z}) & =\sum_{d\mid 2p_1p_2}g_V(\frac{2p_1p_2}{d})M_d(\mathbf{z}) \\
            & =g(p_1p_2)M_2(z)+g(p_2)M_{2p_1}(z,\lambda)+g(p_1)M_{2p_2}(z,\mu) \\
            & =-2M_2(z)+M_{2p_1}(z,\lambda)-M_{2p_2}(z,\mu) \\
            & =M_2(z)(-2+M_{p_1}(\lambda)-M_{p_2}(\mu)) \\
            & =(z-1)(-2+\lambda-\mu),
       \end{split}
     \end{equation*}
          \begin{equation*}
       \begin{split}
          R_W(\mathbf{z}) & =\sum_{d\mid 2p_1p_2}g_W(\frac{2p_1p_2}{d})M_d(\mathbf{z}) \\
            & =g(2p_1p_2)+g(2p_2)M_{p_1}(\lambda)+g(2p_1)M_{p_2}(\mu) \\
            & =2-M_{p_1}(z,\lambda)+M_{p_2}(z,\mu) \\
            & =2-\lambda+\mu.
       \end{split}
     \end{equation*}
  Both $R_V$ and $R_W$   has no zeros in $\mathbb{D}^{S_{2p_1p_2}}$. It follows from Theorem \ref{Eq} that $\mathbf{1}_V,\mathbf{1}_W\in\mathcal{C}$.
\end{exam}

Similarly, one can obtain the following two examples.

\begin{exam}\label{Vtexam3}
   Suppose that $p_1, p_2$ are two primes satisfying $p_1=2p_2-1$, and set
     $$V=(0,\frac{1}{p_2})\cup(\frac{2}{p_1},\frac{2}{p_2})\cup\cdots\cup(\frac{2m-2}{p_1},\frac{m}{p_2})
     \cup\cdots\cup(\frac{p_1-1}{p_1},1).$$
     Then $\mathbf{1}_V\in\mathcal{C}$.
\end{exam}

\begin{exam}\label{Vtexam4}
  Suppose that $p_1, p_2$ are two primes satisfying $p_1=2p_2+1$, and set
     $$V=(0,\frac{2}{p_1})\cup(\frac{1}{p_2},\frac{4}{p_1})\cup\cdots\cup(\frac{m-1}{p_2},\frac{2m}{p_1})
     \cup\cdots\cup(\frac{p_2-1}{p_2},\frac{p_1-1}{p_1}).$$
     Then $\mathbf{1}_V\in\mathcal{C}$.
\end{exam}

We first prove Theorem \ref{Vtthm} and then prove Theorem \ref{pkthm}.
\vskip2mm
\noindent\textbf{Proof of Theorem \ref{Vtthm}.}
We only need to prove  the necessity. For this, assume $\mathbf{1}_V\in\mathcal{C}$  and for $m\in\mathbb{Z}$, set $\widehat{m}=\gcd(m,2t_V)$,
$g(m)=J_V(\frac{m}{t_V})$.  Moreover,
$g(m_1)=g(m_2)$ whenever $\widehat{m_1}=\widehat{m_2}$ by Theorem \ref{criterion}.
In particular, since $3\nmid t_V$ one has
$$g(3m)=g(m)=g(\widehat{m}),\quad m\in\mathbb{Z}.$$
 It suffices to show $g(m)=(-1)^{m-1}g(1)$ for $1\leq m\leq t_V-1$. According to Lemma \ref{abclem} (2), this in turn reduces to showing
$g(m)\neq0$ for $1\leq m\leq t_V-1$.

Assume conversely that there exists $m\in\mathbb{N}$ with $m<t_V$ such that  $g(m)=0$.
Take $m_0$ to be the smallest positive integer among those $m$'s satisfying $g(m)=0$.  We claim that $4\leq m_0\leq\frac{t_V}{2}$ and $3\nmid m_0$. On one hand, the identity $g(\widehat{m_0})=g(m_0)=0$ implies that $\widehat{m_0}=m_0$, i.e., $m_0\mid 2t_V$. Since $m_0<t_V$ and $3\nmid t_V$, we have $m_0\leq\frac{t_V}{2}$ and $3\nmid m_0$. On the other hand, since
$$\widehat{s_0}=\gcd(s_0,2t_V)=\gcd(s_0,2),$$
it follows that $g(\gcd(s_0,2))=g(s_0)\neq0$.
If $s_0$ is odd, then $g(3)=g(1)\neq0$, forcing $g(2)\neq0$ by Lemma \ref{Conclusion 1}. If $s_0$ is even, then $t_V$ is odd and $\widehat{4}=2$. This yields that
$g(4)=g(2)\neq0$, and thus $g(1)=g(3)\neq0$ by Lemma \ref{Conclusion 1}. In either case, $g(m)\neq0$ for $m\in\{1,2,3\}$, which gives $m_0\geq4$.

By Lemma \ref{abclem} (2), \begin{equation}\label{g(m)=(-1)m-1g(1)}
 g(m)=(-1)^{m-1}g(1),\quad 1\leq m\leq m_0-1.
\end{equation} From this, we can deduce that
$m_0\equiv 1\,(\mathrm{mod}\,3)$. In fact, if $m_0\equiv 2\,(\mathrm{mod}\,3)$
then $$g(m_0+1)=g(\frac{m_0+1}{3})
=(-1)^{\frac{m_0-2}{3}}g(1)=(-1)^{m_0-2}g(1)=g(m_0-1)\neq0,$$
which would yield $g(m_0)\neq0$ by Lemma \ref{Conclusion 1}.
Hence,
$$g(m_0+2)=g(\frac{m_0+2}{3})
=(-1)^{\frac{m_0-1}{3}}g(1)=(-1)^{m_0-1}g(1)=-g(m_0-1),$$
forcing $g(m_0+1)=g(m_0)+g(m_0+1)=0$ by Lemma \ref{Conclusion 1}. Since $\widehat{m_0+1}\neq m_0$, we also have $\widehat{m_0+1}=m_0+1$, i.e., $m_0+1\mid 2t_V$.

Set $2u$ to be the even integer in $\{m_0, m_0+1\}$.
Then $u\leq m_0-2$ since $m_0\geq4$.
We have shown previously $u\mid t_V$ and \begin{equation}\label{g(2u)=0}
  g(2u)=0.
\end{equation}
Write $t_V=2^ks$ and $u=2^lv$, where $k,l\geq0$ and $s, v$ are odd.
Put $u'=2^lv'$, where $v'=\frac{s+v}{2}$ when $s\equiv v\,(\mathrm{mod}\,4)$ and $k>l$; $v'=\frac{s-v}{2}$ when $s\not\equiv v\,(\mathrm{mod}\,4)$ or $k=l$.
Then $2u'<t_V-1$ and $2u'\equiv cu\,(\mathrm{mod}\,s)$ for some $c\in\{-1,1\}$, which gives
$$\widehat{2u'\pm1}=\gcd(2u'\pm1,2t_V)=\gcd(2u'\pm1,s)
=\gcd(cu\pm1,s)=\gcd(u\pm c,s).$$
Therefore, $\widehat{2u'\pm1}\leq u+1\leq m_0-1$, and thus (\ref{g(m)=(-1)m-1g(1)}) yields
$$g(2u'\pm1)=g(\widehat{2u'\pm1})=(-1)^{\widehat{2u'\pm1}-1}
g(1)=g(1)\neq0.$$
It follows from Lemma \ref{Conclusion 1} that $g(2u')=-g(1)\neq0$.
However, since
$$\gcd(u',t_V)=2^l\gcd(v',s)
=2^l\gcd(2v',s)
=2^l\gcd(v,s)
=\gcd(u,t_V),$$
we have $$\widehat{2u'}=2\gcd(u',t_V)=2\gcd(u,t_V)=\widehat{2u},$$ and then
$g(2u')=g(2u)=0$ by (\ref{g(2u)=0}). This is a contradiction and completes the proof.
$\hfill \square $
\vskip2mm

To prove Theorem \ref{pkthm}, we need two more lemmas.

\begin{lem}\label{3k} Suppose $k\geq2$ and $g\in E_{2\cdot3^k}$. Assume that either $g(1)$ or $g(2)$ is not equal to $0$, and
  the sum function $S(m)=\sum_{i=1}^{m}g(i)$ of $g$ only takes values $0$ or $1$ for $1\leq m\leq3^k-1$.
  \begin{itemize}
    \item [(1)] If $g(1)g(2)\neq0$ then $g(m)=(-1)^{m-1}$ for every $1\leq m\leq3^k-1$.
    \item [(2)] If $g(1)g(2)=0$ then $g(3^l)=g(2\cdot3^l)=-1$ for every $1\leq l\leq k-1$.
  \end{itemize}
\end{lem}
\begin{proof}
  For every $1\leq l\leq k-1$, since $g\in E_{2\cdot3^k}$ we have
  $$g(3^l-1)=g(3^l+1)=g(2)$$ and $$g(3^l-2)=g(3^l+2)=g(1).$$
  By Lemma \ref{Conclusion 1}, if $g(2)\neq0$ then $g(3^l)=-g(2)$; if $g(2)=0$ then $g(1)\neq0$, which gives
  $$g(3^l)=g(3^l-1)+g(3^l)+g(3^l+1)=-g(1).$$
  Similarly,   $$g(2\cdot3^l)=\begin{cases}
                        -g(1), &  g(1)\neq0; \\
                        -g(2), & g(1)=0,
                      \end{cases}\quad 1\leq l\leq k-1.$$
                      This proves the lemma.
\end{proof}

\begin{lem}\label{abineq} Suppose $k\geq2$, $a,b\in\{-1,0,1\}$ and $ab=0$.
  Then
  $$|a\xi^k-(2a+b)\xi^{k-1}-1|
  >|(b-a)(\xi^k-\xi^{k-1})+2|$$
for some $\xi\in\overline{\mathbb{D}}$.
\end{lem}
The proof of Lemma \ref{abineq} is due to a direct calculation. See Lemma \ref{abineq'} in Appendix.
\vskip2mm
\noindent\textbf{Proof of Theorem \ref{pkthm}.}
According to Theorem \ref{Vtthm}, it remains to prove the  case when $p=3$ and $k\geq2$.
Moreover, we only need to prove  the necessity. For this, assume $\mathbf{1}_V\in\mathcal{C}$  and set 
$g(m)=J_V(\frac{m}{3^k})\ (m\in\mathbb{Z})$.
By Theorem \ref{criterion}, $g\in E_{2\cdot3^k}$ and the polynomial
$$R=\sum_{d\mid 2\cdot3^k}g(\frac{2\cdot3^k}{d})M_d$$
has no zeros in $\mathbb{D}^{S_{2\cdot3^k}}$. There exists a boundary point of $V$ being of form $\frac{s}{3^k}$, where $s$ is an integer that cannot divided by $3$. This implies that either $g(1)$ or $g(2)$ is not equal to $0$ since
$$g(\gcd(s,2))=g(\gcd(s,2\cdot3^k))=g(s)=J_V(\frac{s}{3^k})\neq0.$$ Applying Lemma \ref{3k} (1), one sees that it suffices to show $g(1)g(2)\neq0$.

To reach a contradiction, assume $g(1)g(2)=0$ conversely and set $c=g(1)+g(2)$.
Then exactly one of $\frac{1}{3^k}$ and $\frac{2}{3^k}$ is a boundary point of $V$,
forcing
\begin{equation}\label{g(1)+g(2)}
 c=J_V(\frac{1}{3^k})+J_V(\frac{2}{3^k})=1-2\mathbf{1}_V(0^+)=1-g(0)=1-g(2\cdot3^k).
\end{equation}
Similarly, since $g(3^k-1)=g(2)$ and $g(3^k-2)=g(1)$,
  one has $$g(3^k-2)+g(3^k-1)=J_V(\frac{1}{3^k})+J_V(\frac{2}{3^k})=-2\mathbf{1}_V(1^-)-1=-g(3^k)-1.$$
  That is to say, $g(2\cdot3^k)=1-c$ and $g(3^k)=-1-c$. Also from (\ref{g(1)+g(2)}), we deduce that
  $\sum_{i=1}^mcg(i)$ only takes values $0$ or $1$ for $1\leq m\leq3^k-1$.
  It follow  from Lemma \ref{3k} (2) that $$g(3^l)=g(2\cdot3^l)=-c$$ for every $1\leq l\leq k-1$.
  By a direct calculation,
  \begin{equation*}
    \begin{split}
       R(z_1,z_2) & =g(2\cdot3^k)+g(3^k)M_2(z_1)+\sum_{l=1}^{k-1}g(2\cdot3^l)M_{3^{k-l}}(z_2) \\
         & \ +\sum_{l=1}^{k-1}g(3^l)M_2(z_1)M_{3^{k-l}}(z_2)+g(2)M_{3^k}(z_2)+g(1)M_2(z_1)M_{3^k}(z_2) \\
         & =[g(1)z_2^k-(g(1)+c)z_2^{k-1}-1]z_1+(g(2)-g(1))z_2^k   \\
         & \ +(g(1)-g(2))z_2^{k-1}+2.
    \end{split}
  \end{equation*}
  This yields that $$|g(1)z^k-(g(1)+c)z^{k-1}-1|\leq|(g(2)-g(1))(z^k-z^{k-1})+2|$$
  holds for all $z\in\overline{\mathbb{D}}$, which contradicts with Lemma \ref{abineq}.
The proof is complete.
%
%
%
%
%
$\hfill \square $
\vskip2mm
\section {Proof of Theorem \ref{titjthm}}

In this section, we will prove  Theorem \ref{titjthm}, which is restated below.

\vskip2mm

\noindent\textbf{Theorem \ref{titjthm}.}
\textit{Let $V$ be a proper and non-degenerated open subset of $(0,1)$, and $\{\frac{s_i}{t_i}\}_{i=1}^{M}$ be all  boundary points of $V$  in $(0,1)$ with  $\mathrm{gcd}(s_i,t_i)=1$ for all $i$.
  Assume that not all denominators $t_i$ are the same, and
  \begin{itemize}
    \item [(1)] for each $1\leq i\leq M$, $t_i$ is odd,  and $t_i$ is not a product of two primes;
    \item [(2)] for each $1\leq i,j\leq M$, either $t_i=t_j$ or $\gcd(t_i,t_j)=1$.
  \end{itemize}
Then
   $\mathbf{1}_V\in\mathcal{C}$ if and only if there exist two primes $p_1,p_2$ such that one of the following situations is true:
    \begin{itemize}
   \item [(i)] $p_1=p_2+2$ and
     $$V=(\frac{1}{p_1},\frac{1}{p_2})\cup(\frac{3}{p_1},\frac{3}{p_2})\cup\cdots
     \cup(\frac{2m-1}{p_1},\frac{2m-1}{p_2})\cup\cdots\cup(\frac{p_1-2}{p_1},1);$$
     \item [(ii)]   $p_1=p_2+2$ and
  $$V=(0,\frac{2}{p_1})\cup(\frac{2}{p_2},\frac{4}{p_1})\cup\cdots\cup(\frac{2m-2}{p_2},\frac{2m}{p_1})
     \cup\cdots\cup(\frac{p_2-1}{p_2},\frac{p_1-1}{p_1});$$
    \item [(iii)]   $p_1=2p_2+1$ and
     $$V=(0,\frac{2}{p_1})\cup(\frac{1}{p_2},\frac{4}{p_1})\cup\cdots\cup(\frac{m-1}{p_2},\frac{2m}{p_1})
     \cup\cdots\cup(\frac{p_2-1}{p_2},\frac{p_1-1}{p_1});$$
     \item [(iv)] $p_1=2p_2-1$ and
     $$V=(0,\frac{1}{p_2})\cup(\frac{2}{p_1},\frac{2}{p_2})\cup\cdots\cup(\frac{2m-2}{p_1},\frac{m}{p_2})
     \cup\cdots\cup(\frac{p_1-1}{p_1},1).$$
      \end{itemize}}
\vskip2mm


We list some lemmas in Subsection 8.1 as preparation, and then prove Theorem \ref{titjthm} in Subsection 8.2.

\subsection{Preparatory lemmas}

Recall that if $q\ (q\geq2)$  has prime factorization $q=p_{j_1}^{k_1}\cdots p_{j_l}^{k_l}$,  then $S_q=\{j_1,\cdots,j_l\}$ and $$
   M_q(z_{j_1},\cdots,z_{j_l})=z_{j_1}^{k_1-1}(z_{j_1}-1)\cdots z_{j_l}^{k_l-1}(z_{j_l}-1).
 $$
 Moreover, $M_{q_1q_2}=M_{q_1}M_{q_2}$ whenever $\gcd(q_1,q_2)=1$.

\begin{lem} \label{image}
If $q\ (q\geq2)$ is neither a prime nor a product of two primes, then
 $$M_q(\mathbb{D}^{S_q})\supseteq\{re^{i\theta}:0<r\leq\frac{9}{5}, \frac{3\pi}{4}\leq\theta\leq\frac{5\pi}{4}\}\cup\{\frac{2}{5}\}.$$

 In particular, if $q\ (q\geq2)$ is not a product of two primes, then
 $$M_q(\mathbb{D}^{S_q})\supseteq[-\frac{9}{5},0),\quad
 M_q(\overline{\mathbb{D}^{S_q}})\supseteq\{i-1, -i-1\}.$$
\end{lem}
The proof of Lemma \ref{image} is due to a direct calculation. See Lemma \ref{image'} in Appendix for the details of the calculation.

\begin{lem}\label{Mqinequality}
  If $q\ (q\geq2)$ is not a product of two primes, then there exists $\mathbf{z}^i\in\overline{\mathbb{D}^{S_q}}$ ($i=1,2,3,4$), such that the following four inequalities hold:
  \begin{itemize}
    \item [(1)] $|M_q(\mathbf{z}^1)-1-i|>\sqrt{10}$.
    \item [(2)] $|M_q(\mathbf{z}^2)+3+i|<\sqrt{2}$.
    \item [(3)] $|M_q(\mathbf{z}^3)+3+i|>\sqrt{10}$.
    \item [(4)] $|M_q(\mathbf{z}^4)-3-i|<\sqrt{10}$.
  \end{itemize}
\end{lem}
\begin{proof}
For any complex number $c$ with $|c|>1$, one can take $z, w\in\overline{\mathbb{D}}$ satisfying
$|z-c|=|c|+1$ and $|w-c|=|c|-1$. This fact proves the case when $q$ is a prime.

Now assume that $q$ is also not a prime. By Lemma \ref{image}, there exist $\mathbf{u}, \mathbf{v}, \mathbf{w}\in\mathbb{D}^{S_q}$ such that
$M_q(\mathbf{u})=-\frac{9\sqrt{2}}{10}-\frac{9\sqrt{2}}{10}i$,
$M_q(\mathbf{v})=-\frac{27\sqrt{10}}{50}-\frac{9\sqrt{10}}{50}i$
and $M_q(\mathbf{w})=\frac{2}{5}$. The proof is finished by taking $\mathbf{z}^1=\mathbf{u}$, $\mathbf{z}^2=\mathbf{v}$ and $\mathbf{z}^3=\mathbf{z}^4=\mathbf{w}$.
\end{proof}

\begin{lem}\label{Mq1Mq2}
Suppose that $\frac{1}{2}\leq\lambda<5$, $q_1,q_2\geq2$, $\gcd(q_1,q_2)=1$ and neither $q_1$ nor $q_2$ is a product of two primes.
\begin{itemize}
    \item [(1)] If $2-\lambda M_{q_1}+M_{q_2}$ has no zeros in $\mathbb{D}^{S_{q_1q_2}}$, then $q_1$ is a prime.
    \item [(2)] If $M_{q_1}+M_{q_2}$ has no zeros in $\mathbb{D}^{S_{q_1q_2}}$, then both $q_1$ and $q_2$ are  primes.
  \end{itemize}
\end{lem}
\begin{proof}
 (1)  If $q_1$ is not a prime, then by Lemma \ref{image}, one can take $\mathbf{z}^1\in\mathbb{D}^{S_{q_1}}$ and  $\mathbf{z}^2\in\mathbb{D}^{S_{q_2}}$, such that $M_{q_1}(\mathbf{z}^1)=\frac{2}{5}$ and $M_{q_2}(\mathbf{z}^2)=-\frac{10-2\lambda}{5}$.

 \vskip2mm

 (2) If one of $q_1, q_2$, say $q_1$, is not a prime, then by Lemma \ref{image}, one can take $\mathbf{z}^1\in\mathbb{D}^{S_{q_1}}$ and  $\mathbf{z}^2\in\mathbb{D}^{S_{q_2}}$, such that $M_{q_1}(\mathbf{z}^1)=\frac{2}{5}$ and $M_{q_2}(\mathbf{z}^2)=-\frac{2}{5}$.
\end{proof}

In the sequel, we also need a multi-dimensional version of Hurwitz's theorem \cite[pp. 80]{Na}.
\begin{lem} \label{Hurwitz}
 Let $\Omega\subseteq\mathbb{C}^n$ be a domain and
 $\{f_k\}_{k\geq1}$ a sequence of holomorphic functions on $\Omega$ with each $f_k$ zero-free. If $\{f_k\}_{k\geq1}$ converges to a function $f$ uniformly on compact subset of $\Omega$, then either $f\equiv0$ or $f$ has no zeros in $\Omega$.
\end{lem}

\subsection{Proof of Theorem \ref{titjthm}}

This section is dedicated to the proof of Theorem \ref{titjthm}. The sufficiency has been established in Examples 7.2-7.4.
For the necessity, assume $\mathbf{1}_V\in\mathcal{C}$.

To begin, we introduce some notations. We first list all of the different numbers appearing in $\{t_i\}_{i=1}^{M}$: $d_1, d_2, \cdots, d_n\ (n\geq2)$. Set $d_0=1$, $$q=2t_V=2d_1d_2\cdots d_n,$$
$g(m)=J_V(\frac{m}{t_V})\ (m\in\mathbb{Z})$ and $$a_i=J_V(\frac{1}{d_i})=g(\frac{t_V}{d_i}),\ b_i=J_V(\frac{2}{d_i})=g(\frac{q}{d_i}),\quad 0\leq i\leq n.$$
Then by Theorem \ref{criterion}, $g\in E_q$ and the polynomial $$R=\sum_{d\mid q}g(\frac{q}{d})M_d
=\sum_{i=0}^{n}\left(a_iM_{2d_i}+b_iM_{d_i}\right)$$ has no zeros in $\mathbb{D}^{S_q}$. Since each $d_i$ is odd, we have $$M_{2d_i}=M_2M_{d_i}
=(z_1-1)M_{d_i},$$ and thus
\begin{equation}\label{Rz}
  \begin{split}R & =\sum_{i=0}^{n}[a_i(z_1-1)+b_i]M_{d_i} \\
  &=a_0z_1+b_0-a_0+\sum_{i=1}^{n}(a_iz_1+b_i-a_i)M_{d_i}.
  \end{split}
  \end{equation}
Set $\Lambda=\{-1,0,1\}^2\setminus(0,0)$ and $I=\{1,2,\cdots,n\}$. For $(a,b)\in\Lambda$ put $$I_{a,b}=\{i\in I:a_i=a, b_i=b\},
\quad n_{a,b}=\sharp I_{a,b}.$$
We also put \begin{equation*}
 e_i =\begin{cases}
 \frac{2}{d_i}, & i\in I_{0,1}\cup I_{0,-1}; \\
     \frac{1}{d_i}, & \mbox{otherwise},
      \end{cases}
     \quad f_i=\begin{cases}
 \frac{2}{d_i}, & i\in I_{1,0}\cup I_{-1,0}; \\
     \frac{1}{d_i}, & \mbox{otherwise}.
      \end{cases}
 \end{equation*}
Arrange $\{e_i\}_{i=1}^{n}$ and $\{f_i\}_{i=1}^{n}$ from small to large, respectively:
$$e_{i_1}<e_{i_2}<\cdots<e_{i_n};\quad f_{j_1}<f_{j_2}<\cdots<f_{j_n}.$$
It is obvious that $e_{i_1}$ is the smallest boundary point of $V$ in $(0,1)$, and
$1-f_{j_1}$ is the largest boundary point of $V$ in $(0,1)$. If $b_0=0$ then $$g(te_{i_1})=J_V(e_{i_1})=1,$$ and thus $$i_1\in I_{1,0}\cup I_{0,1}\cup I_{1,1}\cup I_{1,-1};$$ if $b_0=2$ then $$g(te_{i_1})=J_V(e_{i_1})=-1,$$ and thus $$i_1\in I_{-1,0}\cup I_{0,-1}\cup I_{-1,1}\cup I_{-1,-1}.$$ Similarly, if $a_0=0$ then
$$j_1\in I_{-1,0}\cup I_{0,-1}\cup I_{1,-1}\cup I_{-1,-1};$$ if $a_0=-2$ then $$j_1\in I_{1,0}\cup I_{0,1}\cup I_{1,1}\cup I_{-1,1}.$$

Following the notations above, we establish several lemmas.

\begin{lem}\label{RJ}
  Let $J$ be a nonempty subset of $I$, and put
  $$R_J=a_0z_1+b_0-a_0+\sum_{i\in J}(a_iz_1+b_i-a_i)M_{d_i},$$
  $$R_J'=b_0+\sum_{i\in J}b_iM_{d_i},$$
  $$R_J''=b_0-2a_0+\sum_{i\in J}(b_i-2a_i)M_{d_i}.$$
  Then \begin{itemize}
         \item [(1)]  $R_J$ and $R_J''$
  have no zeros in $\mathbb{D}^{S_q}$.
  \item [(2)] $R_J'$
  has no zeros in $\mathbb{D}^{S_q}$ provided $b_i\neq0$ for some $i\in J$.
         \item [(3)] the inequality $$|a_0+\sum_{i\in J}a_iM_{d_i}|\leq|b_0-a_0+\sum_{i\in J}(b_i-a_i)M_{d_i}|$$ holds on $\overline{\mathbb{D}^{S_q}}$.
       \end{itemize}
\end{lem}
\begin{proof}
  Put $S=\bigcup_{i\in J\cup\{0\}}S_{d_i}$. Then $S\subseteq S_q$ and $1\in S$.
For each $k\in\mathbb{N}$, we define a polynomial $R_{J,k}$, which depends on variables $\{z_j:j\in S\}$, by substituting $z_j=1-\frac{1}{k}$ in $R$ for every $j\in S_q\setminus S$.
We further define two polynomials $R_{J,k}'$, $R_{J,k}''$ by  substituting $z_1=1-\frac{1}{k}$ and $z_1=\frac{1}{k}-1$ in $R_{J,k}$, respectively. Then for  each $k\in\mathbb{N}$, $R_{J,k}$, $R_{J,k}'$ and $R_{J,k}''$ have no zeros in $\mathbb{D}^{S_q}$. Moreover, $\{R_{J,k}\}_{k\geq1}$, $\{R_{J,k}'\}_{k\geq1}$ and $\{R_{J,k}''\}_{k\geq1}$ converge uniformly to $R_J$, $R_J'$ and $R_J''$ on $\overline{\mathbb{D}^{S_q}}$, respectively. Since for any $i\in I$, $a_iz_1+b_i-a_i\not\equiv0$ and $b_i-2a_i\neq0$,
(1) and (2) immediately follows from Lemma \ref{Hurwitz}.

To see (3), notice that
$$R_J=(a_0+\sum_{i\in J}a_iM_{d_i})z_1+b_0-a_0+\sum_{i\in J}(b_i-a_i)M_{d_i}.$$
Since each $M_{d_i}$ does not depend on the variable $z_1$, (3) follows from (1).
\end{proof}


\begin{lem}\label{a0+b0} The following identity holds,
  $$a_0+b_0+\sum_{i\in I_{1,0}\cup I_{0,1}}\phi(d_i)+2\sum_{i\in I_{1,1}}\phi(d_i)=
  \sum_{i\in I_{-1,0}\cup I_{0,-1}}\phi(d_i)+2\sum_{i\in I_{-1,-1}}\phi(d_i).$$
\end{lem}
\begin{proof} Set $D_+=\{d\geq3: J_V(\frac{2}{d})=1\}$ and
$D_-=\{d\geq3: J_V(\frac{2}{d})=-1\}$. The by Theorem \ref{num of component}, we have
$$a_0+b_0+\sum_{d\in D_+}\phi(d)=\sum_{d\in D_-}\phi(d).$$
If $d$ is odd then $d\in D_+$ if and only if $d=d_i$ for $i\in I_{0,1}\cup I_{1,1} \cup I_{-1,1}$. If $d$ is even then $d\in D_+$ if and only if $d=2d_i$ for $i\in I_{1,0}\cup I_{1,1} \cup I_{1,-1}$. That is to say, $$D_+=\{d_i: i\in I_{0,1}\cup I_{1,1} \cup I_{-1,1}\}\cup
\{2d_i: i\in I_{1,0}\cup I_{1,1} \cup I_{1,-1}\}.$$
Similarly, $$D_-=\{d_i: i\in I_{0,-1}\cup I_{-1,-1} \cup I_{1,-1}\}\cup
\{2d_i: i\in I_{-1,0}\cup I_{-1,-1} \cup I_{-1,1}\}.$$
Since each $d_i$ is odd, $\phi(2d_i)=\phi(2)\phi(d_i)=\phi(d_i)$, and thus
$$\sum_{d\in D_+}\phi(d)=\sum_{i\in I_{1,0}\cup I_{0,1}}\phi(d_i)+2\sum_{i\in I_{1,1}}\phi(d_i)+\sum_{i\in I_{1,-1}\cup I_{-1,1}}\phi(d_i),$$
$$\sum_{d\in D_-}\phi(d)=\sum_{i\in I_{-1,0}\cup I_{0,-1}}\phi(d_i)+2\sum_{i\in I_{-1,-1}}\phi(d_i)+\sum_{i\in I_{1,-1}\cup I_{-1,1}}\phi(d_i).$$
The proof is complete.
\end{proof}

\begin{lem}\label{i1j1}
 \begin{itemize}
    \item [(1)] If $i_1, j_1\in I_{1,1}\cup I_{-1,1}\cup I_{1,-1}\cup I_{-1,-1}$ then $i_1=j_1$.
    \item [(2)] If $i_1\in I_{0,1}\cup I_{0,-1}$ then $i_1=j_1$.
    \item [(3)] If $j_1\in I_{1,0}\cup I_{-1,0}$ then $i_1=j_1$.
  \end{itemize}
\end{lem}
\begin{proof}
  (1) If $i_1, j_1\in I_{1,1}\cup I_{-1,1}\cup I_{1,-1}\cup I_{-1,-1}$ then
  $$e_{i_1}\leq e_{j_1}=f_{j_1}\leq f_{i_1}=e_{i_1},$$ forcing $i_1=j_1$.

  \vskip2mm

  (2) Note that $e_i\leq2f_i$ for every $i\in I$.
  If $i_1\in I_{0,1}\cup I_{0,-1}$ then
  $$e_{j_1}\geq e_{i_1}=2f_{i_1}\geq2f_{j_1}\geq e_{j_1},$$ forcing $i_1=j_1$.

  (3) The proof is similar to that of (2).
\end{proof}

\begin{lem} \label{0111-11}
  \begin{itemize}
    \item [(1)] If $a_0=-2$, $b_0=0$, then $n_{0,1}=n_{1,1}=n_{-1,1}=0$ and $n_{-1,0}\leq1$.
    \item [(2)] If $a_0=0$, $b_0=2$, then $n_{0,1}+n_{1,1}\leq1$, $n_{-1,0}=n_{-1,1}=0$ and $n_{0,1}+n_{-1,-1}\leq1$.
    \item [(3)] If $a_0=-2$, $b_0=2$, then $n_{0,1}+n_{1,1}+n_{-1,1}=1$ and $2n_{-1,0}+n_{0,1}+3n_{-1,1}+n_{-1,-1}\leq3$.
  \end{itemize}
\end{lem}
\begin{proof} Assume $a_0=-2$ or $b_0=2$.
We first show the following two inequality:
\begin{equation}\label{J1ine}
  n_{0,1}+n_{1,1}+n_{-1,1}\leq\frac{b_0}{2},
\end{equation}
\begin{equation}\label{J2ine}
 2n_{-1,0}+n_{0,1}+3n_{-1,1}+n_{-1,-1}\leq\frac{b_0-2a_0}{2}.
\end{equation}
  Put $J_1=I_{0,1}\cup I_{1,1}\cup I_{-1,1}$ and $J_2=I_{-1,0}\cup I_{0,1}\cup I_{-1,1}\cup I_{-1,-1}$.
  To prove (\ref{J1ine}), assume that $J_1$ is nonempty without loss of generality.
  Lemma \ref{RJ} (1), applied to the set $J_1$, implies that
  $$a_0z_1+b_0-a_0+\sum_{i\in I_{0,1}}M_{d_i}
  +z_1\sum_{i\in I_{1,1}}M_{d_i}+(2-z_1)\sum_{i\in I_{-1,1}}M_{d_i}$$
  has no zeros in $\mathbb{D}^{S_q}$. By Lemma \ref{image}, for any $-\frac{9}{5}\leq\lambda<0$ and any $i\in J$, the polynomial
  $M_{d_i}$ can take value $\lambda$ in $\mathbb{D}^{S_{d_i}}$, and thus
  $$a_0z_1+b_0-a_0+[n_{0,1}+n_{1,1}z_1+n_{-1,1}(2-z_1)]\lambda\neq0.$$
  In particular, by taking $z_1=1-\varepsilon\ (0<\varepsilon<1)$, we have
  \begin{equation}\label{b0-varepsilona0}
    b_0-\varepsilon a_0+[n_{0,1}+(1-\varepsilon)n_{1,1}+(1+\varepsilon)n_{-1,1}]\lambda\neq0,
  \quad -\frac{9}{5}\leq\lambda<0.
  \end{equation}
  Since $b_0-\varepsilon a_0>0$, (\ref{b0-varepsilona0}) is equivalent to the following,
  $$\frac{9}{5}[n_{0,1}+(1-\varepsilon)n_{1,1}+(1+\varepsilon)n_{-1,1}]<b_0-\varepsilon a_0.$$
  Letting $\varepsilon\rightarrow0$, we further have
  $$n_{0,1}+n_{1,1}+n_{-1,1}\leq\frac{5b_0}{9}.$$ This proves (\ref{J1ine}) since $n_{0,1}+n_{1,1}+n_{-1,1}$ is an integer and the integral part of $\frac{5b_0}{9}$ is equal to $\frac{b_0}{2}$.

  To prove (\ref{J2ine}), assume that $J_2$ is nonempty without loss of generality.
  Applying Lemma \ref{RJ} (1) to the set $J_2$, we see that
  $$b_0-2a_0+2\sum_{i\in I_{-1,0}}M_{d_i}+\sum_{i\in I_{0,1}}M_{d_i}+3\sum_{i\in I_{-1,1}}M_{d_i}
  +\sum_{i\in I_{-1,-1}}M_{d_i}$$
  has no zeros in $\mathbb{D}^{S_q}$. Using the argument in the last paragraph, we can obtain
  $$2n_{-1,0}+n_{0,1}+3n_{-1,1}+n_{-1,-1}\leq\frac{5(b_0-2a_0)}{9}.$$
  This proves (\ref{J2ine}) since the integral part of $\frac{5(b_0-2a_0)}{9}$ is equal to $\frac{b_0-2a_0}{2}$.

  (1) and (2) has been proved by combining (\ref{J1ine}) with (\ref{J2ine}). To complete the proof,
  assume $a_0=-2$ and $b_0=2$. Then $i_1\not\in I_{1,0}$ and $j_1\in J_1\cup I_{1,0}$. By Lemma \ref{i1j1} (3), $j_1\not\in I_{1,0}$, and thus $J_1$ is nonempty. This proves (3).
  \end{proof}

\begin{lem} \label{10-10}
  \begin{itemize}
    \item [(1)] If $a_0=-2$, $b_0=0$, then $n_{1,0}n_{0,-1}=n_{1,0}n_{-1,-1}=0$.
    \item [(2)] If $a_0=0$, $b_0=2$, then $n_{1,0}n_{0,1}=n_{1,0}n_{1,1}=0$.
    \item [(3)] If $a_0=-2$, $b_0=2$, then $n_{0,1}n_{-1,0}=n_{1,1}n_{-1,0}=0$.
  \end{itemize}
\end{lem}
\begin{proof}
  To show (1), assume $n_{1,0}\geq1$ and $n_{0,-1}+n_{-1,-1}\geq1$, and take $i\in I_{1,0}$ and
  $j\in I_{0,-1}\cup I_{-1,-1}$. Applying Lemma \ref{RJ} (3) to $\{i,j\}$, we see that
  one of the following two inequalities  holds on $\overline{\mathbb{D}^{S_q}}$:
 $$
   |2-M_{d_i}|\leq|2-M_{d_i}-M_{d_j}|;
 $$ $$
   |2-M_{d_i}+M_{d_j}|\leq|2-M_{d_i}|.
 $$
  By Lemma \ref{image}, $M_{d_i}$ can take value $-1-i$ in $\overline{\mathbb{D}^{S_{d_i}}}$,
  and then one of the following two inequalities  holds on $\overline{\mathbb{D}^{S_{d_j}}}$: $$
   |M_{d_j}-3-i|\geq\sqrt{10};
 $$ $$
   |M_{d_j}+3+i|\leq\sqrt{10}.
 $$ This leads to a contradiction due to Lemma \ref{Mqinequality}.

 Using Lemma \ref{Mqinequality}, one can prove (2) and (3) similarly.
\end{proof}

We are ready to prove the necessity in Theorem \ref{titjthm}, which
   is due to combining the following two lemmas.

\begin{lem}\label{-2}
\begin{itemize}
  \item [(1)] We have $a_0=-2$ or $b_0=2$.
  \item [(2)] If $a_0=0$, $b_0=2$ then $n_{0,1}=1$ or $n_{1,1}=1$.
  \end{itemize}
\end{lem}

\begin{lem}\label{-1}
\begin{itemize}
  \item [(1)] If $a_0=-2$, $b_0=0$, then there exist two primes $p_1,p_2$, such that  $p_1=p_2+2$ and
     $$V=(\frac{1}{p_1},\frac{1}{p_2})\cup(\frac{3}{p_1},\frac{3}{p_2})\cup\cdots
     \cup(\frac{2m-1}{p_1},\frac{2m-1}{p_2})\cup\cdots\cup(\frac{p_1-2}{p_1},1).$$
  \item [(2)] If $a_0=0$, $b_0=2$ and $n_{0,1}=1$, then there exist two primes $p_1,p_2$, such that  $p_1=p_2+2$ and
  $$V=(0,\frac{2}{p_1})\cup(\frac{2}{p_2},\frac{4}{p_1})\cup\cdots\cup(\frac{2m-2}{p_2},\frac{2m}{p_1})
     \cup\cdots\cup(\frac{p_2-1}{p_2},\frac{p_1-1}{p_1}).$$
  \item [(3)] If $a_0=0$, $b_0=2$ and $n_{1,1}=1$, then there exist two primes $p_1,p_2$, such that  $p_1=2p_2+1$ and
     $$V=(0,\frac{2}{p_1})\cup(\frac{1}{p_2},\frac{4}{p_1})\cup\cdots\cup(\frac{m-1}{p_2},\frac{2m}{p_1})
     \cup\cdots\cup(\frac{p_2-1}{p_2},\frac{p_1-1}{p_1}).$$
     \item [(4)] If $a_0=-2$, $b_0=2$, then there exist two primes $p_1,p_2$, such that  $p_1=2p_2-1$ and
     $$V=(0,\frac{1}{p_2})\cup(\frac{2}{p_1},\frac{2}{p_2})\cup\cdots\cup(\frac{2m-2}{p_1},\frac{m}{p_2})
     \cup\cdots\cup(\frac{p_1-1}{p_1},1).$$
  \end{itemize}
\end{lem}

\noindent\textbf{Proof of Lemma \ref{-2}.}
 (1) To reach a contradiction, assume conversely $a_0=b_0=0$.

\noindent\textbf{Claim.} For any $i,j\in I\ (i\neq j)$, the polynomial $(b_i-a_i)M_{d_i}+(b_j-a_j)M_{d_j}$ has no zeros in $\mathbb{D}^{S_q}$.

In fact, applying Lemma \ref{RJ} (1) to $\{i,j\}$,
  we see that $(a_iz_1+b_i-a_i)M_{d_i}+(a_jz_1+b_j-a_j)M_{d_j}$ has no zeros in $\mathbb{D}^{S_q}$,
  and hence so does  $(b_i-a_i)M_{d_i}+(b_j-a_j)M_{d_j}$.

 We first show \begin{equation}\label{8.12(1)1}
                 n_{1,1}=n_{-1,-1}=0.
               \end{equation}
  Suppose $n_{1,1}\geq1$ and take $i\in I_{1,1}$. Then Lemma \ref{RJ} (1) implies that the polynomial $R_i=z_1M_{d_i}$ has no zeros in $\mathbb{D}^{S_q}$, which is impossible. This gives
  $n_{1,1}=0$.
  Similarly, one has $n_{-1,-1}=0$.

  As a consequence of (\ref{8.12(1)1}),  $i_1\in I_{1,0}\cup I_{0,1}\cup I_{1,-1}$ and $j_1\in I_{-1,0}\cup I_{0,-1}\cup I_{1,-1}$. It follows from Lemma \ref{i1j1} (2)  that $i_1\not\in I_{0,1}$, and then $$I_{1,0}\cup I_{1,-1}\neq\emptyset.$$

  We then show  \begin{equation}\label{8.12(1)2}
                 n_{0,1}=n_{-1,0}=n_{-1,1}=0.
               \end{equation}
                Suppose $i\in I_{1,0}\cup I_{1,-1}$ and
  $j\in I_{0,1}\cup I_{-1,0}\cup I_{-1,1}$. Then the claim ensures that $(b_i-a_i)M_{d_i}+(b_j-a_j)M_{d_j}$ has no zeros in $\mathbb{D}^{S_q}$.
  This contradicts with Lemma \ref{image} since $b_i-a_i<0$ and $b_j-a_j>0$, and thus we have shown (\ref{8.12(1)2}).

With (\ref{8.12(1)1}) and (\ref{8.12(1)2}), Lemma \ref{a0+b0} immediately gives  $$\sum_{i\in I_{1,0}}\phi(d_i)=\sum_{i\in I_{0,-1}}\phi(d_i),$$ which yields that $n_{1,0}$ and $n_{0,-1}$ are  either simultaneously
zero or simultaneously nonzero.
We have two cases: $n_{1,0}=n_{0,-1}=0$ or $n_{1,0},n_{0,-1}\geq1$.

\noindent\textbf{Case 1.} $n_{1,0}=n_{0,-1}=0$.

In this case, $I=I_{1,-1}$ and $$\frac{1}{d_{i_1}}<\frac{1}{d_{i_2}}<\cdots<\frac{1}{d_{i_n}}.$$ For any $i,j\in I\ (i\neq j)$, it follows from the claim that $-2M_{d_i}-2M_{d_j}$ has no zeros in $\mathbb{D}^{S_q}$, and then by Lemma \ref{Mq1Mq2} (2),
each $d_i\ (i\in I)$ is a prime. Put $m_0=[\frac{d_{i_1}}{d_{i_2}}]$, the integral part of $\frac{d_{i_1}}{d_{i_2}}$. Then $m_0\geq1$ and
$$\frac{m_0}{d_{i_1}}<\frac{1}{d_{i_2}}<\frac{m_0+1}{d_{i_1}}.$$
Since $J_V(\frac{1}{d_{i_2}})=1$ and
$J_V(\frac{m_0}{d_{i_1}})J_V(\frac{m_0+1}{d_{i_1}})=-1$, there exists a  irreducible fraction $\frac{l}{d_{i_k}}\ (2\leq k\leq n)$, such that
$$\frac{m_0}{d_{i_1}}<\frac{l}{d_{i_k}}<\frac{m_0+1}{d_{i_1}}$$ and $J_V(\frac{l}{d_{i_k}})=-1$.
It follows that $l$ is even (since $i_k\in I_{1,-1}$), and thus
$$\frac{l}{d_{i_k}}\geq\frac{2}{d_{i_2}}>\frac{2m_0}{d_{i_1}}\geq\frac{m_0+1}{d_{i_1}}>\frac{l}{d_{i_k}},$$
which is impossible.

\vskip2mm

\noindent\textbf{Case 2.} $n_{1,0},n_{0,-1}\geq1$.

Take $i\in I_{1,0}$, $j\in I_{0,-1}$. By the claim,
$-M_{d_i}-M_{d_j}$ has no zeros in $\mathbb{D}^{S_q}$, and then by Lemma \ref{Mq1Mq2} (2),
both $d_i$ and $d_j$ are primes. Applying Lemma \ref{RJ} (3) to $\{i,j\}$, we see that the inequality
$$|M_{d_{i}}|\leq|M_{d_{i}}+M_{d_{j}}|$$
holds on $\overline{\mathbb{D}^{S_q}}$. That is to say, $|z-1|\leq|z+w-2|$ for  $z,w\in\overline{\mathbb{D}}$. But this fails for $z=i$ and $w=\frac{2-i}{\sqrt{5}}$, and then (1) has been proved.

\vskip2mm

(2) In the light of Lemma \ref{0111-11} (2), it suffices to show $n_{0,1}\geq1$ or $n_{1,1}\geq1$.
To reach a contradiction, assume conversely $n_{0,1}=n_{1,1}=0$.  Lemma \ref{0111-11} (2) also gives  $n_{-1,0}=n_{-1,1}=0$, which forces $i_1\in I_{0,-1}\cup I_{-1,-1}$ and $j_1\in I_{0,-1}\cup I_{1,-1}\cup I_{-1,-1}$. By Lemma \ref{i1j1}, we have $j_1\not\in I_{1,-1}$, i.e., $j_1\in I_{0,-1}\cup I_{-1,-1}$. It follows that $J_V(e_{i_2})=J_V(1-f_{j_2})=1$,  and then $i_2\in I_{1,0}\cup I_{1,-1}$ and $j_2\in I_{1,0}$.
In particular, $i_2\neq j_1$ and $j_2\neq i_1$.
Since  $$f_{i_2}\geq f_{j_2}=\frac{2}{d_{j_2}}=2e_{j_2}\geq2e_{i_2}=\frac{2}{d_{i_2}}\geq f_{i_2},$$ we further have $i_2=j_2$.

We first show $i_1=j_1\in I_{0,-1}$. Suppose $i_1\in I_{-1,-1}$. Since
\begin{equation}\label{-2(2)1}
 f_{i_1}=\frac{1}{d_{i_1}}=e_{i_1}<e_{j_2}=\frac{1}{d_{j_2}}<\frac{2}{d_{j_2}}=f_{j_2},
\end{equation}
one  has $i_1=j_1$, and thus $$J_V(1-f_{j_1})=J_V(1-\frac{1}{d_{j_1}})=J_V(1-\frac{2}{d_{j_1}})=-1.$$
This implies that there exists some $j\in I\setminus\{j_1\}$, such that $\frac{1}{d_{j_1}}<f_j<\frac{2}{d_{j_1}}$ and $J_V(1-f_j)=1$. Therefore,
$$\frac{2}{d_{j_2}}=f_{j_2}\leq f_j<\frac{2}{d_{j_1}},$$ which contradicts with (\ref{-2(2)1}).
So we have shown $i_1\in I_{0,-1}$, and then Lemma \ref{i1j1} (2) gives $i_1=j_1$.
It follows that \begin{equation}\label{-2(2)2}
 \frac{1}{d_{i_2}}=e_{i_2}>e_{i_1}=\frac{2}{d_{i_1}}.
\end{equation}
In particular, $d_{i_1}>3$.

Let $p$ be the smallest prime satisfying $p\nmid d_{i_1}$. Then $$J_V(1-f_{j_1})=J_V(1-f_{i_1})=J_V(1-\frac{1}{d_{i_1}})=J_V(1-\frac{p}{d_{i_1}})=-1.$$
This implies that there exists some $j\in I\setminus\{j_1\}$, such that $\frac{1}{d_{i_1}}<f_j<\frac{p}{d_{i_1}}$ and $J_V(1-f_j)=1$. Therefore,
\begin{equation}\label{-2(2)3}
  \frac{2}{d_{i_2}}=\frac{2}{d_{j_2}}=f_{j_2}\leq f_j<\frac{p}{d_{j_1}}=\frac{p}{d_{i_1}}.
\end{equation}
Combining this with (\ref{-2(2)2}), one has $p\geq5$, i.e., $3\mid d_{i_1}$. Thus, $d_{i_1}$ is not a prime (since $d_{i_1}>3$) and $3\nmid d_{i_2}$.

Now we have  $n_{-1,-1}=0$.
In fact, if  $i\in I_{-1,-1}$ then
applying Lemma \ref{RJ} (1) to $\{i_1,i\}$, we see that
the polynomial $2-M_{d_{i_1}}+M_{d_i}$ has no zeros in $\mathbb{D}^{S_q}$. This leads to a contradiction due to Lemma \ref{Mq1Mq2} (1).

From (\ref{-2(2)2}) and (\ref{-2(2)3}), we  see that $$\frac{2}{d_{i_1}}<\frac{1}{d_{i_2}}<\frac{3}{d_{i_2}}<\frac{2p}{d_{i_1}}$$
Note that $J_V(\frac{m}{d_{i_1}})=0$ for each $2<m<2p$, and
 $$J_V(e_{i_2})=J_V(\frac{1}{d_{i_2}})=J_V(\frac{3}{d_{i_2}})=1$$
 since $3\nmid d_{i_2}$. Then $n\geq3$ and there exists some $i\in I\setminus\{i_1,i_2\}$, such that $\frac{1}{d_{i_2}}<e_i<\frac{3}{d_{i_2}}$ and $J_V(e_i)=-1$.
It follows that $i\in I_{0,-1}$ and
$$\frac{3}{d_{i_2}}>e_i=\frac{2}{d_i}=2f_i>2f_{i_2}=\frac{4}{d_{i_2}}.$$ This is a contradiction.

$\hfill \square $
\vskip2mm

\noindent\textbf{Proof of Lemma \ref{-1}.}
Let $[x]\ (x\in\mathbb{R})$ be the integral part of $x$.
\vskip2mm
 (1) Assume $a_0=-2$, $b_0=0$. By Lemma \ref{0111-11} (1), we have \begin{equation}\label{8.13(1)1}
               n_{0,1}=n_{1,1}=n_{-1,1}=0,
              \end{equation}     forcing
$j_1\in I_{1,0}$. Therefore, Lemma \ref{i1j1} (3) gives $i_1=j_1$, and Lemma \ref{10-10} gives \begin{equation}\label{8.13(1)2}n_{0,-1}=n_{-1,-1}=0.\end{equation}
Since $i_1\in I_{1,0}$, one has  $J_V(e_{i_2})=-1$, which implies $i_2\in I_{-1,0}$. Again by Lemma \ref{0111-11} (1), $I_{-1,0}=\{i_2\}$. The inequality
$$\frac{1}{d_{i_1}}=e_{i_1}<e_{i_2}=\frac{1}{d_{i_2}}$$
yields \begin{equation}\label{ineqfor11}
         d_{i_1}\geq d_{i_2}+2
       \end{equation} since both $d_{i_1}$ and  $d_{i_2}$ are odd.
 Applying Lemma \ref{RJ} (1) to $\{i_1,i_2\}$, we see that
the polynomial $2-M_{d_{i_1}}+M_{d_{i_2}}$ has no zeros in $\mathbb{D}^{S_q}$, and then by Lemma \ref{Mq1Mq2} (1),
$d_{i_1}$ is a prime. With (\ref{8.13(1)1}) and (\ref{8.13(1)2}), Lemma \ref{a0+b0} and (\ref{ineqfor11}) gives
$$\phi(d_{i_2})=-2+\sum_{i\in I_{1,0}}\phi(d_i)\geq-2+\phi(d_{i_1})=d_{i_1}-3\geq d_{i_2}-1\geq \phi(d_{i_2}).$$
This yields that $d_{i_2}$ is also a prime, $d_{i_1}=d_{i_2}+2$ and $I_{1,0}=\{i_1\}$.
Rewrite $p_k=d_{i_k}\ (k=1,2)$. Then for each $1\leq m\leq \frac{p_1-1}{2}$,
$\frac{2m-1}{p_1}$ is a left endpoint of some component interval of $V$, and $\frac{2m-1}{p_2}$ is a right endpoint of some component interval of $V$. Moreover,
for  $k\in\{1,2\}$,
$J_V(\frac{2m}{p_k})=0\ (1\leq m\leq\frac{p_k-1}{2})$.

It remains to show $n=2$. Assume conversely $n\geq3$. Then $i_k\in I_{1,-1}$ for every $3\leq k\leq n$, $J_V(\frac{1}{d_{i_3}})=1$  and \begin{equation}\label{ineqfor12}\frac{1}{d_{i_3}}=f_{i_3}>f_{j_1}
=\frac{2}{p_1}>\frac{1}{p_2}.\end{equation}
This gives that $$\frac{2m-1}{p_2}<\frac{1}{d_{i_3}}<\frac{2m+1}{p_1}$$ for some $1\leq m\leq \frac{p_2-1}{2}$, and that there is at least one boundary point of $V$ between $\frac{1}{d_{i_3}}$ and $\frac{2m+1}{p_1}$. If $\frac{2}{d_{i_3}}>\frac{2m+1}{p_1}$ then  $n\geq4$ and $J_V(e_{i_4})=-1$. This is impossible since one should have $i_4\in I_{1,-1}$. Hence, we have shown $$\frac{4m-2}{p_2}<\frac{2}{d_{i_3}}<\frac{2m+1}{p_1}=\frac{2m+1}{p_2+2},$$ forcing $m=1$.
It follows from (\ref{ineqfor12}) that
$$\frac{4}{p_1}<\frac{2}{d_{i_3}}<\frac{3}{p_1},$$
which is a contradiction.
\vskip2mm
(2) Assume $a_0=0$, $b_0=2$ and $n_{0,1}=1$. By Lemma \ref{0111-11} (2) and Lemma \ref{10-10} (2), we have \begin{equation}\label{8.13(2)1}n_{1,1}=n_{-1,0}=n_{-1,1}=n_{-1,-1}=n_{1,0}=0,\end{equation} forcing $i_1\in I_{0,-1}$.
It follows from Lemma \ref{i1j1} (2) that $i_1=j_1$, and then  $J_V(1-f_{j_2})=1$.
This yields $I_{0,1}=\{j_2\}$ and \begin{equation}\label{ineqfor2}\frac{1}{d_{j_1}}=f_{j_1}<f_{j_2}=\frac{1}{d_{j_2}}.
\end{equation}
Applying Lemma \ref{RJ} (1) to $\{j_1,j_2\}$, we see that
the polynomial $2-M_{d_{j_1}}+M_{d_{j_2}}$ has no zeros in $\mathbb{D}^{S_q}$, and then by Lemma \ref{Mq1Mq2} (1),
$d_{j_1}$ is a prime. With (\ref{8.13(2)1}), Lemma \ref{a0+b0}  and (\ref{ineqfor2}) gives
$$2+\phi(d_{j_2})=\sum_{i\in I_{0,-1}}\phi(d_i)\geq\phi(d_{j_1})=d_{j_1}-1\geq d_{j_2}+1\geq 2+\phi(d_{j_2}).$$
This yields that $d_{j_2}$ is also a prime, $d_{j_1}=d_{j_2}+2$ and $I_{0,-1}=\{j_1\}$.
Rewrite $p_k=d_{j_k}\ (k=1,2)$. Then for each $1\leq m\leq \frac{p_1-1}{2}$,
$\frac{2m}{p_1}$ is a right endpoint of some component interval of $V$, and $\frac{2m-2}{p_2}$ is a left endpoint of some component interval of $V$. Moreover,
for  $k\in\{1,2\}$,
$J_V(\frac{2m-1}{p_k})=0\ (1\leq m\leq\frac{p_k-1}{2})$.
By using an argument similar to that in the second paragraph of (1), we can show $n=2$, which proves (2).
\vskip2mm
(3) Assume $a_0=0$, $b_0=2$ and $n_{1,1}=1$. By Lemma \ref{0111-11} (2) and Lemma \ref{10-10} (2), we have \begin{equation}\label{8.13(3)1}n_{0,1}=n_{-1,0}=n_{-1,1}=n_{1,0}=0,\end{equation} forcing $i_1\in I_{0,-1}\cup I_{-1,-1}$.

\noindent\textbf{Claim.} $i_1\in I_{0,-1}$.

To see this, assume conversely $i_1\in  I_{-1,-1}$. Then by  Lemma \ref{i1j1} (1), one necessarily has
 $j_1\not\in  I_{1,-1}$.  This gives $j_1\in I_{0,-1}\cup I_{-1,-1}$, and thus $J_V(1-f_{j_2})=1$. It follows that $I_{1,1}=\{j_2\}$ and
  $$\frac{1}{d_{i_1}}=e_{i_1}\leq e_{j_2}=\frac{1}{d_{j_2}}.$$
 Since $i_1\neq j_2$, one has \begin{equation}\label{ineqfor31}d_{i_1}\geq d_{j_2}+2.\end{equation} Applying Lemma \ref{RJ} (2) to $\{i_1,j_2\}$, we see that
the polynomial $2-M_{d_{i_1}}+M_{d_{j_2}}$ has no zeros in $\mathbb{D}^{S_q}$, and then by Lemma \ref{Mq1Mq2} (1),
$d_{i_1}$ is a prime. With (\ref{8.13(3)1}), Lemma \ref{a0+b0} and (\ref{ineqfor31}) gives
\begin{equation*}\begin{split}
                   2+2\phi(d_{j_2}) & =\sum_{i\in I_{0,-1}}\phi(d_i)+2\sum_{i\in I_{-1,-1}}\phi(d_i) \\
                     & \geq2\phi(d_{i_1})=2d_{i_1}-2\geq 2d_{j_2}+2\geq 4+2\phi(d_{j_2}),
                \end{split}\end{equation*} which is a contradiction.

By the claim and Lemma \ref{i1j1} (2), we have $i_1=j_1$, and then $J_V(1-f_{j_2})=1$. It follows that $I_{1,1}=\{j_2\}$ and $$\frac{2}{d_{j_1}}=e_{j_1}=e_{i_1}\leq e_{j_2}=\frac{1}{d_{j_2}},$$
which yields \begin{equation}\label{ineqfor32}d_{j_1}\geq2d_{j_2}+1.
\end{equation}
Applying Lemma \ref{RJ} (2) to $\{j_1,j_2\}$, we see that
the polynomial $2-M_{d_{j_1}}+M_{d_{j_2}}$ has no zeros in $\mathbb{D}^{S_q}$, and thus by Lemma \ref{Mq1Mq2} (1),
$d_{j_1}$ is a prime. With (\ref{8.13(3)1}), Lemma \ref{a0+b0} and (\ref{ineqfor32}) gives
\begin{equation*}\begin{split}2+2\phi(d_{j_2}) & =\sum_{i\in I_{0,-1}}\phi(d_i)+2\sum_{i\in I_{-1,-1}}\phi(d_i)\\ & \geq\phi(d_{j_1})=d_{j_1}-1\geq 2d_{j_2}\geq 2+2\phi(d_{j_2}).\end{split}\end{equation*}
This yields that $d_{j_2}$ is also a prime, $d_{j_1}=2d_{j_2}+1$, $I_{0,-1}=\{j_1\}$ and $n_{-1,-1}=0$.
Rewrite $p_k=d_{j_k}\ (k=1,2)$. Then for each $1\leq m\leq \frac{p_1-1}{2}$,
$\frac{2m}{p_1}$ is a right endpoint of some component interval of $V$, and $\frac{m-1}{p_2}$ is a left endpoint of some component interval of $V$. Moreover,
$J_V(\frac{2m-1}{p_1})=0\ (1\leq m\leq\frac{p_1-1}{2})$.

It remains to show $n=2$. Assume conversely $n\geq3$. Then $j_k\in I_{1,-1}$ for every $3\leq k\leq n$, $J_V(1-\frac{1}{d_{j_3}})=J_V(\frac{2}{d_{j_3}})=-1$  and $$\frac{1}{d_{j_3}}=f_{j_3}>f_{j_2}=\frac{1}{p_2}.$$
This gives that $$\frac{2m}{p_1}<1-\frac{1}{d_{j_3}}<\frac{m}{p_2}$$ for some $1\leq m\leq p_2-1$, and that there is at least one boundary point of $V$ between $1-\frac{1}{d_{i_3}}$ and $\frac{2m}{p_1}$.
If $1-\frac{2}{d_{j_3}}<\frac{2m}{p_1}$ then  $n\geq4$ and $J_V(1-f_{j_4})=1$. This is impossible since one should have $j_4\in I_{1,-1}$. Hence, we have shown  $$2-\frac{2m}{p_2}<\frac{2}{d_{j_3}}<1-\frac{2m}{p_1}=1-\frac{2m}{2p_2+1},$$
which leads to a contradiction since $m\leq p_2-1$.
\vskip2mm
(4) Assume $a_0=-2$, $b_0=2$. Then $i_1\in I_{-1,0}\cup I_{0,-1}\cup I_{-1,1}\cup I_{-1,-1}$, and
$j_1\in I_{1,0}\cup I_{0,1}\cup I_{1,1}\cup I_{-1,1}$. By Lemma \ref{i1j1} (2) (3), $i_1\not\in I_{0,-1}$ and $j_1\not\in I_{1,0}$. It follows from Lemma \ref{0111-11} (3) that there are three possible cases:
\begin{itemize}
  \item [(i)] $I_{0,1}=\{j_1\}$ and $n_{1,1}=n_{-1,1}=0$;
  \item [(ii)] $I_{1,1}=\{j_1\}$ and $n_{0,1}=n_{-1,1}=0$;
  \item  [(iii)] $I_{-1,1}=\{j_1\}$ and $n_{0,1}=n_{1,1}=0$.
\end{itemize}
In either case, we further have $n_{-1,0}=0$ by Lemma \ref{0111-11} (3) and Lemma \ref{10-10} (3). In what follows, we will prove that Cases (ii) and (iii) cannot occur.

In Case (ii), one necessarily has $i_1\in I_{-1,-1}$. Thus $i_1=j_1$ by Lemma \ref{i1j1} (1), which is impossible.

Now consider Case (iii). In this case, we also have $n_{-1,-1}=0$ by Lemma \ref{0111-11} (3),
forcing $i_1\in I_{-1,1}$. Hence, Lemma \ref{i1j1} (1) gives $i_1=j_1$. Since $$n_{0,1}=n_{1,1}=n_{-1,0}=n_{-1,-1}=0,$$ it follows from Lemma \ref{a0+b0} that
$$\sum_{i\in I_{1,0}}\phi(d_i)=\sum_{i\in I_{0,-1}}\phi(d_i),$$
and thus $n_{1,0}$ and $n_{0,-1}$ are either simultaneously
zero or simultaneously nonzero.

 First, we assume $n_{1,0}=n_{0,-1}=0$. Then $i_k\in I_{1,-1}$ for every $2\leq k\leq n$. Set $m_0=[\frac{d_{i_1}}{d_{i_2}}]$.
The inequality $$\frac{1}{d_{i_1}}=e_{i_1}<e_{i_2}=\frac{1}{d_{i_2}}$$ gives $m_0\geq1$.
Since $$\frac{m_0}{d_{i_1}}<\frac{1}{d_{i_2}}<\frac{m_0+1}{d_{i_1}}$$ and $J_V(\frac{1}{d_{i_2}})=1$, we see that if both $m_0$ and $m_0+1$ are relatively prime to $d_{i_1}$, then $J_V(\frac{m_0}{d_{i_1}})J_V(\frac{m_0+1}{d_{i_1}})=-1$. This would imply that there exists an  irreducible fraction $\frac{l}{d_{i_k}}\ (2\leq k\leq n)$, such that
$$\frac{m_0}{d_{i_1}}<\frac{l}{d_{i_k}}<\frac{m_0+1}{d_{i_1}}$$ and $J_V(\frac{l}{d_{i_k}})=-1$.
It follows that $l$ is even (since $i_k\in I_{1,-1}$) and
$$\frac{l}{d_{i_k}}\geq\frac{2}{d_{i_k}}=2e_{i_k}\geq2e_{i_2}=\frac{2}{d_{i_2}}\geq
\frac{2m_0}{d_{i_1}}\geq\frac{m_0+1}{d_{i_1}}>\frac{l}{d_{i_k}},$$ which is impossible. Hence,  $d_{i_1}$ cannot be a prime, and $\gcd(m_0(m_0+1), d_{i_1})>1$.
Let $p$ be the smallest prime factor of $d_{i_1}$. Then $p\leq m_0+1$, $J_V(\frac{p-1}{d_{i_1}})=J_V(\frac{p+1}{d_{i_1}})=1$ and $J_V(\frac{p}{d_{i_1}})=0$.
So there exists an  irreducible fraction $\frac{l'}{d_{i_{k'}}}\ (2\leq k'\leq n)$, such that
$$\frac{p-1}{d_{i_1}}<\frac{l'}{d_{i_{k'}}}<\frac{p+1}{d_{i_1}}$$ and $J_V(\frac{l'}{d_{i_{k'}}})=-1$.
Similarly, $l'$ is even and
$\frac{l'}{d_{i_{k'}}}\geq
\frac{2m_0}{d_{i_1}}$, forcing $p+1>2m_0\geq2p-2$. This  also leads to a contradiction.

We then assume $n_{1,0},n_{0,-1}\geq1$.
Fix $i\in I_{1,0}$ and $j\in I_{0,-1}$.
Applying Lemma \ref{RJ} (1) to $\{i,i_1\}$, we see that
the polynomial $2-\frac{2}{3}M_{d_{i}}+M_{d_{i_1}}$ has no zeros in $\mathbb{D}^{S_q}$, and thus by Lemma \ref{Mq1Mq2} (1),
$d_{i}$ is a prime. Similarly,  by applying Lemma \ref{RJ} (2) to $\{j,i_1\}$,
we also see that $d_j$ is a prime. Finally, applying Lemma \ref{RJ} (3) to $\{i,j,i_1\}$, one has that the inequality
$$|2-M_{d_{i}}+M_{d_{i_1}}|\leq|4-M_{d_{i}}-M_{d_{j}}+2M_{d_{i_1}}|$$
holds on $\overline{\mathbb{D}^{S_q}}$.
Since both $d_i$ and $d_j$ are primes, and by Lemma \ref{image}, $M_{d_{i_1}}$ can takes value $-\frac{9}{5}$ in $\mathbb{D}^{S_{d_{i_1}}}$, the above inequality yields
$$|\frac{1}{5}-(z-1)|\leq|\frac{2}{5}-(z-1)-(w-1)|,\quad z,w\in\overline{\mathbb{D}}.$$
It is obvious that for any $z\in\overline{\mathbb{D}}$, one can take $w\in\overline{\mathbb{D}}$ such that
$$|\frac{2}{5}-(z-1)-(w-1)|=|\frac{12}{5}-z-w|=|\frac{12}{5}-z|-1.$$
Therefore, we further  have $$|\frac{6}{5}-z|+1=|\frac{1}{5}-(z-1)|+1\leq|\frac{12}{5}-z|,\quad z\in\overline{\mathbb{D}}.$$ But this fails for $z=\frac{5}{8}+i\frac{\sqrt{39}}{8}$,
and thus Case (iii) cannot occur.

We have shown that only Case (i) can occur. Therefore, $i_1\in I_{-1,-1}$ and
\begin{equation}\label{ineqfor5}
\frac{1}{d_{j_1}}=f_{j_1}<f_{i_1}=\frac{1}{d_{i_1}}=e_{i_1}<e_{j_1}=\frac{2}{d_{j_1}}.
\end{equation}
Applying Lemma \ref{RJ} (2) to $\{i_1,j_1\}$, we see that
the polynomial $2-M_{d_{i_1}}+M_{d_{j_1}}$ has no zeros in $\mathbb{D}^{S_q}$, and thus by Lemma \ref{Mq1Mq2},
$d_{i_1}$ is a prime.

For $(a,b)\in\Lambda$ with $n_{a,b}\geq1$, put
$d_{a,b}=\max\{d_i:i\in I_{a,b}\}$.

\noindent\textbf{Claim.} If $n_{1,0}\geq1$ then $n_{1,-1}\geq1$ and $[\frac{d_{i_1}}{d_{1,-1}}]<2[\frac{d_{i_1}}{d_{1,0}}]$.

For this, suppose $n_{1,0}\geq1$ and  write $m_1=[\frac{d_{i_1}}{d_{1,0}}]$. Then for any  $j\in I_{1,0}$,
\begin{equation}\label{1dj}
  \frac{1}{d_j}=e_j>e_{i_1}=\frac{1}{d_{i_1}},
\end{equation}
and
\begin{equation}\label{mdj}
\frac{m}{d_j}\leq 1-\frac{2}{d_j}\leq1-\frac{2}{d_{1,0}}\leq1-\frac{2m_1}{d_{i_1}},\quad m=1,3,\cdots,d_j-2.
\end{equation} (\ref{1dj}) immediately gives $m_1\geq1$.
 To reach a contradiction, assume conversely that $n_{1,-1}=0$ or $[\frac{d_{i_1}}{d_i}]\geq2m_1$ for every $i\in I_{1,-1}$. Thus, if $i\in I_{1,-1}$ then
\begin{equation}\label{mdi}
\frac{m}{d_i}\leq 1-\frac{1}{d_i}\leq1-\frac{1}{d_{1,-1}}\leq1-
\frac{1}{d_{i_1}}[\frac{d_{i_1}}{d_{1,-1}}]\leq1-\frac{2m_1}{d_{i_1}},\quad m=1,2,\cdots,d_i-1.
\end{equation}

 Now we will show
$J_V(\frac{2m_1}{d_{j_1}})=0$ or $\frac{2m_1}{d_{j_1}}>\frac{m_1+1}{d_{i_1}}$.
In fact, if $J_V(\frac{2m_1}{d_{j_1}})\neq0$ and  $\frac{2m_1}{d_{j_1}}<\frac{m_1+1}{d_{i_1}}$,
then both $\frac{1}{d_{1,0}}$ and $\frac{2m_1}{d_{j_1}}$ belong to the interval $(\frac{m_1}{d_{i_1}},\frac{m_1+1}{d_{i_1}})$ (see (\ref{ineqfor5}) and the definition of $m_1$), and $J_V(\frac{1}{d_{1,0}})=J_V(\frac{2m_1}{d_{j_1}})=1$. This implies that there exists an irreducible fraction $\frac{l}{d_\nu}\ (\nu\in I)$ between $\frac{1}{d_{1,0}}$ and $\frac{2m_1}{d_{j_1}}$, such that $J_V(\frac{l}{d_\nu})=-1$, which gives $\nu\in I_{0,-1}\cup I_{1,-1}\cup (I_{-1,-1}\setminus\{i_1\})$.
Since $$1-f_\nu\geq1-e_\nu\geq1-\frac{l}{d_\nu}>1-\frac{m_1+1}{d_{i_1}},$$
we have $$1-\frac{m_1'+1}{d_{i_1}}<1-f_\nu<1-\frac{m_1'}{d_{i_1}}$$ for some $0\leq m_1'\leq m_1$.
Note that $J_V(1-\frac{m_1'+1}{d_{i_1}})=J_V(1-f_\nu)=-1$ and $J_V(1-\frac{m_1'}{d_{i_1}})<0$. Then there exist two irreducible fractions $\frac{l_k}{d_{\nu_k}}\ (k=1,2)$, such that
$$1-\frac{2m_1}{d_{i_1}}\leq1-\frac{m_1'+1}{d_{i_1}}<\frac{l_1}{d_{\nu_1}}<1-f_\nu<\frac{l_2}{d_{\nu_2}}
<1-\frac{m_1'}{d_{i_1}}$$
and $J_V(\frac{l_k}{d_{\nu_k}})=1\ (k=1,2)$.
This yields that for $k\in\{1,2\}$, $\nu_k=j_1$ or $\nu_k\in I_{1,0}\cup I_{1,-1}$.
By (\ref{mdj}) and (\ref{mdi}), $\nu_k\not\in I_{1,0}\cup I_{1,-1}\ (k=1,2)$. Hence, $\nu_1=\nu_2=j_1$, and thus $\frac{2}{d_{j_1}}<\frac{1}{d_{i_1}}$ (since $l_2-l_1$ is a positive even integer),  contradicting with (\ref{ineqfor5}). So we have shown that
$J_V(\frac{2m_1}{d_{j_1}})=0$ or $\frac{2m_1}{d_{j_1}}>\frac{m_1+1}{d_{i_1}}$. As a consequence, $m_1\geq2$.

For each $1\leq m\leq2m_1-1$, take an irreducible fraction $\frac{l_m'}{d_{\nu_m'}}\in (1-\frac{m+1}{d_{i_1}},1-\frac{m}{d_{i_1}})$ satisfying $J_V(\frac{l_m'}{d_{\nu_m'}})=1$.
Again by (\ref{mdj}) and (\ref{mdi}), for any $1\leq m\leq2m_1-1$ we have $\nu'_m\not\in I_{1,0}\cup I_{1,-1}$, and then $\nu'_m=j_1$ and $l_m'$ is even. On one hand, since $1-\frac{1}{d_{j_1}}>1-\frac{1}{d_{i_1}}$, $$l_m'\leq l_1'-2(m-1)\leq d_{j_1}-2m-1$$
for each $1\leq m\leq2m_1-1$. On the other hand, since for each $1\leq m\leq2m_1-1$, $1-\frac{1}{d_{j_1}}-\frac{l_m'}{d_{j_1}}<\frac{m+1}{d_{i_1}}$, it follows that
$$d_{j_1}-1-l_m'<\frac{d_{j_1}(m+1)}{d_{i_1}}<2(m+1),$$
 i.e., $d_{j_1}-1-l_m'\leq2m$.
Therefore
\begin{equation}\label{dj1-2m-1}l_m'=d_{j_1}-2m-1,\quad m=1,2,\cdots,2m_1-1.\end{equation}
If $J_V(\frac{2m_1}{d_{j_1}})=0$ then $d_{j_1}$ is not a prime, and $c:=\gcd(m_1,d_{j_1})>1$.
It follows that $J_V(1-\frac{c}{d_{j_1}})=0$, which contradicts with (\ref{dj1-2m-1})
since $c\leq m_1\leq4m_1-1$.
If $\frac{2m_1}{d_{j_1}}>\frac{m_1+1}{d_{i_1}}$ then by (\ref{dj1-2m-1}),
$$\frac{2m_1}{m_1+1}>\frac{d_{j_1}}{d_{i_1}}=\frac{d_{j_1}}{2m_1}\cdot\frac{2m_1}{d_{i_1}}
>\frac{d_{j_1}}{2m_1}\cdot\frac{d_{j_1}-l_{2m_1-1}'}{d_{j_1}}=\frac{4m_1-1}{2m_1},$$
which is also a contradiction. Thus, the claim has been proved.

We then use the claim to show \begin{equation}\label{101-1}
                                n_{1,0}=n_{1,-1}=0.
                              \end{equation} Suppose $n_{1,-1}\geq1$ and write $m_2=[\frac{d_{i_1}}{d_{1,-1}}]$.  For any $i\in I_{1,-1}$,
$$\frac{1}{d_i}=e_i>e_{i_1}=\frac{1}{d_{i_1}},$$
forcing $m_2\geq1$. Since $$1-\frac{m_2+1}{d_{i_1}}<1-\frac{1}{d_{1,-1}}<1-\frac{m_2}{d_{i_1}}$$
and $$J_V(1-\frac{m_2+1}{d_{i_1}})=J_V(1-\frac{1}{d_{1,-1}})=J_V(1-\frac{m_2}{d_{i_1}})=-1,$$
there exist two irreducible fractions $\frac{l_k}{d_{\nu_k}}\ (k=1,2)$, such that
$$1-\frac{m_2+1}{d_{i_1}}<\frac{l_1}{d_{\nu_1}}<1-\frac{1}{d_{1,-1}}<\frac{l_2}{d_{\nu_2}}<1-\frac{m_2}{d_{i_1}}$$
and $J_V(\frac{l_k}{d_{\nu_k}})=1\ (k=1,2)$.
This yields that for $k\in\{1,2\}$, $\nu_k=j_1$ or $\nu_k\in I_{1,0}\cup I_{1,-1}$.
By (\ref{ineqfor5}), $\nu_1$ and $\nu_2$ cannot be equal to $j_1$ simultaneously.
Moreover, since for any  $i\in I_{1,-1}$, $J_V(1-\frac{1}{d_i})=-1$ and
$$1-\frac{m}{d_i}\leq1-\frac{2}{d_i}\leq1-\frac{2}{d_{1,-1}}<1-\frac{2m_2}{d_{i_1}}\leq1-\frac{m_2+1}{d_{i_1}}<\frac{l_1}{d_{\nu_1}}<\frac{l_2}{d_{\nu_2}}$$
for every $m\geq2$, we have $\nu_k\not\in I_{1,-1}\ (k=1,2)$. Therefore, $\nu_s\in I_{1,0}$ for some $s\in\{1,2\}$, and then $l_s$ is odd. It follows that
$$\frac{m_2+1}{2}\geq\frac{m_2+1}{d_{\nu_s}-l_s}
=\frac{d_{i_1}}{d_{\nu_s}}\cdot\frac{\frac{m_2+1}{d_{i_1}}}{1-\frac{l_s}{d_{\nu_s}}}
>\frac{d_{i_1}}{d_{\nu_s}}
\geq\frac{d_{i_1}}{d_{1,0}}\geq[\frac{d_{i_1}}{d_{1,0}}],$$
i.e., $m_2\geq2[\frac{d_{i_1}}{d_{1,0}}]$,  contradicting with the claim. Hence, we have shown (\ref{101-1}).

Recall  $I_{0,1}=\{j_1\}$ and $n_{1,1}=n_{-1,1}=n_{-1,0}=0$.
With (\ref{101-1}), Lemma \ref{a0+b0} and (\ref{ineqfor5}) gives
$$\phi(d_{j_1})=\sum_{i\in I_{0,-1}}\phi(d_i)+2\sum_{i\in I_{-1,-1}}\phi(d_i)\geq2\phi(d_{i_1})=2d_{i_1}-2\geq d_{j_1}-1\geq \phi(d_{j_1}).$$
This yields that $d_{j_2}$ is also a prime, $d_{j_1}=2d_{i_1}-1$, $I_{-1,-1}=\{i_1\}$ and $n_{0,-1}=0$.
Rewrite $p_1=d_{j_1}$, $p_2=d_{i_1}$. Then for each $1\leq m\leq p_2$,
$\frac{m}{p_2}$ is a right endpoint of some component interval of $V$, and $\frac{2m-2}{p_1}$ is a left endpoint of some component interval of $V$. Moreover,
$J_V(\frac{2m-1}{p_1})=0\ (1\leq m\leq\frac{p_1-1}{2})$. The proof is complete.
$\hfill \square $
\vskip3mm
\noindent \textbf{Acknowledgement}
This work was supported by the National Natural Science Foundation of China (Grant No.12231005), National Key R$\&$D Program of China (2024YFA1013400), and  Natural
Science Foundation of Sichuan (2024NSFSC1338).

\appendix
\section*{Appendix}
In this  Appendix, we will  present the proofs of Lemmas \ref{t=36}, \ref{abineq} and \ref{image}.
\begin{lemA}\label{t=36'} Let $V,g$ and $R$ be given as in (\ref{Vdisplay}), (\ref{g(m)}) and (\ref{R}).
  \begin{itemize}
     \item [(1)] If $t_V=3$
     and $g\in E_6$, then $R(z_1,z_2)$ has no zeros in $\mathbb{D}^{2}$ if and only if $V=V_{3,0}, V_{3,1}, (0,\frac{2}{3})$ or $(\frac{1}{3},1)$.
     \item [(2)] If $t_V=6$
     and $g\in E_{12}$, then $R(z_1,z_2)$ has no zeros in $\mathbb{D}^{2}$ if and only if $V=V_{6,0}, V_{6,1}, (0,\frac{1}{6})\cup(\frac{1}{3},\frac{5}{6}),
      (0,\frac{1}{3})\cup(\frac{1}{2},\frac{2}{3}),
      (0,\frac{1}{3})\cup(\frac{1}{2},1),
      (0,\frac{1}{2})\cup(\frac{2}{3},1)$ or
      $(\frac{1}{6},\frac{2}{3})\cup(\frac{5}{6},1)$.
   \end{itemize}
\end{lemA}

\begin{lemA}\label{abineq'} Suppose $k\geq2$, $a,b\in\{-1,0,1\}$ and $ab=0$. Then
  $$|a\xi^k-(2a+b)\xi^{k-1}-1|
  >|(b-a)(\xi^k-\xi^{k-1})+2|$$
for some $\xi\in\overline{\mathbb{D}}$.
\end{lemA}

\begin{lemA} \label{image'}
If $q\ (q\geq2)$ is neither a prime nor a product of two primes, then
 $$M_q(\mathbb{D}^{S_q})\supseteq\{re^{i\theta}:0<r\leq\frac{9}{5}, \frac{3\pi}{4}\leq\theta\leq\frac{5\pi}{4}\}\cup\{\frac{2}{5}\}.$$

 In particular, if $q\ (q\geq2)$ is not a product of two primes, then
 $$M_q(\mathbb{D}^{S_q})\supseteq[-\frac{9}{5},0),\quad
 M_q(\overline{\mathbb{D}^{S_q}})\supseteq\{i-1, -i-1\}.$$
\end{lemA}

To prove Lemma \ref{t=36'}, we need the following result.
\begin{lemA} \label{1to6lem}
   For any $z\in\overline{\mathbb{D}}$, we have
   \begin{itemize}
     \item [(1)] $|z^2-2z+3|\geq|z-1|^k,\ k=1,2$.
     \item [(2)] $|z^2+z-3|\geq|z^2-z-1|$.
   \end{itemize}
 \end{lemA}
 \begin{proof}
   Put $x=\mathrm{Re}\,z$ and $y=\mathrm{Im}\,z$.

   (1) Since $$\mathrm{Re}\,(z-1)^2=(x-1)^2-y^2\geq(x-1)^2+x^2-1\geq-1,$$
   one has $$|z^2-2z+3|=|(z-1)^2+2|\geq|(z-1)^2|=|z-1|^2.$$
   When $|z-1|\geq1$,
   $$|z^2-2z+3|\geq|z-1|^2\geq|z-1|;$$
   when $|z-1|\leq1$,
   $$|z^2-2z+3|\geq2-|z-1|^2\geq1\geq|z-1|.$$
   \vskip2mm

   (2) A direct calculation gives \begin{equation*}
   \begin{split}
      |z^2-z-3|^2-|z^2-z-1|^2 & =(x^2+x-3-y^2)^2+(2xy+y)^2- \\
        & \ \  [(x^2-x-1-y^2)^2+(2xy-y)^2] \\
        & =4(2-x^2)(1-x)+4(x+1)y^2 \\
        & \geq0.
   \end{split}
       \end{equation*}
 \end{proof}

\noindent\textbf{Proof of Lemma \ref{t=36'}.}
Without loss of generality, assume $V\neq V_{t_V,0}$, $V_{t_V,1}$ for $t_V\in\{3,6\}$.

(1) By a direct calculation,
 \begin{equation*}
  \begin{split}R(z_1,z_2) & =g(6)+g(3)M_2(z_1)+g(2)M_3(z_2)
 +g(1)M_6(z_1,z_2)  \\
 & =g(0)+g(3)(z_1-1)+g(2)(z_2-1)+g(1)(z_1-1)(z_2-1).
 \end{split}
 \end{equation*}
 There are only $4$ possibilities:
 \begin{itemize}
   \item [(a)] $g(0)=2$, $g(1)=-1$, $g(2)=g(3)=0$, $$R(z_1,z_2)=1+z_1+z_2-z_1z_2.$$ Then $R$ has  the zero $(-\frac{1}{2},-\frac{1}{3})$ in $\mathbb{D}^2$.
   \item [(b)] $g(0)=2$, $g(2)=-1$, $g(1)=g(3)=0$, $$R(z_1,z_2)=3-z_2.$$ Then  $R$ has no zeros in $\mathbb{D}^2$.
   \item [(c)] $g(1)=1$, $g(3)=-2$, $g(0)=g(2)=0$, $$R(z_1,z_2)=(z_1-1)(z_2-3).$$ Then  $R$ has no zeros in $\mathbb{D}^2$.
   \item [(d)] $g(2)=1$, $g(3)=-2$, $g(0)=g(1)=0$, $$R(z_1,z_2)=1-2z_1+z_2.$$ Then  $R$ has the zero $(\frac{1}{2},0)$ in $\mathbb{D}^2$.
 \end{itemize}
\vskip2mm
 (2) By a direct calculation,
 \begin{equation*}
  \begin{split}R(z_1,z_2) & =g(12)+g(6)M_2(z_1)+g(4)M_3(z_2)+g(3)M_4(z_1)+ \\
 &\  g(2)M_6(z_1,z_2)
 +g(1)M_{12}(z_1,z_2)  \\
 & =g(0)+g(6)(z_1-1)+g(4)(z_2-1)+g(3)z_1(z_1-1)+  \\
 &\  g(2)(z_1-1)(z_2-1)+g(1)z_1(z_1-1)(z_2-1).
 \end{split}
 \end{equation*}

There are  $12$ possibilities:
\begin{itemize}
  \item [(a)] $g(2)=1$, $g(3)=-1$, $g(0)=g(1)=g(5)=g(4)=g(6)=0$, $$R(z_1,z_2)=(z_1-1)(z_2-z_1-1).$$ Then $R$ has the zero $(-\frac{1}{2},\frac{1}{2})$ in $\mathbb{D}^2$.
  \item [(b)] $g(3)=1$, $g(4)=-1$, $g(0)=g(1)=g(5)=g(2)=g(6)=0$, $$R(z_1,z_2)=z_1^2-z_1-z_2+1.$$ Then $R$ has the zero $(\frac{1}{2},\frac{3}{4})$ in $\mathbb{D}^2$.
  \item [(c)] $g(0)=2$, $g(1)=g(5)=-1$, $g(2)=1$, $g(3)=g(4)=g(6)=0$, $$R(z_1,z_2)=-(z_1-1)^2z_2+z_1^2-2z_1+3.$$ Since for any $z\in\mathbb{D}$, $z^2-2z+3\neq0$ and $|z^2-2z+3|\geq|z-1|^2$ by Lemma \ref{1to6lem} (1), $R$ has no zeros  in $\mathbb{D}^2$.
  \item [(d)] $g(0)=2$, $g(1)=g(5)=-1$, $g(3)=1$, $g(2)=g(4)=g(6)=0$, $$R(z_1,z_2)=2-z_1(z_1-1)(z_2-2).$$ Then $R$ has the zero $(\lambda,-\frac{1}{2})$ in $\mathbb{D}^2$, where $\lambda$ is a root of $z^2-z+\frac{4}{5}$.
  \item [(e)] $g(0)=2$, $g(1)=g(5)=-1$, $g(4)=1$, $g(2)=g(3)=g(6)=0$, $$R(z_1,z_2)=2-(z_1^2-z_1-1)(z_2-1).$$ Then $R$ has the zero $(\frac{1}{2},-\frac{3}{5})$ in $\mathbb{D}^2$.
  \item [(f)] $g(1)=g(5)=1$, $g(2)=-1$, $g(6)=-2$, $g(0)=g(3)=g(4)=0$, $$R(z_1,z_2)=(z_1-1)(z_1z_2-z_1-z_2-1).$$ Then $R$ has the zero $(-\frac{1}{2},-\frac{1}{3})$ in $\mathbb{D}^2$.
  \item [(g)] $g(1)=g(5)=1$, $g(3)=-1$, $g(6)=-2$, $g(0)=g(2)=g(4)=0$, $$R(z_1,z_2)=(z_1-1)(z_1z_2-2z_1-2).$$ Then $R$ has the zero $(-\frac{3}{4},-\frac{2}{3})$ in $\mathbb{D}^2$.
  \item [(h)] $g(1)=g(5)=1$, $g(4)=-1$, $g(6)=-2$, $g(0)=g(2)=g(3)=0$, $$R(z_1,z_2)=(z_1^2-z_1-1)z_2-(z_1^2+z_1-3).$$ Since for any $z\in\mathbb{D}$, $z^2-z+3\neq0$ and $|z^2+z-3|\geq|z^2-z-1|$  by Lemma \ref{1to6lem} (2),  $R$ has no zeros in $\mathbb{D}^2$.
  \item [(i)] $g(0)=2$, $g(2)=-1$, $g(3)=1$, $g(4)=-1$, $g(1)=g(5)=g(6)=0$, $$R(z_1,z_2)=2+z_1^2-z_1z_2.$$ Then $R$ has no zeros in $\mathbb{D}^2$.
  \item [(j)] $g(2)=1$, $g(3)=-1$, $g(4)=1$, $g(6)=-2$, $g(0)=g(1)=g(5)=0$, $$R(z_1,z_2)=2-z_1^2+z_1z_2-2z_1.$$ Then $R$ has the zero $(\sqrt{3}-1,0)$ in $\mathbb{D}^2$.
  \item [(k)] $g(0)=2$, $g(2)=-1$, $g(3)=1$, $g(6)=-2$, $g(1)=g(5)=g(4)=0$, $$R(z_1,z_2)=-(z_1-1)z_2+z_1^2-2z_1+3.$$ Since for any $z\in\mathbb{D}$, $z^2-2z+3\neq0$ and $|z^2-2z+3|\geq|z-1|$  by Lemma \ref{1to6lem} (1), $R$ has no zeros  in $\mathbb{D}^2$.
  \item [(l)] $g(0)=2$,  $g(3)=-1$, $g(4)=1$, $g(6)=-2$, $g(1)=g(5)=g(2)=0$, $$R(z_1,z_2)=3-z_1^2-z_1+z_2.$$ Then $R$ has no zeros in $\mathbb{D}^2$.
\end{itemize}
$\hfill \square $
\vskip2mm

To prove Lemma \ref{abineq'}, we need to establish the following.
    \begin{lemA}\label{zk(z-1)}
   Suppose $k\geq3$, $0<\theta_0\leq\pi$, $0<r<2\cos\frac{\theta_0}{2k-1}$ and $k\pi-\theta_0\leq\theta\leq k\pi+\theta_0$. Then there exists $z\in\mathbb{D}$, such that $z^k-z^{k-1}=re^{i\theta}$.
\end{lemA}
\begin{proof}
   Let
  $f(z)=z^k-z^{k-1}$. By the continuity of $f$ on $(-1,0)$,  $f(z)=(-1)^{k}\cdot r=re^{ik\pi}$ for some $z\in(-1,0)$. We also note that the image of $f$ on $\mathbb{D}$ is symmetric with respect to the real line. So it suffices to consider $\theta\in[k\pi-\theta_0,k\pi)$, and then
  $$\frac{\pi}{2}\leq\frac{\theta-\pi}{k-1}<\frac{2\theta-\pi}{2k-1}<\pi.$$ Set
$a=\frac{\theta-\pi}{k-1}$, $b=\frac{2\theta-\pi}{2k-1}$ and
   $$\lambda(\alpha)=\frac{\sin(\theta-(k-1)\alpha)}{\sin(\theta-k\alpha)]},\quad \alpha\in[a,b].$$
   Since for $a\leq\alpha\leq b$, one has $\frac{\pi}{2}\leq \theta-k\alpha\leq \pi$ and
   $0<\theta-(k+1)\alpha\leq\frac{\pi}{2}$, and thus $\lambda$ is strictly decreasing on $[a,b]$.
   It follows that $0<\lambda(\alpha)<1$ for $a<\lambda<b$ since $\lambda(a)=1$ and $\lambda(b)=0$.
  Put $z(\alpha)=\lambda(\alpha)e^{i\alpha}$ and $w(\alpha)=f(z(\alpha))$. Then for any $a<\alpha<b$, $z(\alpha)\in\mathbb{D}$ and
  \begin{equation*}
  \begin{split}\arg(z(\alpha)-1) &=-\arctan\frac{\lambda(\alpha)\sin\alpha}{1-\lambda(\alpha)\cos\alpha}  \\
  &=-\arctan\frac{\sin(\theta-(k-1)\alpha)\sin\alpha}{\sin(\theta-k\alpha)-\sin(\theta-(k-1)\alpha)\cos\alpha}    \\
  &=-\arctan\frac{\sin(\theta-(k-1)\alpha)\sin\alpha}{-\cos(\theta-(k-1)\alpha)\sin\alpha}   \\
  &=\theta-(k-1)\alpha,
  \end{split}
  \end{equation*} which gives $\arg w(\alpha)=\theta$. Since $|w(b)|=0$ and
   $$|w(a)|=2\cos\frac{k\pi-\theta}{2k-1}\geq2\cos\frac{\theta_0}{2k-1}>r,$$
   there exists $a<\alpha_0<b$ such that $|w(\alpha_0)|=r$. Therefore, one can take $z=z(\alpha_0)$ to meet the requirement.
\end{proof}

\noindent\textbf{Proof of Lemma \ref{abineq'}.}
Set
$$f(z)=|az^k-(2a+b)z^{k-1}-1|
  -|(b-a)(z^k-z^{k-1})+2|,\quad z\in\mathbb{C}.$$
There are four possibilities:
  \begin{itemize}
    \item [(a)] $a=1$, $b=0$, and then
    $$f(z)=|z^k-2z^{k-1}-1|-|z^k-z^{k-1}-2|;$$
    \item [(b)] $a=-1$, $b=0$, and then
    $$f(z)=|z^k-2z^{k-1}+1|-|z^k-z^{k-1}+2|;$$
    \item [(c)] $a=0$, $b=1$, and then
    $$f(z)=|z^{k-1}+1|-|z^k-z^{k-1}+2|;$$
    \item [(d)] $a=0$, $b=-1$, and then
    $$f(z)=|z^{k-1}-1|-|z^k-z^{k-1}-2|.$$
  \end{itemize}
  When $k=2$,  in both cases (a) and (d), $f(\xi)>0$ for $\xi=-1$;  in the case (b), $f(\xi)>0$ for $\xi=i$; in the case (c), $f(\xi)>0$ for $\xi=\frac{1}{2}+\frac{\sqrt{3}}{2}i$. Now assume $k\geq3$. Since $$2\cos\frac{\pi}{2k-1}\geq2\cos\frac{\pi}{5}>\frac{3}{2},$$ letting $\theta_0=\pi$ in Lemma \ref{zk(z-1)},
  we see that there exist $\lambda, \mu\in\mathbb{D}$, such that $$\lambda^k-\lambda^{k-1}=\frac{3}{2},\quad\mu^k-\mu^{k-1}=-\frac{3}{2}.$$
  Therefore,
  $$\frac{1}{2}\geq|\lambda^k-2\lambda^{k-1}-1|=|2-\lambda^k|\geq2-|\lambda|^k>1,$$
  $$\frac{1}{2}\geq|\lambda^{k-1}-1|=|\lambda^k-\frac{5}{2}|\geq\frac{5}{2}-|\lambda|^k>\frac{3}{2}.$$
  That is to say, in both cases (a) and (d), $f(\xi)>0$ for $\xi=\lambda$.
  Similarly, in both cases (b) and (c), $f(\xi)>0$ for $\xi=\mu$.
$\hfill \square $
\vskip2mm

To prove Lemma \ref{image'}, we need an auxiliary result.
\begin{lemA}\label{(z-1)k}
\begin{itemize}
    \item [(1)] Suppose $k\geq2$ and $w\in\mathbb{D}$, there exists $z\in\mathbb{D}$, such that $z^k-z^{k-1}=w$.
    \item [(2)]
   Suppose $k\in\{3,4\}$, $0<r\leq\frac{9}{5}$ and $\frac{3\pi}{4}\leq\theta\leq\frac{5\pi}{4}$. Then there exist $z,w\in\mathbb{D}$, such that $(z-1)^k=re^{i\theta}$ and
        $(w-1)^k=\frac{2}{5}$.
        \end{itemize}
\end{lemA}
\begin{proof}
(1) Since the roots $\alpha_1, \alpha_2, \cdots, \alpha_{k+1}$ of $f(z)=z^k(z-1)-w$ satisfy $$|\alpha_1\alpha_2\cdots\alpha_{k+1}|=|w|<1,$$
   there exists $i\in\{1,2,\cdots,k+1\}$ such that $\alpha_i\in\mathbb{D}$.
\vskip2mm

  (2)
  First note that for $\lambda>0$ and $\frac{\pi}{2}<\alpha<\frac{3\pi}{2}$,
   $|\lambda e^{i\alpha}+1|<1$ if and only if $\lambda<-2\cos\alpha$. In particular, $|(\frac{2}{5})^{\frac{1}{3}}e^{i\frac{2\pi}{3}}+1|<1$. Then we can take $w=(\frac{2}{5})^{\frac{1}{3}}e^{i\frac{2\pi}{3}}+1$ for $k=3$ or $w=1-(\frac{2}{5})^{\frac{1}{4}}$ for $k=4$. Now put $\lambda=r^{\frac{1}{k}}$,
    $\alpha=\frac{\theta+2\pi}{k}$ and $z=\lambda e^{i\alpha}+1$.
   Then $(z-1)^k=re^{i\theta}$. Moreover, when $k=3$, $\frac{11\pi}{12}<\alpha<\frac{13\pi}{12}$ and
   $$\lambda\leq(\frac{9}{5})^{\frac{1}{3}}<2\cos\frac{\pi}{12}<-2\cos\alpha;$$
   when $k=4$, $\frac{11\pi}{16}<\alpha<\frac{13\pi}{16}$ and
   $$\lambda\leq(\frac{9}{5})^{\frac{1}{4}}<2\cos\frac{5\pi}{16}<-2\cos\alpha.$$
   In either case, we have $z\in\mathbb{D}$.
\end{proof}

\vskip2mm
\noindent\textbf{Proof of Lemma \ref{image'}.}
We have the following  three cases.

\noindent\textbf{Case 1.} $q=p^k$ for some prime $p$ and some integer $k\ (k\geq3)$.

By (\ref{repofMq}), $M_q$ only depends on one variable and $M_q(z)=z^k-z^{k-1}$.
Taking $\theta_0=\pi$ in Lemma \ref{zk(z-1)}, we have $$\label{Mq(D)}
  M_q(\mathbb{D})\supseteq\{z\in\mathbb{C}:|z|<2\cos\frac{\pi}{2k-1}\}.
$$
Note that $2\cos\frac{\pi}{5}>\frac{2}{5}$ and for $k\geq4$, $$2\cos\frac{\pi}{2k-1}\geq2\cos\frac{\pi}{7}>\frac{9}{5}.$$ It remains to show that when $k=3$,
$$M_q(\mathbb{D})\supseteq\{re^{i\theta}:0<r\leq\frac{9}{5}, \frac{3\pi}{4}\leq\theta\leq\frac{5\pi}{4}\}.$$
For $k=3$, taking $\theta_0=\frac{\pi}{4}$ in Lemma \ref{zk(z-1)}, we have $$M_q(\mathbb{D})\supseteq\{re^{i\theta}:0<r<2\cos\frac{\pi}{20}, \frac{3\pi}{4}\leq\theta\leq\frac{5\pi}{4}\}.$$
This completes the proof for Case 1.
\vskip2mm

\noindent\textbf{Case 2.} $q=p_{j_1}p_{j_2}\cdots p_{j_l}$, where $l\geq3$ and
$j_1<j_2<\cdots<j_l$.

By (\ref{repofMq}), $M_q$  depends on $l$ variables $z_{j_1}, z_{j_2}, \cdots, z_{j_l}$ and $$M_q(z)=(z_{j_1}-1)(z_{j_2}-1)\cdots(z_{j_l}-1).$$
When $l$ is odd,  by taking $z_{j_1}=z_{j_2}=z_{j_3}$ and $z_{j_k}=0\ (4\leq k\leq l)$,  one obtains that $$M_q(\mathbb{D}^{S_q})\supseteq\{w:w=(z-1)^3\ \text{for some}\ z\in\mathbb{D}\}.$$
Similarly, when $l$ is even,
$$M_q(\mathbb{D}^{S_q})\supseteq\{w:w=(z-1)^4\ \text{for some}\ z\in\mathbb{D}\}.$$
Hence, in Case 2, the conclusion immediately follows from Lemma \ref{(z-1)k} (2).
 \vskip2mm

\noindent\textbf{Case 3.} $q$ is not a power of some prime, and $p^2\mid q$ for some prime $p$.

Take $k\in\mathbb{N}$ such that $p^k\mid q$ and $p^{k+1}\nmid q$, and put $q'=\frac{q}{p^k}$.
Then $M_q=M_{p^k}M_{q'}$ and $$M_q(\mathbb{D}^{S_q})=M_{p^k}(\mathbb{D}^{S_{p^k}})\cdot M_{q'}(\mathbb{D}^{S_{q'}}):=\{\lambda\mu:\lambda\in M_{p^k}(\mathbb{D}^{S_{p^k}}),\mu\in M_{q'}(\mathbb{D}^{S_{q'}})\}.$$
Since $k\geq2$, by Lemma \ref{(z-1)k} (1),
$M_{p^k}(\mathbb{D}^{S_{p^k}})\supseteq\mathbb{D}$.
Moreover, letting $I$ denote the interval $(-1,1)$, we have $M_{q'}(\mathbb{D}^{S_{q'}})\supseteq M_{q'}(I^{S_{q'}})$. It follows that $M_{q'}(\mathbb{D}^{S_{q'}})$ contains $(0,2)$ or $(-2,0)$, and thus $$M_q(\mathbb{D}^{S_q})\supseteq\{z\in\mathbb{C}:|z|<2\}.$$
The proof is complete.
$\hfill \square $

\vskip2mm

\vskip3mm \noindent{Hui Dan, School of Mathematics, Sichuan
University, Chengdu, Sichuan, 610065, China,
   E-mail:  hdan@scu.edu.cn

\noindent Kunyu Guo, School of Mathematical Sciences, Fudan
University, Shanghai, 200433, China, E-mail: kyguo@fudan.edu.cn

\end{document}